\documentclass[a4paper]{amsart}

\usepackage{booktabs}
\usepackage{xcolor}
\usepackage{hyperref}
\hypersetup{
    colorlinks,
    linkcolor={red!50!black},
    citecolor={blue!50!black},
    urlcolor={blue!80!black}
}
\usepackage{longtable}
\usepackage{pdflscape}
\usepackage[all]{xy}
\usepackage{cite}
\usepackage[margin=2.5cm]{geometry}
\theoremstyle{plain}
\newtheorem{theorem}{\bf Theorem}[section]
\newtheorem{proposition}[theorem]{\bf Proposition}

\theoremstyle{definition}
\newtheorem{definition}[theorem]{\bf Definition}
\newtheorem{remark}[theorem]{\bf Remark}
\DeclareMathOperator{\Amp}{\overline{Amp}}
\DeclareMathOperator{\Ann}{Ann}
\newcommand{\be}{{\boldsymbol{e}}}

\newcommand{\CC}{\mathbb{C}}

\newcommand{\Cstar}{\CC^\times}

\newcommand{\cE}{\mathcal{E}}

\newcommand{\cF}{\mathcal{F}}

\DeclareMathOperator{\Fl}{Fl}
\newcommand{\hG}{\widehat{G}}
\DeclareMathOperator{\GL}{GL}
\DeclareMathOperator{\Gr}{Gr}

\newcommand{\cL}{\mathcal{L}}

\DeclareMathOperator{\NE}{NE}

\newcommand{\MC}{\overline{\NE}}
\newcommand{\cO}{\mathcal{O}}

\newcommand{\PP}{\mathbb{P}}

\newcommand{\QQ}{\mathbb{Q}}

\DeclareMathOperator{\rf}{rf}
\DeclareMathOperator{\rk}{rk}

\newcommand{\Str}{\mathrm{Str}}

\newcommand{\bbV}{\mathbb{V}}
\newcommand{\cV}{\mathcal{V}}

\newcommand{\ZZ}{\mathbb{Z}}

\newcommand{\GIT}{/\!\!/}

\newcommand{\Obro}[2]{\text{\rm B\O S}^{#1}_{#2}} 
\newcommand{\MM}[2]{\mathrm{MM}^3_{#1\text{--}#2}} 
\newcommand{\BB}[2]{B^{#1}_{#2}}
\newcommand{\VV}[2]{V^{#1}_{#2}}
\newcommand{\MW}[2]{\mathrm{MW}^{#1}_{#2}}
\renewcommand{\SS}[2]{S^2_{#2}}
\newcommand{\FI}[2]{\mathrm{FI}^{#1}_{#2}}

\newlength{\matrixheight}
\newlength{\padabove}

\usepackage{colortbl}

\newcommand{\evnrow}{\rowcolor[gray]{0.95}}
\newcommand{\oddrow}{}

\title{Quantum Periods for Certain Four-Dimensional Fano Manifolds}

\author[Coates]{Tom Coates}
\author[Galkin]{Sergey Galkin}
\author[Kasprzyk]{Alexander Kasprzyk}
\author[Strangeway]{Andrew Strangeway}

\address{Department of Mathematics\\
Imperial College London\\
180 Queen's Gate\\
London SW7 2AZ
\\UK}

\email{t.coates@imperial.ac.uk}
\email{a.m.kasprzyk@imperial.ac.uk}
\email{a.strangeway09@imperial.ac.uk}

\address{National Research University Higher School of Economics (HSE)\\
Faculty of Mathematics and Laboratory of Algebraic Geometry\\
7 Vavilova str., 117312, Moscow, Russia}

\email{Sergey.Galkin@phystech.edu}



\begin{document}

\begin{abstract}
  We collect a list of known four-dimensional Fano manifolds and compute their quantum periods.  This list includes all four-dimensional Fano manifolds of index greater than one, all four-dimensional toric Fano manifolds, all four-dimensional products of lower-dimensional Fano manifolds, and certain complete intersections in projective bundles.
\end{abstract}
\maketitle

\section{Introduction}

In this paper we take the first step towards implementing a program, laid out in \cite{ProcECM}, to find and classify four-dimensional Fano manifolds using mirror symmetry.  We compute quantum periods and quantum differential equations for many known four-dimensional Fano manifolds, using techniques described in~\cite{QC105}.  Our basic reference for the theory of Fano manifolds is the book by Iskovskikh--Prokhorov~\cite{Iskovskikh--Prokhorov}.  Recall that the \emph{index} of a Fano manifold $X$ is the largest integer~$r$ such that ${-K_X} = rH$ for some ample divisor~$H$.  A four-dimensional Fano manifold has index at most~$5$~\cite{Shokurov}.  Four-dimensional Fano manifolds with index $r>1$ have been classified.  In what follows we compute the quantum periods and quantum differential equations for all four-dimensional Fano manifolds of index $r>1$, for all four-dimensional Fano toric manifolds, and for certain other four-dimensional Fano manifolds of index~$1$.

\subsection*{Highlights}  We draw the reader's attention to:
\begin{itemize}
\item \S\ref{sec:double_cover}, where new tools for computing Gromov--Witten invariants (twisted $I$-functions for toric complete intersections~\cite{CCIT:toric_stacks_2} and an improved Quantum Lefschetz theorem~\cite{Coates}) make a big practical difference to the computation of quantum periods.  This should be contrasted with~\cite[\S19]{QC105}, where the new techniques were not available.
\item \S\ref{sec:null_correlation}, which relies on a new construction of  Szurek--Wi\'sniewski's null-correlation bundle~\cite{Szurek--Wisniewski} that may be of independent interest.
\item The tables of regularized quantum period sequences in Appendix~\ref{appendix:periods}.
\item The numerical calculation of quantum differential operators in \S\ref{sec:qdes} and Appendix~\ref{appendix:qdes}.  This suggests in particular that, for each four-dimensional Fano manifold $X$ with Fano index $r>1$, the regularized quantum differential equation of $X$ is either extremal or of low ramification.
\item \S\ref{sec:MW^4_17} and \S\ref{operator:MW^4_17}, which together give an example of a product such that the regularized quantum differential equation for each factor is extremal, but the regularized quantum differential equation for the product itself is not.
\end{itemize}
This paper is accompanied by fully-commented source code, written in the computational algebra system Magma~\cite{Magma}. This will allow the reader to verify the calculations presented here, or to perform similar computations.

\section{Methodology}
\label{sec:methodology}

The quantum period $G_X$  of a Fano manifold $X$ is a generating function
\begin{align}
  \label{eq:quantum_period}
  G_X(t) = 1 + \sum_{d=1}^\infty c_d t^d && \text{$t \in \CC$}
\end{align}
for certain genus-zero Gromov--Witten invariants~$c_d$ of $X$.  A precise definition can be found in \cite[\S B]{QC105}, but roughly speaking $c_d$ is the `virtual number' of degree-$d$ rational curves $C$ in $X$ that pass through a given point and satisfy certain constraints on their complex structure.  (The degree of a curve $C$ here is the quantity $\langle {-K_X}, C\rangle$.)  The quantum period is discussed in detail in \cite{ProcECM,QC105}; one property that will be important in what follows is that the regularized quantum period
\begin{align}
  \label{eq:regularized_quantum_period}
  \hG_X(t) = 1 + \sum_{d=1}^\infty d! c_d t^d && \text{$t \in \CC$, $|t| \ll \infty$}
\end{align}
satisfies a differential equation called the \emph{regularized quantum differential equation} of $X$:
\begin{align}
  \label{eq:regularized_QDE}
  L_X \hG_X \equiv 0 && L_X = \sum_{m=0}^{m=N} p_m(t) D^m
\end{align}
where the $p_m$ are polynomials and $D = t \frac{d}{dt}$.  It is expected that the regularized quantum differential equation for a Fano manifold $X$ is \emph{extremal} or \emph{of low ramification}, as described in \S\ref{sec:qdes} below.  This is a strong constraint on the Gromov--Witten invariants $c_d$ of $X$.

Quantum periods for a broad class of toric complete intersections can be computed using Givental's mirror theorem~\cite{Givental}:

\begin{theorem}[\!\!\protect{\cite[Corollary~C.2]{QC105}}]
  \label{thm:toric_mirror}
  Let $X$ be a toric Fano manifold and let $D_1,\ldots,D_N \in  H^2(X;\QQ)$ be the cohomology classes Poincar\'e-dual to the   torus-invariant divisors on $X$.  The quantum period of $X$ is:
  \[
  G_X(t) = \sum_{\substack{
      \beta \in H_2(X;\ZZ) : \\
      \text{$\langle \beta, D_i \rangle \geq 0$ $\forall i$}
    }}
  \frac{
    t^{\langle \beta,{-K_X}\rangle}
  }{
    \prod_{i=1}^N \langle \beta, D_i \rangle !
  }
  \]
\end{theorem}

\begin{theorem}[\!\!\protect{\cite[Corollary~D.5]{QC105}}]
  \label{thm:toric_ci_mirror}
  Let $Y$ be a toric Fano manifold, and let $D_1,\ldots,D_N \in H^2(Y;\QQ)$ be the cohomology classes Poincar\'e-dual to the torus-invariant divisors on $Y$.  Let $X$ be the complete intersection in $Y$ defined by a regular section of $E = L_1 \oplus \cdots \oplus L_s$ where each $L_i$ is a nef line bundle, and let $\rho_i = c_1(L_i)$, $1 \leq i \leq s$.  Suppose that the class $c_1(Y)-\Lambda$ is ample on $Y$, where $\Lambda =  c_1(L_1) + \cdots + c_1(L_s)$.  Then $X$ is Fano, and the quantum period of $X$ is:
  \[
  G_X(t) = e^{{-c} t} \sum_{\substack{
      \beta \in H_2(Y;\ZZ) : \\
      \text{$\langle \beta, D_i \rangle \geq 0$ $\forall i$}
      }}
    t^{\langle \beta,{-K_Y}- \Lambda\rangle}
    \frac{
      \prod_{j=1}^s \langle \beta, \rho_j \rangle !
    }{
      \prod_{i=1}^N \langle \beta, D_i \rangle !
    }
    \]
    where $c$ is the unique rational number such that the right-hand side has the form $1 + O(t^2)$.  
\end{theorem} 

An analogous mirror theorem holds for certain complete intersections in toric Deligne--Mumford stacks, but we will need only the case where the ambient stack is a weighted projective space:

\begin{theorem}[\!\!\protect{\cite[Proposition~D.9]{QC105}}]
  \label{thm:wps_ci_mirror}
  Let $Y$ be the weighted projective space $\PP(w_0,\ldots,w_n)$, let $X$ be a smooth Fano manifold given as a complete intersection in $Y$ defined by a section of $E = \cO(d_1) \oplus \cdots \oplus \cO(d_m)$, and let ${-k} = w_0 + \cdots + w_n - d_1 - \cdots - d_m$.  Suppose that each $d_i$ is a positive integer, that ${-k}>0$, and that $w_i$ divides $d_j$ for all $i$, $j$ such that $0 \leq i \leq n$ and $1 \leq j \leq m$.  Then the quantum period of $X$ is:
  \[
  G_X(t) = e^{{-c} t} \sum_{d=0}^\infty
  t^{{-k} d}
  \frac{
    \prod_{j=1}^m (d d_j) !
  }{
    \prod_{i=1}^n (d w_i) !
  }
  \]
  where $c$ is the unique rational number such that the right-hand side has the form $1 + O(t^2)$.  
\end{theorem}

\noindent The quantum period of a product is the product of the quantum periods:

\begin{theorem}[\!\!\protect{\cite[Corollary~E.4]{QC105}}]
  \label{thm:products}
  Let $X$ and $Y$ be smooth projective complex manifolds.  Then:
  \[
  G_{X \times Y}(t) = G_X(t)\, G_Y(t)
  \]
\end{theorem}

\noindent As we will see below, another powerful tool for computing quantum periods is the Abelian/non-Abelian Correspondence of Ciocan-Fontanine--Kim--Sabbah~\cite{Ciocan-Fontanine--Kim--Sabbah}.  We now proceed to the calculation of quantum periods.

\section{Four-Dimensional Fano Manifolds of Index~$5$}
\label{sec:index_5}

The only example here is $\PP^4$ \cite{Kobayashi--Ochiai,Kollar,Serpico}.  This is a toric variety.  Theorem~\ref{thm:toric_mirror} yields:
\begin{align*}
  G_{\PP^4}(t)  = \sum_{d=0}^\infty \frac{t^{5d}}{(d!)^5} 
  &&
  \text{[\hyperref[table:index_5]{regularized quantum period p.~\pageref*{table:index_5}}, \hyperref[operator:P4]{operator p.~\pageref*{operator:P4}}]}
\end{align*}

\section{Four-Dimensional Fano Manifolds of Index~$4$}
\label{sec:index_4}

The only example here is the quadric $Q^4 \subset \PP^5$~\cite{Kollar,Serpico}.  This is a complete intersection in a toric variety.  Theorem~\ref{thm:toric_ci_mirror} yields:
\begin{align*}
  G_{Q^4}(t)  = \sum_{d=0}^\infty \frac{(2d)!}{(d!)^6} t^{4d}
  &&
  \text{[\hyperref[table:index_4]{regularized quantum period p.~\pageref*{table:index_4}}, \hyperref[operator:Q4]{operator p.~\pageref*{operator:Q4}}]}
\end{align*}

\section{Four-Dimensional Fano Manifolds of Index~$3$}
\label{sec:index_3}

There are six examples \cite{Fujita:polarized_1,Fujita:polarized_2,Fujita:polarized_3,Fujita:book,Iskovskikh:Fano_1,Iskovskikh:anticanonical,Iskovskikh--Prokhorov}, which are known as del~Pezzo fourfolds: 
\begin{itemize}
\item a sextic hypersurface $\FI{4}{1}$ in the weighted projective space $\PP^5(1^4,2,3)$;
\item a quartic hypersurface $\FI{4}{2}$ in the weighted projective space $\PP^5(1^5,2)$;
\item a cubic hypersurface $\FI{4}{3} \subset \PP^5$;
\item a complete intersection $\FI{4}{4} \subset \PP^6$ of type $(2H) \cap (2H)$, where $H = \cO_{\PP^6}(1)$; 
\item a complete intersection $\FI{4}{5} \subset \Gr(2,5)$ of type $H \cap H$, where $H$ is the hyperplane bundle; and
\item $\FI{4}{6} = \PP^2 \times \PP^2$.
\end{itemize}
The first four examples here are complete intersections in weighted projective spaces.  Theorem~\ref{thm:wps_ci_mirror} yields:
\begin{align*}
  & G_{\FI{4}{1}}(t)  = \sum_{d=0}^\infty \frac{(6d)!}{(3d)!(2d)!(d!)^4} t^{3d} \\
  & G_{\FI{4}{2}}(t)  = \sum_{d=0}^\infty \frac{(4d)!}{(2d)!(d!)^5} t^{3d} \\
  & G_{\FI{4}{3}}(t)  = \sum_{d=0}^\infty \frac{(3d)!}{(d!)^6} t^{3d} \\
  & G_{\FI{4}{4}}(t)  = \sum_{d=0}^\infty \frac{(2d)!(2d)!}{(d!)^7} t^{3d} 
\end{align*}
For $\FI{4}{5} \subset \Gr(2,5)$ we use the Abelian/non-Abelian Correspondence, applying Theorem~F.1 in \cite{QC105} with $a=2$, $b=c=d=e=0$.  This yields:
\[
G_{\FI{4}{5}}(t) = \sum_{l=0}^\infty \sum_{m=0}^\infty
(-1)^{l+m}
t^{3l+3m}
\frac
{
  (l+m)! (l+m)!
}
{
  (l!)^5 (m!)^5 
}
\big(1-5 (m-l)H_m \big)
\]
where $H_m$ is the $m$th harmonic number.
For $\PP^2 \times \PP^2$, combining Theorem~\ref{thm:products} with \cite[Example~G.2]{QC105} yields:
\[
G_{\PP^2 \times \PP^2}(t) = 
\sum_{l=0}^\infty \sum_{m=0}^\infty \frac{t^{3l+3m}}{(l!)^3 (m!)^3}
\]

\begin{table*}[h!]
  \centering
  \begin{tabular}{cccp{2ex}ccc} \toprule
    \multicolumn{1}{c}{$X$} &
    \multicolumn{1}{c}{$\hG_X$} &
    \multicolumn{1}{c}{$L_X$} & &
    \multicolumn{1}{c}{$X$} &
    \multicolumn{1}{c}{$\hG_X$} &
    \multicolumn{1}{c}{$L_X$} \\
    \cmidrule{1-3} \cmidrule{5-7} 
    $\FI{4}{1}$ & \hyperref[table:index_3]{p.~\pageref*{table:index_3}} & \hyperref[operator:FI^4_1]{p.~\pageref*{operator:FI^4_1}} &&
    $\FI{4}{4}$ & \hyperref[table:index_3]{p.~\pageref*{table:index_3}} & \hyperref[operator:FI^4_4]{p.~\pageref*{operator:FI^4_4}} \\
    $\FI{4}{2}$ & \hyperref[table:index_3]{p.~\pageref*{table:index_3}} & \hyperref[operator:FI^4_2]{p.~\pageref*{operator:FI^4_2}} &&
    $\FI{4}{5}$ & \hyperref[table:index_3]{p.~\pageref*{table:index_3}} & \hyperref[operator:FI^4_5]{p.~\pageref*{operator:FI^4_5}} \\
    $\FI{4}{3}$ & \hyperref[table:index_3]{p.~\pageref*{table:index_3}} & \hyperref[operator:FI^4_3]{p.~\pageref*{operator:FI^4_3}} &&
    $\PP^2 \times \PP^2$ & \hyperref[table:index_3]{p.~\pageref*{table:index_3}} & \hyperref[operator:FI^4_6]{p.~\pageref*{operator:FI^4_6}} \\
    \bottomrule \\
  \end{tabular}
\end{table*}

\section{Four-Dimensional Fano Manifolds of Index~$2$}

Consider now a four-dimensional Fano manifold with index~$r=2$ and Picard rank~$\rho$.  

\subsection{The Case $\rho=1$}
\label{sec:index_2_rank_1}

Four-dimensional Fano manifolds with index~$r=2$ and Picard rank~$\rho=1$ have been classified~\cite{Mukai:PNAS,Wilson},~\cite[Chapter~5]{Iskovskikh--Prokhorov}.  Up to deformation, there are 9~examples: the `linear unsections' of smooth three-dimensional Fano manifolds with $\rho=1$, $r=1$, and degree at most $144$.   We compute the quantum periods of these examples using the constructions in \cite[\S\S8--16]{QC105}, writing $V^4_k$ for a four-dimensional Fano manifold with $\rho=1$, $r=2$, and degree~$16k$.

\subsubsection{$V^4_2$} \hfill [\hyperref[table:index_2]{regularized quantum period p.~\pageref*{table:index_2}}, \hyperref[operator:V^4_2]{operator p.~\pageref*{operator:V^4_2}}] \label{sec:V^4_2}  

This is a sextic hypersurface in $\PP^5(1^5,3)$.  Proposition~D.9 in \cite{QC105} yields:
\[
G_{V^4_2}(t) = \sum_{d=0}^\infty \frac{(6d)!}{(d!)^5(3d)!} t^{2d}
\]

\subsubsection{$V^4_4$} \hfill [\hyperref[table:index_2]{regularized quantum period p.~\pageref*{table:index_2}}, \hyperref[operator:V^4_4]{operator p.~\pageref*{operator:V^4_4}}] \label{sec:V^4_4}  

This is a quartic hypersurface in $\PP^5$.  Theorem~\ref{thm:toric_ci_mirror} yields:
\[
G_{V^4_4}(t) = \sum_{d=0}^\infty \frac{(4d)!}{(d!)^6} t^{2d}
\]

\subsubsection{$V^4_6$} \hfill [\hyperref[table:index_2]{regularized quantum period p.~\pageref*{table:index_2}}, \hyperref[operator:V^4_6]{operator p.~\pageref*{operator:V^4_6}}] \label{sec:V^4_6}  

This is a complete intersection of type $(2H) \cap (3H)$ in $\PP^6$, where $H = \cO_{\PP^6}(1)$.  Theorem~\ref{thm:toric_ci_mirror} yields:
\[
G_{V^4_6}(t) = \sum_{d=0}^\infty \frac{(2d)!(3d)!}{(d!)^7} t^{2d}
\]

\subsubsection{$V^4_8$} \hfill [\hyperref[table:index_2]{regularized quantum period p.~\pageref*{table:index_2}}, \hyperref[operator:V^4_8]{operator p.~\pageref*{operator:V^4_8}}] \label{sec:V^4_8}  

This is a complete intersection of type $(2H) \cap (2H) \cap (2H)$ in $\PP^7$, where $H = \cO_{\PP^7}(1)$.  Theorem~\ref{thm:toric_ci_mirror} yields:
\[
G_{V^4_8}(t) = \sum_{d=0}^\infty \frac{\big((2d)!\big)^3}{(d!)^8} t^{2d}
\]

\subsubsection{$V^4_{10}$} \hfill [\hyperref[table:index_2]{regularized quantum period p.~\pageref*{table:index_2}}, \hyperref[operator:V^4_10]{operator p.~\pageref*{operator:V^4_10}}] \label{sec:V^4_10} 

This is a complete intersection in $\Gr(2,5)$, cut out by a regular section of $\cO(1) \oplus \cO(2)$ where $\cO(1)$ is the pullback of $\cO(1)$ on projective space under the Pl\"ucker embedding.  We apply Theorem~F.1 in \cite{QC105} with $a=b=1$ and $c=d=e=0$.  This yields:
\[
G_{V^4_{10}}(t) = \sum_{l=0}^\infty \sum_{m=0}^\infty
(-1)^{l+m}
t^{2l+2m}
\frac
{
  (l+m)! (2l+2m)!
}
{
  (l!)^5 (m!)^5 
}
\big(1-5 (m-l)H_m \big)
\]
where $H_m$ is the $m$th harmonic number.

\subsubsection{$V^4_{12}$} \hfill [\hyperref[table:index_2]{regularized quantum period p.~\pageref*{table:index_2}}, \hyperref[operator:V^4_12]{operator p.~\pageref*{operator:V^4_12}}] \label{sec:V^4_12} 

This is the subvariety of $\Gr(2,5)$ cut out by a regular section of $S^\star \otimes \det S^\star$, where $S$ is the universal bundle of subspaces on $\Gr(2,5)$.  We apply Theorem~F.1 in \cite{QC105} with $c = 1$ and $a  = b = d = e = 0$. This yields:
\[
G_{V^4_{12}}(t) = \sum_{l=0}^\infty \sum_{m=0}^\infty
({-1})^{l+m} t^{2l+2m}
\frac
{
  (2l+m)! (l+2m)!
}
{
  (l!)^5 (m!)^5 
}
\big(1+(m-l)(H_{2l+m} + 2 H_{l + 2m}-5H_m) \big)
\]

\subsubsection{$V^4_{14}$} \hfill [\hyperref[table:index_2]{regularized quantum period p.~\pageref*{table:index_2}}, \hyperref[operator:V^4_14]{operator p.~\pageref*{operator:V^4_14}}] \label{sec:V^4_14}  

This is a complete intersection in $\Gr(2,6)$, cut out by a regular section of $\cO(1)^{\oplus 4}$ where $\cO(1)$ is the pullback of $\cO(1)$ on projective space under the Pl\"ucker embedding.  We apply Theorem~F.1 in \cite{QC105} with $a=4$ and $b=c=d=e=0$.  This yields:
\[
G_{V^4_{14}}(t) = \sum_{l=0}^\infty \sum_{m=0}^\infty
(-1)^{l+m}
t^{2l+2m}
\frac
{
  \big((l+m)!\big)^4
}
{
  (l!)^6 (m!)^6
}
\big(1-6 (m-l)H_m \big)
\]

\subsubsection{$V^4_{16}$} \hfill [\hyperref[table:index_2]{regularized quantum period p.~\pageref*{table:index_2}}, \hyperref[operator:V^4_16]{operator p.~\pageref*{operator:V^4_16}}] \label{sec:V^4_16} 

This is the subvariety of $\Gr(3,6)$ cut out by a regular section of $\wedge^2 S^\star \oplus (\det S^\star)^{\oplus 2}$, where $S$ is the universal bundle of subspaces on $\Gr(3,6)$.  We apply
Theorem~F.1 in \cite{QC105} with $a=2$, $b=c=d=0$, and $e=1$.   This shows that the quantum period $G_{V^4_{16}}(t)$ is the coefficient of $(p_2-p_1)(p_3-p_1)(p_3-p_2)$ in the expression:
\begin{multline*}
  \sum_{l_1 = 0}^\infty
  \sum_{l_2 = 0}^\infty
  \sum_{l_3 = 0}^\infty
  t^{2l_1 + 2 l_2 + 2 l_3}
  {
    \prod_{k=1}^{l_1+l_2+l_3} (p_1+p_2+p_3 + k)^2
    \over
    \prod_{j=1}^{j=3}
    \prod_{k=1}^{k=l_j} (p_j + k)^6
  }
  \prod_{1 \leq i < j \leq 3} \prod_{k=1}^{l_i + l_j} (p_i + p_j + k) \\
  \times  \prod_{1 \leq i < j \leq 3} \big(p_j-p_i + (l_j-l_i) \big) 
\end{multline*}
(Since this expression is totally antisymmetric in $p_1$,~$p_2$,~$p_3$, it is divisible by $(p_2-p_1)(p_3-p_1)(p_3-p_2)$.)

\subsubsection{$V^4_{18}$} \hfill [\hyperref[table:index_2]{regularized quantum period p.~\pageref*{table:index_2}}, \hyperref[operator:V^4_18]{operator p.~\pageref*{operator:V^4_18}}] \label{sec:V^4_18} 

This is the subvariety of $\Gr(5,7)$ cut out by a regular section of $\big( S \otimes \det S^\star\big) \oplus \det S^{\star}$, where $S$ is the universal bundle of subspaces on $\Gr(5,7)$.  We apply Theorem~F.1 in \cite{QC105} with $a=d=1$ and $b=c=e=0$.  This shows that the quantum period $G_{V^4_{18}}(t)$ is the coefficient of $\prod_{1 \leq i < j \leq 5} (p_j-p_i)$ in the expression:
\begin{multline*}
  \sum_{l_1 = 0}^\infty
  \sum_{l_2 = 0}^\infty
  \sum_{l_3 = 0}^\infty
  \sum_{l_4 = 0}^\infty
  \sum_{l_5 = 0}^\infty
  t^{2|l|}
  {
    \prod_{k=1}^{k=|l|} (p_1+p_2+\cdots + p_5 + k)
    \over
    \prod_{j=1}^{j=5}
    \prod_{k=1}^{k=l_j} (p_j + k)^7
  }
  \prod_{j=1}^{j=5}  \prod_{k=1}^{|l| - l_j}  (p_1 + p_2 + \cdots + p_5 - p_j + k) \\
  \times 
  \prod_{1 \leq i < j \leq 5} \big(p_j-p_i + (l_j-l_i) \big) 
\end{multline*}
where $|l| = l_1 + l_2 + \cdots + l_5$.  (As above, antisymmetry implies that the long formula here is divisible by $\prod_{1 \leq i < j \leq 5} (p_j-p_i)$.)


\subsection{The Case $\rho>1$}
\label{sec:index_2_higher_rank}

Four-dimensional Fano manifolds with $\rho>1$ and $r=2$ have been classified by Mukai~\cite{Mukai:PNAS,Mukai:10} and Wi\'sniewski~\cite{Wisniewski:classification}.  There are 18 deformation families, as follows.  We denote the $k$th such deformation family, as given in \cite[Table~12.7]{Iskovskikh--Prokhorov}, by $\MW{4}{k}$.

\subsubsection{$\MW{4}{1}$} \hfill [\hyperref[table:index_2]{regularized quantum period p.~\pageref*{table:index_2}}, \hyperref[operator:MW^4_1]{operator p.~\pageref*{operator:MW^4_1}}] \label{sec:MW^4_1}  

This is the product $\PP^1 \times B^3_1$.  Combining Theorem~\ref{thm:products} with \cite[Example~G.1]{QC105} and \cite[\S3]{QC105} yields:
\[
G_{\MW{4}{1}}(t) = 
\sum_{l=0}^\infty \sum_{m=0}^\infty \frac{(6m)!}{(l!)^2 (m!)^3 (2m)! (3m)!} t^{2l+2m}
\]

\subsubsection{$\MW{4}{2}$} \hfill [\hyperref[table:index_2]{regularized quantum period p.~\pageref*{table:index_2}}, \hyperref[operator:MW^4_2]{operator p.~\pageref*{operator:MW^4_2}}] \label{sec:MW^4_2}  

This is the product $\PP^1 \times B^3_2$.  Combining Theorem~\ref{thm:products} with \cite[Example~G.1]{QC105} and \cite[\S4]{QC105} yields:
\[
G_{\MW{4}{2}}(t) = 
\sum_{l=0}^\infty \sum_{m=0}^\infty \frac{(4m)!}{(l!)^2 (m!)^4 (2m)!} t^{2l+2m}
\]

\subsubsection{$\MW{4}{3}$} \hfill [\hyperref[table:index_2]{regularized quantum period p.~\pageref*{table:index_2}}, \hyperref[operator:MW^4_3]{operator p.~\pageref*{operator:MW^4_3}}] \label{sec:MW^4_3}  

This is the product $\PP^1 \times B^3_3$.  Combining Theorem~\ref{thm:products} with \cite[Example~G.1]{QC105} and \cite[\S5]{QC105} yields:
\[
G_{\MW{4}{3}}(t) = 
\sum_{l=0}^\infty \sum_{m=0}^\infty \frac{(3m)!}{(l!)^2 (m!)^5} t^{2l+2m}
\]

\subsubsection{$\MW{4}{4}$} \hfill [\hyperref[table:index_2]{regularized quantum period p.~\pageref*{table:index_2}}, \hyperref[operator:MW^4_4]{operator p.~\pageref*{operator:MW^4_4}}] \label{sec:MW^4_4}  
\label{sec:double_cover}

This is a double cover of $\PP^2 \times \PP^2$, branched over a divisor of bidegree $(2,2)$.  Consider  the toric
variety $F$ with weight data:
\[
\begin{array}{rrrrrrrl} 
  \multicolumn{1}{c}{x_0} & 
  \multicolumn{1}{c}{x_1} & 
  \multicolumn{1}{c}{x_2} & 
  \multicolumn{1}{c}{y_0} & 
  \multicolumn{1}{c}{y_1} & 
  \multicolumn{1}{c}{y_2} & 
  \multicolumn{1}{c}{w} & \\ 
  \cmidrule{1-7}
  1 & 1 & 1 & 0 & 0 & 0 & 1  & \hspace{1.5ex} L\\ 
  0 & 0 & 0 & 1 & 1 & 1 & 1 & \hspace{1.5ex} M \\
\end{array}
\]
and $\Amp F = \langle L, L+M \rangle$.  Let $X$ be a member of the linear system $|2L+2M|$ defined by the equation $w^2=f_{2,2}$, where $f_{2,2}$ is a bihomogeneous polynomial of degrees $2$ in $x_0$,~$x_1$,~$x_2$ and $2$ in $y_0$,~$y_1$,~$y_2$.  Let $p \colon F \dashrightarrow \PP^2 \times \PP^2$ be the rational map which sends (contravariantly) the homogeneous co-ordinate functions $[x_0,x_1,x_2,y_0,y_1,y_2] $ on $\PP^2_{x_0,x_1,x_2} \times \PP^2_{y_0,y_1,y_2}$ to $[x_0,x_1,x_2,y_0,y_1,y_2]$.  The restriction of $p$ to $X$ is a morphism, which exhibits $X$ as a double cover of $\PP^2 \times \PP^2$ branched over the locus $(f_{2,2}=0) \subset \PP^2_{x_0,x_1,x_2} \times \PP^2_{y_0,y_1,y_2}$.  Thus $X = \MW{4}{4}$.

Recall the definition of the $J$-function $J_X(t,z)$ from \cite[equation~11]{Coates--Givental}.  Recall from \cite{Coates} that there is a Lagrangian cone $\cL_X \subset H^\bullet(X;\Lambda_X)\otimes \CC(\!(z^{-1})\!)$ that encodes all genus-zero Gromov--Witten invariants of $X$, and a Lagrangian cone $\cL_\be \subset H^\bullet(F;\Lambda_F)\otimes \CC(\!(z^{-1})\!) \otimes \CC(\lambda)$ that encodes all genus-zero $(\be,2L+2M)$-twisted Gromov--Witten invariants of $F$.  Here $\Lambda_X$ and $\Lambda_F$ are certain Novikov rings and $\be$ is the total Chern class with parameter $\lambda$ (or, equivalently, $\be$ is the $S^1$-equivariant Euler class with respect to an action of $S^1$ described in \cite{Coates}; in this case one should regard $\lambda$ as the standard generator for the $S^1$-equivariant cohomology algebra of a point). The $J$-function $J_X$ is characterised by the fact that $J_X(t,{-z})$ is the unique point on $\cL_X$ of the form ${-z} + t + O(z^{-1})$.

Let $p_1$,~$p_2 \in H^2(F;\QQ)$ denote the first Chern class of $L$,~$L+M$ respectively and let $P_1$,~$P_2 \in H^2(X;\QQ)$ denote the pullbacks of $p_1$,~$p_2$ along the inclusion map $i \colon X \to F$.  Let $Q_1$,~$Q_2$ denote the elements of the Novikov ring $\Lambda_X$ that are dual respectively to $P_1$,~$P_2$, and note that $\Lambda_X$ and $\Lambda_F$ are canonically isomorphic (via $i_\star$). Theorem~21 in \cite{CCIT:toric_stacks_2} implies that:
\[
I(t_1,t_2,\lambda,z) = z e^{t_1 p_1/z} e^{t_2 p_2/z} 
\sum_{l=0}^\infty \sum_{m=0}^\infty
\frac{
  Q_1^l Q_2^m e^{l t_1} e^{m t_2} 
  \prod_{k=1}^{k=2m} (\lambda + 2p_2 + k z)
}{
  \prod_{k=1}^{k=l} (p_1 + k z)^3
  \prod_{k=1}^{k=m} (p_2 + k z)
}
\frac{
  \prod_{k={-\infty}}^{k=0} (p_2-p_1+kz)^3
}{
  \prod_{k={-\infty}}^{k=m-l} (p_2-p_1+kz)^3
}
\]
satisfies $I(t_1,t_2,\lambda,{-z}) \in \cL_\be$.  Theorem~1.1 in \cite{Coates} gives that $i^\star \cL_\be \big|_{\lambda = 0} \subset \cL_X$, and therefore that:
\[
i^\star I(t_1,t_2,0,{-z}) \in \cL_X
\]
Since the hypersurface $X$ misses the locus $y_1 = y_2 = y_3 = 0$ in $F$, we have that $i^\star (p_2-p_1)^3 = 0$.  Thus:
\[
i^\star I(t_1,t_2,0,z) =
z e^{t_1 P_1/z} e^{t_2 P_2/z} 
\sum_{l=0}^\infty \sum_{m=l}^\infty
\frac{
  Q_1^l Q_2^m e^{l t_1} e^{m t_2} 
  \prod_{k=1}^{k=2m} (2P_2 + k z)
}{
  \prod_{k=1}^{k=l} (P_1 + k z)^3
  \prod_{k=1}^{k=m} (P_2 + k z)
}
\frac{
  1
}{
  \prod_{k=1}^{k=m-l} (P_2-P_1+kz)^3
}
\]
In particular, $i^\star I(t_1,t_2,0,{-z})$ has the form ${-z} + t_1 P_1 + t_2 P_2 + O(z^{-1})$ and, from the characterisation of $J_X$ discussed above, we conclude that $J_X(t_1P_1+t_2P_2,{-z}) = i^\star I(t_1,t_2,0,{-z})$.  

To extract the quantum period $G_X$ from the $J$-function $J_X(t_1P_1+t_2P_2,z)$ we take the component along the unit class $1 \in H^\bullet(X;\QQ)$, set $z=1$, set $t_1 = t_2 = 0$, and set $Q_1 =1$,~$ Q_2 =t^2$, obtaining: 
\[
G_{\MW{4}{4}}(t) = \sum_{l=0}^\infty \sum_{m=l}^\infty \frac{(2m)!}{(l!)^3 m! ((m-l)!)^3} t^{2m}
\]
\subsubsection{$\MW{4}{5}$} \hfill [\hyperref[table:index_2]{regularized quantum period p.~\pageref*{table:index_2}}, \hyperref[operator:MW^4_5]{operator p.~\pageref*{operator:MW^4_5}}] \label{sec:MW^4_5}  

This is a divisor on $\PP^2 \times \PP^3$ of bidegree $(1,2)$.  Theorem~\ref{thm:toric_ci_mirror} yields:
\[
G_{\MW{4}{5}}(t) = \sum_{l=0}^\infty \sum_{m=0}^\infty \frac{(l+2m)!}{(l!)^3 (m!)^4} t^{2l+2m}
\]

\subsubsection{$\MW{4}{6}$} \hfill [\hyperref[table:index_2]{regularized quantum period p.~\pageref*{table:index_2}}, \hyperref[operator:MW^4_6]{operator p.~\pageref*{operator:MW^4_6}}] \label{sec:MW^4_6}

This is the product $\PP^1 \times B^3_4$.  Combining Theorem~\ref{thm:products} with \cite[Example~G.1]{QC105} and \cite[\S6]{QC105} yields:
\[
G_{\MW{4}{6}}(t) = 
\sum_{l=0}^\infty \sum_{m=0}^\infty \frac{(2m)!(2m)!}{(l!)^2 (m!)^6} t^{2l+2m}
\]

\subsubsection{$\MW{4}{7}$} \hfill [\hyperref[table:index_2]{regularized quantum period p.~\pageref*{table:index_2}}, \hyperref[operator:MW^4_7]{operator p.~\pageref*{operator:MW^4_7}}] \label{sec:MW^4_7}

This is a complete intersection of two divisors in $\PP^3 \times \PP^3$, each of bidegree $(1,1)$.  Theorem~\ref{thm:toric_ci_mirror} yields:
\[
G_{\MW{4}{7}}(t) = \sum_{l=0}^\infty \sum_{m=0}^\infty \frac{(l+m)!(l+m)!}{(l!)^4 (m!)^4} t^{2l+2m}
\]

\subsubsection{$\MW{4}{8}$} \hfill [\hyperref[table:index_2]{regularized quantum period p.~\pageref*{table:index_2}}, \hyperref[operator:MW^4_8]{operator p.~\pageref*{operator:MW^4_8}}] \label{sec:MW^4_8}  

This is a divisor on $\PP^2 \times Q^3$ of bidegree $(1,1)$.  Theorem~\ref{thm:toric_ci_mirror} yields:
\[
G_{\MW{4}{8}}(t) = \sum_{l=0}^\infty \sum_{m=0}^\infty \frac{(l+m)!(2m)!}{(l!)^3 (m!)^5} t^{2l+2m}
\]

\subsubsection{$\MW{4}{9}$} \hfill [\hyperref[table:index_2]{regularized quantum period p.~\pageref*{table:index_2}}, \hyperref[operator:MW^4_9]{operator p.~\pageref*{operator:MW^4_9}}] \label{sec:MW^4_9}  

This is the product $\PP^1 \times B^3_5$.  Combining Theorem~\ref{thm:products} with \cite[Example~G.1]{QC105} and \cite[\S7]{QC105} yields:
\[
G_{\MW{4}{9}}(t) = \sum_{l=0}^\infty \sum_{m=0}^\infty \sum_{n=0}^\infty
(-1)^{m+n}
t^{2l+2m+2n}
\frac
{
  \big((m+n)!\big)^3
}
{
  (l!)^2 (m!)^5 (n!)^5 
}
\big(1-5 (n-m)H_n \big)
\]

\subsubsection{$\MW{4}{10}$} \hfill [\hyperref[table:index_2]{regularized quantum period p.~\pageref*{table:index_2}}, \hyperref[operator:MW^4_10]{operator p.~\pageref*{operator:MW^4_10}}] \label{sec:MW^4_10}

This is the blow-up of the quadric $Q^4$ along a conic that is not contained in a plane lying in $Q^4$.  Consider the toric variety $F$ with weight data:
\[ 
\begin{array}{rrrrrrrl} 
  \multicolumn{1}{c}{s_0} & 
  \multicolumn{1}{c}{s_1} & 
  \multicolumn{1}{c}{s_2} & 
  \multicolumn{1}{c}{x} & 
  \multicolumn{1}{c}{x_3} & 
  \multicolumn{1}{c}{x_4} & 
  \multicolumn{1}{c}{x_5} & \\ 
  \cmidrule{1-7}
  1 & 1 & 1 & -1 & 0 & 0 & 0 & \hspace{1.5ex} L\\ 
  0 & 0 & 0 & 1 & 1  & 1 & 1 & \hspace{1.5ex} M \\
\end{array}
\]
and $\Amp F = \langle L, M\rangle$.  The morphism $F \to \PP^5$ that sends (contravariantly) the homogeneous co-ordinate functions $[x_0,x_1,\dots,x_5]$ to $[xs_0, xs_1, x s_2, x_3, x_4, x_5]$ blows up the plane $\Pi = (x_0=x_1=x_2=0)$ in $\PP^5$.  Thus a general member of $|2M|$ on $F$ is the blow-up of $Q^4$ with centre a conic on $\Pi$.  In other words, a general member of $|2M|$ on $F$ is $\MW{4}{10}$.  We have:
\begin{itemize}
\item $-K_F=2L+4M$ is ample, so that $F$ is a Fano variety;
\item $\MW{4}{10}\sim 2M$ is ample;
\item $-(K_F+2M)\sim 2L+2M$ is ample.
\end{itemize}
Theorem~\ref{thm:toric_ci_mirror} yields:
\[
G_{\MW{4}{10}}(t) = \sum_{l=0}^\infty \sum_{m=l}^\infty \frac{(2m)!}{(l!)^3(m-l)!(m!)^3} t^{2l+2m}
\]

\subsubsection{$\MW{4}{11}$} \hfill [\hyperref[table:index_2]{regularized quantum period p.~\pageref*{table:index_2}}, \hyperref[operator:MW^4_11]{operator p.~\pageref*{operator:MW^4_11}}] \label{sec:MW^4_11}
\label{sec:null_correlation}

This is the projective bundle $\PP_{\PP^3}(\cE^\vee)$, where $\cE \to \PP^3$ is the null-correlation bundle of Szurek--Wi\'sniewski~\cite{Szurek--Wisniewski}.  

\begin{remark}
  \label{rem:slice}
  For us $\PP(E)$ denotes the projective bundle of lines in $E$, whereas in Szurek--Wi\'sniewski and Iskovskikh--Prokhorov, $\PP(E)$ denotes the projective bundle of one-dimensional quotients.  With our conventions, if $\pi \colon \PP(E) \to X$ is a projective bundle then $E^\star = \pi_\star \cO_{\PP(E)}(1)$, and so a regular section $s \in \Gamma\big(\PP(E),\cO_{\PP(E)}(1)\big)$ vanishes on $\PP(F^\star) \subset \PP(E)$, where the vector bundle $F \to X$ is the cokernel of $s \colon \cO_{\PP(E)} \to E^\star$.
\end{remark}

\begin{proposition}
  \label{pro:better_model}
  Let $V = \CC^4$, so that $\PP(V) = \PP^3$.  Consider the partial flag manifold $\Fl_{1,2}(V)$ and the natural projections
  \begin{equation}
    \label{eq:flag_diagram}
    \begin{aligned}
      \xymatrix{ & \Fl_{1,2}(V) \ar[dl]_{p_1} \ar[dr]^{p_2} \\ \PP(V) && \Gr(2,V) }
    \end{aligned}
  \end{equation}
  Let $|L|$ denote the linear system defined by $\cO(1)$ for the projective bundle $p_1$.  Then a general element of $|L|$ is $\PP(\cE^\vee)$, where $\cE \to \PP(V)$ is the null-correlation bundle.
\end{proposition}

\begin{proof}
  The null-correlation bundle has rank~$2$, and so the perfect pairing $\cE \otimes \cE \to \det \cE$ gives canonical isomorphisms $\cE^\vee \cong \cE \otimes (\det \cE)^{-1}$ and $\PP(\cE^\vee) \cong \PP(\cE)$.
  There is an exact sequence:
  \[
  \xymatrix{0 \ar[r] & \cE(-1) \ar[r] & T_{\PP(V)}(-2) \ar[r] & \cO_{\PP(V)} \ar[r] & 0}
  \]
  and the map $s^\star \colon T_{\PP(V)}(-2) \to \cO_{\PP(V)}$ therein defines a section $s \in \Gamma\big(\PP(T_{\PP(V)}(-2)),\cO_{\PP(T_{\PP(V)}(-2))}(1)\big)$.  The construction in Remark~\ref{rem:slice} now exhibits $\PP\big(\cE(-1)\big) \cong \PP(\cE)$ as the locus $(s=0)$ in $\PP(T_{\PP(V)}(-2))$.  We will identify $\PP(T_{\PP(V)}(-2))$ with the partial flag manifold $\Fl_{1,2}(V)$.

  For a vector bundle $\cF \to X$ of rank~$3$, the perfect pairing $\cF \otimes \wedge^2 \cF \to \det \cF$ gives a canonical isomorphism $\cF^\star \cong (\wedge^2 \cF) \otimes (\det \cF)^{-1}$.  Applying this with $\cF \to X$ equal to $\Omega_{\PP(V)}(2) \to \PP(V)$ gives:
  \[
  T_{\PP(V)}(-2) \cong \Omega_{\PP(V)}^2(2)
  \]
  where $\Omega_{\PP(V)}^2 := \wedge^2 \Omega_{\PP(V)}$.  We thus need to identify $\PP\big(\Omega^2_{\PP(V)}(2)\big)$ with $\Fl_{1,2}(V)$.

  The Pl\"ucker embedding $\Gr(2,V) \to \PP\big(\wedge^2 V\big)$ maps a subspace $W \in \Gr(2,V)$ to the antisymmetric linear map $L_W \colon V^\star \to V$, well-defined up to scale, given by:
  \[
  L_W(f) = f(w_1) w_2 - f(w_2) w_1 
  \]
  where $\{w_1,w_2\}$ is a basis for $W$.  The kernel of $L_W$ is the annihilator $W^\perp \subset V^\star$.  If $f \not \in W^\perp$ then $\langle L_W(f) \rangle = \ker f \cap W$; this implies in particular that $\rk L_W = 2$.  Thus the image of the Pl\"ucker embedding consists of (the lines spanned by) antisymmetric linear maps $L_W \colon V^\star \to V$ of rank~$2$, and one can recover $W \in \Gr(2,V)$ from its image $\langle L_W \rangle$ by taking the annihilator of the kernel:
  \[
  W = \big(\ker L_W \big)^\perp
  \]
  There is a canonical isomorphism $\Ann \colon \Gr(2,V) \to \Gr(2,V^\star)$ which maps $W \in \Gr(2,V)$ to $W^\perp$.
  
  Recall that our goal is to identify $\PP\big(\Omega^2_{\PP(V)}(2)\big)$ with $\Fl_{1,2}(V)$.  Let $q_1 \colon \PP\big(\Omega^2_{\PP(V)}(2)\big) \to \PP(V)$ denote the projection.  The Euler sequence:
  \[
  \xymatrix{ 0 \ar[r] & \Omega_{\PP(V)} \ar[r] & \pi^\star V^\star (-1) \ar[r] & \cO_{\PP(V)} \ar[r] & 0}
  \]
  gives, via \cite[II, Exercise 5.16]{Hartshorne}:
  \[
  \xymatrix{ 0 \ar[r] & \Omega^2_{\PP(V)} \ar[r] & \pi^\star \big(\wedge^2 V^\star \big) (-2) \ar[r] & \Omega_{\PP(V)} \ar[r] & 0}
  \]
  and thus:
  \begin{equation}
    \label{eq:Omega_2_sequence}
    \xymatrix{ 0 \ar[r] & \Omega^2_{\PP(V)}(2) \ar[r] & \pi^\star \big(\wedge^2 V^\star\big) \ar[r] & \pi^\star V^\star (1) \ar[r] & \cO_{\PP(V)}(2) \ar[r] & 0}
  \end{equation}
  This defines a map $f \colon \PP\big(\Omega^2_{\PP(V)}(2)\big) \to \PP\big(\wedge^2 V^\star\big)$.  Consider the fiber of the sequence \eqref{eq:Omega_2_sequence} over $[v] \in \PP(V)$.  The map $\pi^\star \big(\wedge^2 V^\star \big) \to \pi^\star V^\star (1)$ here is given by contraction with $v$, and so non-zero elements of the kernel are antisymmetric linear maps $V \to V^\star$ of rank~$2$.  (They are antisymmetric, hence have rank $0$,~$2$, or~$4$; they are non-zero, hence are not of rank~$0$; and they have the non-zero element $v$ in their kernel, hence are not of rank~$4$.)\phantom{.}  In particular, we see that the image of $f$ lies in $\Gr(2,V^\star)  \subset \PP\big(\wedge^2 V^\star\big)$.  Given $[x] \in \PP\big(\Omega^2_{\PP(V)}(2)\big)$, write $W_{[x]} \subset V^\star$ for the linear subspace defined by $f([x])$.  Suppose that $[x] \in \PP\big(\Omega^2_{\PP(V)}(2)\big)$ lies over $[v] \in \PP(V)$.  Then, applying the discussion in the previous paragraph but with $V$ there replaced by $V^\star$, we see that $v \in W_{[x]}^\perp$.  Thus, writing $q_2 \colon \PP\big(\Omega^2_{\PP(V)}(2)\big) \to \Gr(2,V)$ for the composition
  \[
  \xymatrix{ \PP\big(\Omega^2_{\PP(V)}(2)\big) \ar[r]^f & \Gr(2,V^\star) \ar[r]^\Ann & \Gr(2,V) }
  \]
  we have that $q_1([x]) \subset q_2([x])$, i.e., that the diagram:
  \[
  \xymatrix{ 
    & \PP\big(\Omega^2_{\PP(V)}(2)\big) \ar[dl]_{q_1} \ar[rd]^{q_2} \\
    \PP(V) && \Gr(2,V)}
  \]
  coincides with the diagram~\eqref{eq:flag_diagram}.  This identifies $\PP\big(\Omega^2_{\PP(V)}(2)\big)$ with the partial flag manifold $\Fl_{1,2}(V)$, and exhibits $\PP(\cE^\vee)$ as an element of the linear system $|L|$ as claimed.
\end{proof}

\subsection*{Abelianization:}

To compute the quantum period, we use the Abelian/non-Abelian Correspondence of Ciocan-Fontanine--Kim--Sabbah, as in \cite[\S39]{QC105}.  Consider the situation as in \S3.1 of \cite{Ciocan-Fontanine--Kim--Sabbah} with:
\begin{itemize}
\item $X = \CC^{10}$, regarded as the space of pairs:
  \[
  \{(v,w) : \text{$v \in \CC^2$ is a row vector, $w$ is a $2 \times 4$ complex matrix}\}
  \]
\item $G = \Cstar \times \GL_2(\CC)$, acting on $X$ as:
  \[
  (\lambda, g) \colon (v, w) \mapsto (\lambda v g^{-1}, gw)
  \]
\item $T = (\Cstar)^3$, the diagonal subtorus in $G$;
\item the group that is denoted by $S$ in \cite{Ciocan-Fontanine--Kim--Sabbah} set equal to the trivial group;
\item $\cV$ equal to the representation of $G$ given by the determinant of the standard  representation of the second factor $\GL_2(\CC)$.
\end{itemize}
Then $X \GIT G$ is the partial flag manifold $\Fl=\Fl_{1,2}(\CC^4)$, whereas $X \GIT T$ is the toric variety with weight data:
\[
\begin{array}{rrrrrrrrrrl} 
1 & 1 & 1  & 1 & 0 & 0 & 0 & 0 & -1 & 0  & \hspace{1.5ex}  L_1\\ 
0 & 0 & 0  & 0 & 1 & 1 & 1 & 1 & 0   & -1& \hspace{1.5ex} L_2 \\
0 & 0 & 0  & 0 & 0 & 0 & 0 & 0 & 1   & 1  & \hspace{1.5ex} H
\end{array}
\]
and $\Amp =\langle L_1, L_2 , H \rangle$; that is, $X \GIT T$ is the projective bundle $\PP(\cO(-1,0)\oplus \cO(0,-1))$ over $\PP^3\times \PP^3$. The non-trivial element of the Weyl group $W=\ZZ/2\ZZ$ exchanges the two factors of $\PP^3\times \PP^3$. The representation $\cV$ induces the line bundle $\cV_G = L$ over $X \GIT G = \Fl$, where $L$ was defined in the statement of Proposition~\ref{pro:better_model}, whereas the representation $\cV$ induces the line bundle $\cV_T = L_1+L_2$ over $X \GIT T$.

\subsection*{The Abelian/non-Abelian Correspondence:}

Let $p_1$, $p_2$, and $p_3 \in H^2(X\GIT T;\QQ)$ denote the first Chern classes of the line bundles $L_1$, $L_2$, and $H$ respectively. We fix a lift of $H^\bullet(X \GIT G;\QQ)$ to $H^\bullet(X \GIT T,\QQ)^W$ in the sense of \cite[\S3]{Ciocan-Fontanine--Kim--Sabbah}; there are many possible choices for such a lift, and the precise choice made will be unimportant in what follows.  The lift allows us to regard $H^\bullet(X \GIT G;\QQ)$ as a subspace of $H^\bullet(X \GIT T,\QQ)^W$, which maps isomorphically to the Weyl-anti-invariant part $H^\bullet(X \GIT T,\QQ)^a$ of $H^\bullet(X \GIT T,\QQ)$ via:
\[
\xymatrix{
  H^\bullet(X \GIT T,\QQ)^W \ar[rr]^{\cup(p_2-p_1)} &&
  H^\bullet(X \GIT T,\QQ)^a}
\]
We compute the quantum period of $\MW{4}{11} \subset X \GIT G$ by computing the $J$-function of $\Fl = X \GIT G$ twisted, in the sense of \cite{Coates--Givental}, by the Euler class and the bundle $\cV_G$, using the Abelian/non-Abelian Correspondence.

Our first step is to compute the $J$-function of $X \GIT T$ twisted by the Euler class and the bundle $\cV_T$.  As in
\cite[\S D.1]{QC105} and as in \cite{Ciocan-Fontanine--Kim--Sabbah}, consider the bundles $\cV_T$ and $\cV_G$ equipped with the canonical $\Cstar$-action that rotates fibers and acts trivially on the base.  Recall the definition of the twisted $J$-function $J_{\be,\cV_T}$ of $X \GIT T$ from \cite[\S D.1]{QC105}.  We will compute $J_{\be,\cV_T}$ using the Quantum Lefschetz theorem; $J_{\be,\cV_T}$ is the restriction to the locus $\tau \in H^0(X \GIT T) \oplus H^2(X \GIT T)$ of what was denoted by $J^{S \times \Cstar}_{\cV_T}(\tau)$ in \cite{Ciocan-Fontanine--Kim--Sabbah}.  The toric variety $X \GIT T$ is Fano, so Theorem~C.1 in \cite{QC105} gives:
\[
J_{X \GIT T}(\tau) = e^{\tau/z} \sum_{l, m, n \geq 0} {
  Q_1^l Q_2^m Q_3^n 
  e^{l \tau_1} e^{m \tau_2} e^{m \tau_3}
  \over
  \prod_{k=1}^{k=l} (p_1 + k z)^4
  \prod_{k=1}^{k=m} (p_2 + k z)^4
}
{
\prod_{k =-\infty}^{k=0} p_3-p_1 + k z 
\over
\prod_{k=-\infty}^{k=n-l} p_3-p_1 + k z 
}
{
\prod_{k = -\infty}^{k=0} p_3-p_2 + k z 
\over
\prod_{k=-\infty}^{k= n-m} p_3-p_2 + k z 
}
\]
where $\tau = \tau_1 p_1 + \tau_2 p_2 + \tau_3 p_3$ and we have identified the group ring $\QQ[H_2(X \GIT T;\ZZ)]$ with $\QQ[Q_1,Q_2,Q_3]$ via the $\QQ$-linear map that sends $Q^\beta$ to $Q_1^{\langle \beta, p_1 \rangle} Q_2^{\langle \beta, p_2\rangle} Q_3^{\langle \beta, p_3 \rangle}$. The line bundles $L_1$, $L_2$, and $H$ are nef, and $c_1(X \GIT T) - c_1(\cV_T)$ is ample, so Theorem~D.3 in \cite{QC105} gives:
\begin{multline*}
  J_{\be,\cV_T}(\tau) = 
  e^{\tau/z}
  \sum_{l=0}^\infty   \sum_{m=0}^\infty   \sum_{n=0}^\infty
  Q_1^l Q_2^m Q_3^n e^{l \tau_1} e^{m \tau_2} e^{m \tau_3}
  {
    \prod_{k=1}^{k=l+m} (\lambda + p_1 + p_2 + k z)
    \over
    \prod_{k=1}^{k=l} (p_1 + k z)^4
    \prod_{k=1}^{k=m} (p_2 + k z)^4
  } \times \\
  {
    \prod_{k =-\infty}^{k=0} p_3-p_1 + k z 
    \over
    \prod_{k=-\infty}^{k=n-l} p_3-p_1 + k z 
  }
  {
    \prod_{k = -\infty}^{k=0} p_3-p_2 + k z 
    \over
    \prod_{k=-\infty}^{k= n-m} p_3-p_2 + k z 
  }
\end{multline*}

Consider now $\Fl = X \GIT G$ and a point $t \in H^\bullet(\Fl)$.  Recall that $\Fl=\PP(S)$ is the projectivization of the universal bundle $S$ of subspaces on $\Gr := \Gr(2,4)$. Let $\epsilon_1 \in H^2(\Fl;\QQ)$ be the pullback to $\Fl$ (under the projection map $p_2\colon\Fl\to \Gr$) of the ample generator of $H^2(\Gr)$, and let $\epsilon_2 \in H^2(\Fl;\QQ)$ be the first Chern class of $\cO_{\PP(S)}(1)$.  Identify the group ring $\QQ[H_2(\Fl;\ZZ)]$ with $\QQ[q_1,q_2]$ via the $\QQ$-linear map which sends $Q^\beta$ to $q_1^{\langle \beta,\epsilon_1 \rangle} q_2^{\langle \beta,\epsilon_2 \rangle}$.  In \cite[\S 6.1]{Ciocan-Fontanine--Kim--Sabbah} the authors consider the lift $\tilde{J}^{S \times \Cstar}_{\cV_G}(t)$ of their twisted $J$-function $J^{S \times \Cstar}_{\cV_G}(t)$ determined by a choice of lift $H^\bullet(X \GIT G;\QQ) \to H^\bullet(X \GIT T,\QQ)^W$.  We restrict to the locus $t \in H^0(X \GIT G;\QQ) \oplus H^2(X \GIT G;\QQ)$, considering the lift: 
\begin{align*}
  \tilde{J}_{\be,\cV_G}(t) := \tilde{J}^{S \times \Cstar}_{\cV_G}(t)   && t \in H^0(X \GIT G;\QQ) \oplus H^2(X \GIT G;\QQ)
\end{align*}
of our twisted $J$-function $J_{\be,\cV_G}$ determined by our choice of lift $H^\bullet(X \GIT G;\QQ) \to H^\bullet(X \GIT T,\QQ)^W$. Theorems~4.1.1 and~6.1.2 in \cite{Ciocan-Fontanine--Kim--Sabbah} imply that:
\[
\tilde{J}_{\be,\cV_G}\big(\varphi(t)\big) \cup (p_2 - p_1) = \Big[ 
\textstyle 
\big(z {\partial \over \partial \tau_2} 
- 
z {\partial \over \partial \tau_1} 
\big) J_{\be,\cV_T}(\tau) \Big]_{\tau=t, Q_1 = Q_2 = -q_1, Q_3=q_2}
\]
for some function $\varphi:H^2(X \GIT G;\QQ) \to H^\bullet(X \GIT G; \Lambda_G)$. Setting $t = 0$ gives:
\begin{multline*}
  \tilde{J}_{\be,\cV_G}\big(\varphi(0)\big) \cup (p_2 - p_1) = 
  \sum_{l=0}^\infty   \sum_{m=0}^\infty   \sum_{n=0}^\infty
  (-1)^{l+m} q_1^{l+m} q_2^n 
  {
    \prod_{k=1}^{k=l+m} (\lambda + p_1 + p_2 + k z)
    \over
    \prod_{k=1}^{k=l} (p_1 + k z)^4
    \prod_{k=1}^{k=m} (p_2 + k z)^4
  } \times \\
  {
    \prod_{k =-\infty}^{k=0} p_3-p_1 + k z 
    \over
    \prod_{k=-\infty}^{k=n-l} p_3-p_1 + k z 
  }
  {
    \prod_{k = -\infty}^{k=0} p_3-p_2 + k z 
    \over
    \prod_{k=-\infty}^{k= n-m} p_3-p_2 + k z 
  }
  \big(p_2-p_1 + (m-l) z \big)
\end{multline*}
For symmetry reasons the right-hand side here is divisible by $p_2-p_1$; it takes the form:
\[
(p_2-p_1)\Big(1 + O(z^{-2})\Big)
\]
whereas:
\[
\tilde{J}_{\be,\cV_G}\big(\varphi(0)\big) \cup (p_2 - p_1) = 
(p_2-p_1) \Big(1 + \varphi(0) z^{-1} + O(z^{-2}) \Big)
\]
We conclude that $\varphi(0) = 0$.  Thus:
\begin{multline}
  \label{eq:almost_there}
  \tilde{J}_{\be,\cV_G}(0) \cup (p_2 - p_1) = 
  \sum_{l=0}^\infty   \sum_{m=0}^\infty   \sum_{n=0}^\infty
  (-1)^{l+m} q_1^{l+m} q_2^n 
  {
    \prod_{k=1}^{k=l+m} (\lambda + p_1 + p_2 + k z)
    \over
    \prod_{k=1}^{k=l} (p_1 + k z)^4
    \prod_{k=1}^{k=m} (p_2 + k z)^4
  } \times \\
  {
    \prod_{k =-\infty}^{k=0} p_3-p_1 + k z 
    \over
    \prod_{k=-\infty}^{k=n-l} p_3-p_1 + k z 
  }
  {
    \prod_{k = -\infty}^{k=0} p_3-p_2 + k z 
    \over
    \prod_{k=-\infty}^{k= n-m} p_3-p_2 + k z 
  }
  \big(p_2-p_1 + (m-l) z \big)
\end{multline}

To extract the quantum period $G_{\MW{4}{11}}$ from the twisted $J$-function $J_{\be,\cV_G}(0)$, we proceed as in \cite[Example~D.8]{QC105}: we take the non-equivariant limit, extract the component along the unit class $1 \in H^\bullet(X \GIT G;\QQ)$, set $z=1$, and set $Q^\beta = t^{\langle \beta, {-K} \rangle}$ where $K = K_{\MW{4}{11}}$.  Thus we consider the right-hand side of \eqref{eq:almost_there}, take the non-equivariant limit, extract the coefficient of $p_2-p_1$, set $z=1$, and set $q_1 = q_2 =2t$, obtaining: 
\begin{multline*}
  G_{\MW{4}{11}}(t) = 
  \sum_{l=0}^\infty \sum_{m=0}^\infty \sum_{n=\max(l,m)}^\infty 
  (-1)^{l+m}
  t^{2l+2m+2n}
  \frac
  {
    (l+m)!
  }
  {
    (l!)^4 (m!)^4 (n-l)! (n-m)!
  }
\Big(1 + (m-l)(H_{n-m}-4H_m)\Big)
  \\
  +
\sum_{l=0}^\infty \sum_{m=l+1}^\infty \sum_{n=l}^{m-1}
  (-1)^{l+n}
  t^{2l+2m+2n}
  \frac
  {
    (l+m)!
    (m-n-1)!
  }
  {
    (l!)^4 (m!)^4 (n-l)!
  }
  (m-l)
\end{multline*}

\begin{remark}
  The quantum period of $\PP\big(\Omega^2_{\PP(V)}(2)\big)$ can also be computed using Strangeway's reconstruction theorem for the quantum cohomology of Fano bundles~\cite[Theorem~1]{Strangeway}.  Thus the quantum period of $\MW{4}{11}$ can be derived from this result together with the Quantum Lefschetz theorem.  The Gromov--Witten invariants required as input to the reconstruction theorem can be computed via \cite[Lemma~1]{Strangeway}, using Schubert calculus on $\Gr(2,4)$ and intersection numbers in $\PP^3$.
\end{remark}

\subsubsection{$\MW{4}{12}$} \hfill [\hyperref[table:index_2]{regularized quantum period p.~\pageref*{table:index_2}}, \hyperref[operator:MW^4_12]{operator p.~\pageref*{operator:MW^4_12}}] \label{sec:MW^4_12}

This is the blow-up of the quadric $Q^4$ along a line.  Consider the toric variety $F$ with weight data:
\[ 
\begin{array}{rrrrrrrl} 
  \multicolumn{1}{c}{s_0} & 
  \multicolumn{1}{c}{s_1} & 
  \multicolumn{1}{c}{s_2} & 
  \multicolumn{1}{c}{s_3} & 
  \multicolumn{1}{c}{x} & 
  \multicolumn{1}{c}{x_4} & 
  \multicolumn{1}{c}{x_5} & \\ 
  \cmidrule{1-7}
  1 & 1 & 1 & 1 & -1 & 0 & 0 & \hspace{1.5ex} L\\ 
  0 & 0 & 0 & 0 & 1  & 1 & 1 & \hspace{1.5ex} M \\
\end{array}
\]
and $\Amp F = \langle L, M\rangle$.  The morphism $F \to \PP^5$ that sends (contravariantly) the homogeneous co-ordinate functions $[x_0,x_1,\dots,x_5]$ to $[xs_0, xs_1, x s_2, x s_3, x_4, x_5]$ blows up the line $(x_0=x_1=x_2=x_3=0)$ in $\PP^5$, and $\MW{4}{12}$ is the proper transform of a quadric containing this line.  Thus $\MW{4}{12}$ is a member of $|L+M|$ in the toric variety $F$.  We have:
\begin{itemize}
\item $-K_F=3L+3M$ is ample, so that $F$ is a Fano variety;
\item $\MW{4}{12}\sim L+M$ is ample;
\item $-(K_F+L+M)\sim 2L+2M$ is ample.
\end{itemize}
Theorem~\ref{thm:toric_ci_mirror} yields:
\[
G_{\MW{4}{12}}(t) = \sum_{l=0}^\infty \sum_{m=l}^\infty \frac{(l+m)!}{(l!)^4(m-l)!(m!)^2} t^{2l+2m}
\]

\subsubsection{$\MW{4}{13}$} \hfill [\hyperref[table:index_2]{regularized quantum period p.~\pageref*{table:index_2}}, \hyperref[operator:MW^4_13]{operator p.~\pageref*{operator:MW^4_13}}] \label{sec:MW^4_13}  

This is the projective bundle $\PP_{Q^3}\big(\cO(1) \oplus \cO\big)$ or, equivalently, a member of $|2L|$ in the toric variety $F$ with weight data:
\[
\begin{array}{rrrrrrrl} 
  \multicolumn{1}{c}{x_0} & 
  \multicolumn{1}{c}{x_1} & 
  \multicolumn{1}{c}{x_2} & 
  \multicolumn{1}{c}{x_3} & 
  \multicolumn{1}{c}{x_4} & 
  \multicolumn{1}{c}{u} & 
  \multicolumn{1}{c}{v} & \\ 
  \cmidrule{1-7}
  1 & 1 & 1 & 1 & 1 & 0 & -1 & \hspace{1.5ex} L\\ 
  0 & 0 & 0 & 0 & 0 & 1 & 1 & \hspace{1.5ex} M \\
\end{array}
\]
and $\Amp F = \langle L, M\rangle$.  We have:
\begin{itemize}
\item $-K_F=4L+2M$ is ample, that is $F$ is a Fano variety;
\item $\MW{4}{13} \sim 2L$ is nef;
\item $-(K_F+2L)\sim 2L+2M$ is ample.
\end{itemize}
The projection $[x_0:x_1:x_2:x_3:x_4:x_5:u:v] \mapsto [x_0:x_1:x_2:x_3:x_4:x_5]$ exhibits $F$ as the scroll $\PP_{\PP^4}\big(\cO(1) \oplus \cO\big)$ over $\PP^4$, and passing to a member of $|2L|$ restricts this scroll to $Q^3 \subset \PP^4$.  Theorem~\ref{thm:toric_ci_mirror} yields:
\[
G_{\MW{4}{13}}(t) = \sum_{l=0}^\infty \sum_{m=l}^\infty \frac{(2l)!}{(l!)^5 m! (m-l)!} t^{2l+2m}
\]

\subsubsection{$\MW{4}{14}$} \hfill [\hyperref[table:index_2]{regularized quantum period p.~\pageref*{table:index_2}}, \hyperref[operator:MW^4_14]{operator p.~\pageref*{operator:MW^4_14}}] \label{sec:MW^4_14}

This is the product $\PP^1 \times \PP^3$.  Combining Theorem~\ref{thm:products} with \cite[Example~G.1]{QC105} and \cite[\S1]{QC105} yields:
\[
G_{\MW{4}{14}}(t) = \sum_{l=0}^\infty \sum_{m=0}^\infty \frac{t^{2l+4m}}{(l!)^2 (m!)^4}
\]

\subsubsection{$\MW{4}{15}$} \hfill [\hyperref[table:index_2]{regularized quantum period p.~\pageref*{table:index_2}}, \hyperref[operator:MW^4_15]{operator p.~\pageref*{operator:MW^4_15}}] \label{sec:MW^4_15}  

This is the projective bundle $\PP_{\PP^3}\big(\cO(1) \oplus \cO(-1)\big)$, or in other words, the toric variety with weight data:
\[
\begin{array}{rrrrrrl} 
  \multicolumn{1}{c}{x_0} & 
  \multicolumn{1}{c}{x_1} & 
  \multicolumn{1}{c}{x_2} & 
  \multicolumn{1}{c}{x_3} & 
  \multicolumn{1}{c}{u} & 
  \multicolumn{1}{c}{v} \\ 
  \cmidrule{1-6}
  1 & 1 & 1 & 1 & 0 & -2 & \hspace{1.5ex} L\\ 
  0 & 0 & 0 & 0 & 1 & 1 & \hspace{1.5ex} M\\ 
\end{array}
\]
and $\Amp F = \langle L, M\rangle$.  Theorem~\ref{thm:toric_mirror} yields:
\[
G_{\MW{4}{15}}(t) = \sum_{l=0}^\infty \sum_{m=2l}^\infty \frac{t^{2l+2m}}{(l!)^4 m! (m-2l)!} 
\]

\subsubsection{$\MW{4}{16}$} \hfill [\hyperref[table:index_2]{regularized quantum period p.~\pageref*{table:index_2}}, \hyperref[operator:MW^4_16]{operator p.~\pageref*{operator:MW^4_16}}] \label{sec:MW^4_16}

This is the product $\PP^1 \times W^3$, where $W^3 \subset \PP^2 \times \PP^2$ is a divisor of bidegree $(1,1)$.  Theorem~\ref{thm:toric_ci_mirror} yields:
\[
G_{\MW{4}{16}}(t) = \sum_{l=0}^\infty \sum_{m=0}^\infty \sum_{n=0}^\infty \frac{(m+n)!}{(l!)^2 (m!)^3 (n!)^3} t^{2l+2m+2n}
\]

\subsubsection{$\MW{4}{17}$} \hfill [\hyperref[table:index_2]{regularized quantum period p.~\pageref*{table:index_2}}, \hyperref[operator:MW^4_17]{operator p.~\pageref*{operator:MW^4_17}}] \label{sec:MW^4_17}

This is the product $\PP^1 \times B^3_7$, where $B^3_7$ is the blow-up of $\PP^3$ at a point.  Note that $B^3_7$ is the projective bundle $\PP_{\PP^2}(\cO \oplus \cO(-1))$.  It follows that $\MW{4}{17}$ is the toric variety with weight data:
\[
\begin{array}{rrrrrrrl} 
  \multicolumn{1}{c}{x_0} & 
  \multicolumn{1}{c}{x_1} & 
  \multicolumn{1}{c}{y_0} & 
  \multicolumn{1}{c}{y_1} & 
  \multicolumn{1}{c}{y_2} & 
  \multicolumn{1}{c}{u} &
  \multicolumn{1}{c}{v} \\ 
  \cmidrule{1-7}
  1 & 1 & 0 & 0 & 0 & 0 & 0 & \hspace{1.5ex} L\\ 
  0 & 0 & 1 & 1 & 1 & 0 & -1 & \hspace{1.5ex} M\\ 
  0 & 0 & 0 & 0 & 0 & 1 & 1 & \hspace{1.5ex} N\\ 
\end{array}
\]
and $\Amp F = \langle L, M, N\rangle$.  Theorem~\ref{thm:toric_mirror} yields:
\[
G_{\MW{4}{17}}(t) = \sum_{l=0}^\infty \sum_{m=0}^\infty \sum_{n=m}^\infty \frac{t^{2l+2m+2n}}{(l!)^2 (m!)^3 n! (n-m)!}
\]

\subsubsection{$\MW{4}{18}$} \hfill [\hyperref[table:index_2]{regularized quantum period p.~\pageref*{table:index_2}}, \hyperref[operator:MW^4_18]{operator p.~\pageref*{operator:MW^4_18}}] \label{sec:MW^4_18}

This is the product $\PP^1 \times \PP^1 \times \PP^1 \times \PP^1$.  Combining Theorem~\ref{thm:products} with \cite[Example~G.1]{QC105} yields:
\[
G_{\MW{4}{18}}(t) = \sum_{k=0}^\infty \sum_{l=0}^\infty \sum_{m=0}^\infty \sum_{n=0}^\infty \frac{t^{2k+2l+2m+2n}}{(k!)^2(l!)^2(m!)^2(n!)^2}
\]

\section{Four-Dimensional Fano Toric Manifolds}
\label{sec:toric}

Four-dimensional Fano toric manifolds were classified by Batyrev~\cite{Batyrev} and Sato~\cite{Sato}.  {\O}bro classified Fano toric manifolds in dimensions 2--8~\cite{Obro} and, to standardise notation, we will write $\Obro{4}{k}$ for the $k$th four-dimensional Fano toric manifold in {\O}bro's list.  $\Obro{4}{k}$ is the $(23+k)$th Fano toric manifold in the Graded Ring Database~\cite{GRDB}, as the list there is the concatenation of {\O}bro's lists in dimensions 2--8. We can compute the quantum periods of the $\Obro{4}{k}$ using Theorem~\ref{thm:toric_mirror};  the first few Taylor coefficients of their regularized quantum periods can be found in the tables in the Appendix.  

\section{Product Manifolds and Other Index $1$ Examples}
\label{sec:other_index_1}

Quantum periods for one-, two- and three-dimensional Fano manifolds were computed in~\cite{QC105}.  Combining these results with Theorem~\ref{thm:products} allows us to compute the quantum period of any four-dimensional Fano manifold that is a product of lower-dimensional manifolds.  Many of these examples have Fano index $r=1$.

In his thesis~\cite{Strangeway:thesis}, Strangeway determined the quantum periods of two four-dimensional Fano manifolds of index $r=1$ that have not yet been discussed.  These manifolds arise as complete intersections in the $9$-dimensional projective bundle $F = \PP\big(\Omega^2_{\PP^4}(2) \big)$.  Let $\pi \colon F \to \PP^4$ denote the canonical projection, let $p \in H^2(F)$ be the first Chern class of $\pi^\star \cO_{\PP^4}(1)$, and let $\xi \in H^2(F)$ be the first Chern class of the tautological bundle $\cO_F(1)$.
The manifold $F$ is Fano of Picard rank~$2$, with nef cone generated by $\{\xi,p\}$ and ${-K_F} = 6\xi + 2p$.  Let:
\begin{align*}
  & \text{$\Str_1 \subset F$ denote a complete intersection of five divisors of type $\xi$} \\
  & \text{$\Str_2 \subset F$ denote a complete intersection of four divisors of type $\xi$ and a divisor of type $p$} \\
  \intertext{We consider also:}
  & \text{$\Str_3 \subset F$, a complete intersection of four divisors of type $\xi$ and a divisor of type $\xi + p$} 
\end{align*}
which was unaccountably omitted from \cite{Strangeway:thesis}.  

The manifolds $\Str_k$, $k \in \{1,2,3\}$, each have Picard rank two.  To see this, observe that the ambient manifold $F$ is the blow-up of $\PP^9$ along $\Gr(2,5)$, where $\Gr(2,5) \to \PP^9$ is the Pl\"ucker embedding~\cite{Strangeway}; the blow-up $F \to \PP^9$ and the projection $\pi \colon F \to \PP^4$ are the extremal contractions corresponding to the two extremal rays in $\MC(F)$.  Thus $\Str_1$ is the blow-up of $\PP^4$ along an elliptic curve $E_5 \subset \PP^4$ of degree~$5$.  Consider the five-dimensional Fano manifold $F_5$ given by the complete intersection of four divisors of type $\xi$ in $F$.  Then $F_5$ is the blow-up of $\PP^5$ along a del~Pezzo surface $S_5$ of degree~$5$; in particular, $F_5$ has Picard rank two.  $\Str_3$ is an ample divisor (of type $\xi + p$) in $F_5$, so the Picard rank of $\Str_3$ is also two.  The manifold $\Str_2$ is a divisor in $F_5$ of type~$p$, and $F_5$ arises as the closure of the graph of the map $\PP^5 \to \PP^4$ given by the $5$-dimensional linear system of quadrics passing through $S_5$. This exhibits $\Str_2$ as the blow-up of a smooth four-dimensional quadric $Q^4$ along $S_5$, which implies that the Picard rank of $\Str_2$ is two.

We can compute the quantum periods of $\Str_k$, $k \in \{1,2,3\}$, by observing that a complete intersection in $F$ of five divisors of type~$\xi$ and one divisor of type~$p$ is a three-dimensional Fano manifold~$\MM{2}{17}$, `unsectioning' to compute the quantum period for $F$, and then applying the quantum Lefschetz theorem to compute the quantum periods for $\Str_1$,~$\Str_2$, and~$\Str_3$.  Recall the definition of the $J$-function $J_X(t,z)$ from \cite[equation~11]{Coates--Givental}.  The identity component of the $J$-function of $\MM{2}{17}$ is:
\begin{equation}
  \label{eq:equate_me_1}
  e^{-q_1 - q_2} 
  \sum_{l_1,l_2,l_3\geq 0}
  (-q_1)^{l_1+l_2} q_2^{l_3}
  \frac
  {
    (l_1+l_2)! (l_1+l_3)!(l_2+l_3)! (l_1+l_2+l_3)!
  }
  {
    (l_1!)^4 (l_2!)^4 (l_3!)^4 z^{l_1+l_2+l_3}
  }
  \Big(1 + (l_2-l_1)(H_{l_2+l_3} - 4H_{l_2})\Big)
\end{equation}
where $q_1$,~$q_2$ are generators of the Novikov ring for $\MM{2}{17}$ dual respectively to $\xi$ and $p$; see~\cite[\S 34]{QC105}.   The identity component of the $J$-function of $F$ takes the form:
\[
\sum_{l=0}^\infty \sum_{m=0}^\infty c_{l,m} z^{-6l-2m} q_1^l q_2^m
\]
for some coefficients $c_{l,m} \in \QQ$.  The Quantum Lefschetz theorem implies (cf.~\cite[\S D.1]{QC105}) that the identity component of the $J$-function of $\MM{2}{17}$ is equal to:
\begin{equation}
  \label{eq:equate_me_2}
  e^{-c_{1,0} q_1 -c_{0,1} q_2} \sum_{l=0}^\infty \sum_{m=0}^\infty (l!)^5 m! c_{l,m} z^{-l-m} q_1^l q_2^m
\end{equation}
and it is known that $c_{1,0} = 1$ and $c_{0,1} = 0$~\cite[\S5.1]{Strangeway}.  
Equating \eqref{eq:equate_me_1} and \eqref{eq:equate_me_2} determines the $c_{l,m}$:
\[
c_{l,m}=
\sum_{i=0}^{l} \sum_{j=0}^{m}
(-1)^{j+l}
\frac
{
  (m+l-i-j)!(i+m-j)! (m+l-j)!
}
{
  ((l-i)!)^4 (i!)^4 ((m-j)!)^4j!m!(l!)^4
}
\Big(1 + (2i-l)(H_{i+m-j} - 4H_{i})\Big)
\]
The Quantum Lefschetz theorem now gives that:
\begin{align*}
  & G_{\Str_1}(t) = e^{-t}
  \sum_{l=0}^\infty
  \sum_{m=0}^\infty
  (l!)^5 c_{l,m} t^{l+2m} \\
  & G_{\Str_2}(t) = 
  \sum_{l=0}^\infty
  \sum_{m=0}^\infty
  (l!)^4 m! \, c_{l,m} t^{2l+m} \\
  & G_{\Str_3}(t) = 
  e^{-t} \sum_{l=0}^\infty
  \sum_{m=0}^\infty
  (l!)^4 (l+m)! \, c_{l,m} t^{l+m} 
\end{align*}

\section{Numerical Calculations of Quantum Differential Operators}
\label{sec:qdes}

As discussed in \S\ref{sec:methodology}, the regularized quantum period $\hG_X(t)$ of a Fano manifold $X$ satisfies a differential equation:
\begin{align}
  \label{eq:regularized_qde_again}
  L_X \hG_X \equiv 0 && L_X = \sum_{k=0}^{k=N} p_k(t) D^k
\end{align}
called the regularized quantum differential equation.  Here the $p_m$ are polynomials and $D = t \frac{d}{dt}$.  The regularized quantum differential equation for $X$ coincides with the (unregularized) quantum differential equation for an anticanonical Calabi--Yau manifold $Y \subset X$; the study of the regularized quantum period from this point of view was pioneered by Batyrev--Ciocan-Fontanine--Kim--van~Straten~\cite{BCFKvS:1,BCFKvS:2}.  The differential equation \eqref{eq:regularized_qde_again} is expected to be Fuchsian, and the local system of solutions to $L_X f \equiv 0$ is expected to be of \emph{low ramification} in the following sense.

\begin{definition}[\!\!\protect{\cite{ProcECM}}]
  Let $S \subset \PP^1$ a finite set, and $\bbV \to \PP^1 \setminus S$ a local system.  Fix a basepoint $x \in \PP^1 \setminus S$.  For $s \in S$, choose a small loop that winds once anticlockwise around $s$ and connect it to $x$ via a path, thereby making a loop $\gamma_s$ about $s$ based at $x$.  Let $T_s \colon \bbV_x \to \bbV_x$ denote the monodromy of $\bbV$  along $\gamma_s$.  The \emph{ramification} of $\bbV$ is:
  \[
  \rf(\bbV) := \sum_{s \in S} \dim\Big(\bbV_x/{\bbV_x}^{\!\!\!T_s}\Big)
  \]
  The \emph{ramification defect} of $\bbV$ is the quantity $\rf(\bbV) - 2\rk(\bbV)$.  Non-trivial irreducible local systems $\bbV \to \PP^1 \setminus S$ have non-negative ramification defect; this gives a lower bound for the ramification of $\bbV$.  A local system of ramification defect zero is called \emph{extremal}.
\end{definition}

\begin{definition}
  The \emph{ramification} (respectively \emph{ramification defect}) of a differential operator $L_X$ is the ramification (respectively ramification defect) of the local system of solutions $L_X f \equiv 0$.
\end{definition}

\begin{definition}
  The \emph{quantum differential operator} for a Fano manifold $X$ is the operator $L_X \in \QQ[t]\langle D \rangle$ such that $L_X \hG_X \equiv 0$ which is of lowest order in $D$ and, among all such operators of this order, is of lowest degree in $t$.  (This defines $L_X$ only up to an overall scalar factor, but this suffices for our purposes.)
\end{definition}
Suppose that each of the polynomials $p_0,\ldots,p_N$ are of degree at most $r$, and write:
\begin{align*}
  L_X = \sum_{k=0}^{k=N} p_k(t) D^k && p_k(t) = \sum_{l=0}^r a_{kl} t^l
\end{align*}
The differential equation $L_X \hG_X \equiv 0$ gives a system of linear equations for the coefficients $a_{kl}$ which, given sufficiently many terms of the Taylor expansion of $\hG_X$, becomes over-determined.  Given \emph{a~priori} bounds on $N$ and $r$, therefore, we could compute the quantum differential operator $L_X$ by calculating sufficiently many terms in the Taylor expansion.  In general we do not have such bounds, but nonetheless by ensuring the linear system for $(a_{kl})$ is highly over-determined we can be reasonably confident that the operator $L_X$ which we compute is correct.  In addition, since $L_X$ is expected to correspond under mirror symmetry to a Picard--Fuchs differential equation for the mirror family, $L_X$ is expected to be of Fuchsian type.  This is an extremely delicate condition on the coefficients $(a_{kl})$, and it can be checked by exact computation.  

We computed candidate quantum differential operators $L_X$ for all four-dimensional Fano manifolds of Fano index $r>1$, and checked the Fuchsian condition in each case.  The operators $L_X$, together with their ramification defects and the log-monodromy data $\{\log T_s : s \in S\}$ in Jordan normal form, can be found in Appendix~\ref{appendix:qdes}.  In $24$ cases, the local system of solutions to the regularized quantum differential equation is extremal, and in the remaining $11$ cases it is of ramification defect~$1$.

To compute the ramification of $L_X$, we follow Kedlaya~\cite[\S7.3]{Kedlaya}.  This involves only linear algebra over a splitting field for $p_N(t)$---recall that every singular point of $L_X$ occurs at a root of $p_N(t)$---and thus can be implemented using exact (not numerical) computer algebra.  For this we use Steel's symbolic implementation of $\overline{\QQ}$ in the computational algebra system Magma~\cite{Magma,Steel}.  

\subsection*{Source Code} This paper is accompanied by full source code, written in Magma.  See the included file {\tt README.txt} for usage instructions.  The source code, but not the text of this paper, is released under a Creative Commons~CC0 license~\cite{CC0}: see the included file {\tt COPYING.txt} for details.    If you make use of the source code in an academic or commercial context, you should acknowledge this by including a reference or citation to this paper.

\subsection*{Acknowledgements}

We thank Alessio Corti for a number of very useful conversations.    This research was supported by a Royal Society University Research Fellowship (TC); ERC Starting Investigator Grant number~240123; the Leverhulme Trust; grant MK-1297.2014.1; AG Laboratory NRU-HSE, RF government grant ag.~11.G34.31.0023; Grant of Leading Scientific Schools (N.Sh.~2998.2014.1); and EPSRC grant EP/I008128/1. 

\pagebreak

\appendix

\section{Regularized Quantum Period Sequences}
\label{appendix:periods}

In this Appendix we record the description, degree, and Picard rank~$\rho_X$ for each of the four-dimensional Fano manifolds $X$ considered in this paper, together with the first few Taylor coefficients $\alpha_d$ of the regularized quantum period:
\[
\hG_X(t) = \sum_{d=0}^\infty \alpha_d t^d
\]
The tables are divided by Fano index $r$.  We include only coefficients $\alpha_d$ with $d \equiv 0 \bmod r$, since coefficients $\alpha_d$ with $d \not \equiv 0 \bmod r$ are zero.   Notation is as follows:
\begin{itemize}
\item $\PP^n$ denotes $n$-dimensional complex projective space;
\item $Q^n$ denotes a quadric hypersurface in $\PP^{n+1}$;
\item $\FI{4}{k}$ is as in \S\ref{sec:index_3} above;
\item $\VV{4}{k}$ is as in \S\ref{sec:index_2_rank_1} above;
\item $\MW{4}{k}$ is as in \S\ref{sec:index_2_higher_rank} above;
\item $\Obro{4}{k}$ is as in \S\ref{sec:toric} above;
\item $\Str_{k}$ is as in \S\ref{sec:other_index_1} above;
\item $\SS_k$ denotes the del~Pezzo surface of degree~$k$;
\item $F_1$ denotes the Hirzebruch surface $\PP\big(\cO_{\PP^1}(-1) \oplus \cO_{\PP^1}\big)$;
\item $\VV{3}{k}$ denotes the three-dimensional Fano manifold of Picard rank~$1$, Fano index~$1$, and degree~$k$;
\item $\BB{3}{k}$ denotes the three-dimensional Fano manifold of Picard rank~$1$, Fano index~$2$, and degree~$8k$;
\item $\MM{\rho}{k}$ denotes the $k$th entry in the Mori--Mukai list of three-dimensional Fano manifolds of Picard rank~$\rho$~\cite{Mori--Mukai:Manuscripta,Mori--Mukai:81,Mori--Mukai:84,Mori--Mukai:erratum,Mori--Mukai:fanoconf}.  We use the the ordering as in~\cite{QC105}, which agrees with the original papers of Mori--Mukai except when $\rho=4$.
\end{itemize}
We prefer to express manifolds as products of lower-dimensional manifolds where possible, so for example $\Obro{4}{122}$ is the product $\PP^1 \times \PP^3$, but we refer to this space as $\PP^1 \times \PP^3$ rather than $\Obro{4}{122}$.   The tables for Fano index~$r$ with $r \in \{2,3,4,5\}$ are complete.  The table for $r=1$ is very far from complete.

\begin{longtable}{cccccccc}
\caption{Four-dimensional Fano manifolds with Fano index $r=5$}
\label{table:index_5}\\
\toprule
\multicolumn{1}{c}{$X$}&\multicolumn{1}{c}{$(-K_X)^4$}&\multicolumn{1}{c}{$\rho_X$}&\multicolumn{1}{c}{$\alpha_{0}$}&\multicolumn{1}{c}{$\alpha_{5}$}&\multicolumn{1}{c}{$\alpha_{10}$}&\multicolumn{1}{c}{$\alpha_{15}$}&\multicolumn{1}{c}{$\alpha_{20}$}\\
\midrule
\endfirsthead
\multicolumn{8}{l}{\vspace{-0.7em}\tiny Continued from previous page.}\\
\addlinespace[1.7ex]
\midrule
\multicolumn{1}{c}{$X$}&\multicolumn{1}{c}{$(-K_X)^4$}&\multicolumn{1}{c}{$\rho_X$}&\multicolumn{1}{c}{$\alpha_{0}$}&\multicolumn{1}{c}{$\alpha_{5}$}&\multicolumn{1}{c}{$\alpha_{10}$}&\multicolumn{1}{c}{$\alpha_{15}$}&\multicolumn{1}{c}{$\alpha_{20}$}\\
\midrule
\endhead
\midrule
\multicolumn{8}{r}{\raisebox{0.2em}{\tiny Continued on next page.}}\\
\endfoot
\bottomrule
\endlastfoot
\oddrow \hyperref[sec:index_5]{$\PP^4$}&$625$&$1$&$1$&$120$&$113400$&$168168000$&$305540235000$\\
\addlinespace[1.1ex]
\end{longtable}

\begin{longtable}{cccccccc}
\caption{Four-dimensional Fano manifolds with Fano index $r=4$}
\label{table:index_4}\\
\toprule
\multicolumn{1}{c}{$X$}&\multicolumn{1}{c}{$(-K_X)^4$}&\multicolumn{1}{c}{$\rho_X$}&\multicolumn{1}{c}{$\alpha_{0}$}&\multicolumn{1}{c}{$\alpha_{4}$}&\multicolumn{1}{c}{$\alpha_{8}$}&\multicolumn{1}{c}{$\alpha_{12}$}&\multicolumn{1}{c}{$\alpha_{16}$}\\
\midrule
\endfirsthead
\multicolumn{8}{l}{\vspace{-0.7em}\tiny Continued from previous page.}\\
\addlinespace[1.7ex]
\midrule
\multicolumn{1}{c}{$X$}&\multicolumn{1}{c}{$(-K_X)^4$}&\multicolumn{1}{c}{$\rho_X$}&\multicolumn{1}{c}{$\alpha_{0}$}&\multicolumn{1}{c}{$\alpha_{4}$}&\multicolumn{1}{c}{$\alpha_{8}$}&\multicolumn{1}{c}{$\alpha_{12}$}&\multicolumn{1}{c}{$\alpha_{16}$}\\
\midrule
\endhead
\midrule
\multicolumn{8}{r}{\raisebox{0.2em}{\tiny Continued on next page.}}\\
\endfoot
\bottomrule
\endlastfoot
\oddrow \hyperref[sec:index_4]{$Q^4$}&$512$&$1$&$1$&$48$&$15120$&$7392000$&$4414410000$\\
\addlinespace[1.1ex]
\end{longtable}

\begin{longtable}{cccccccc}
\caption{Four-dimensional Fano manifolds with Fano index $r=3$}
\label{table:index_3}\\
\toprule
\multicolumn{1}{c}{$X$}&\multicolumn{1}{c}{$(-K_X)^4$}&\multicolumn{1}{c}{$\rho_X$}&\multicolumn{1}{c}{$\alpha_{0}$}&\multicolumn{1}{c}{$\alpha_{3}$}&\multicolumn{1}{c}{$\alpha_{6}$}&\multicolumn{1}{c}{$\alpha_{9}$}&\multicolumn{1}{c}{$\alpha_{12}$}\\
\midrule
\endfirsthead
\multicolumn{8}{l}{\vspace{-0.7em}\tiny Continued from previous page.}\\
\addlinespace[1.7ex]
\midrule
\multicolumn{1}{c}{$X$}&\multicolumn{1}{c}{$(-K_X)^4$}&\multicolumn{1}{c}{$\rho_X$}&\multicolumn{1}{c}{$\alpha_{0}$}&\multicolumn{1}{c}{$\alpha_{3}$}&\multicolumn{1}{c}{$\alpha_{6}$}&\multicolumn{1}{c}{$\alpha_{9}$}&\multicolumn{1}{c}{$\alpha_{12}$}\\
\midrule
\endhead
\midrule
\multicolumn{8}{r}{\raisebox{0.2em}{\tiny Continued on next page.}}\\
\endfoot
\bottomrule
\endlastfoot
\oddrow \hyperref[sec:index_3]{$\PP^2\times \PP^2$}&$486$&$2$&$1$&$12$&$900$&$94080$&$11988900$\\
\evnrow \hyperref[sec:index_3]{$\FI{4}{5}$}&$405$&$1$&$1$&$18$&$1710$&$246960$&$43347150$\\
\oddrow \hyperref[sec:index_3]{$\FI{4}{4}$}&$324$&$1$&$1$&$24$&$3240$&$672000$&$169785000$\\
\evnrow \hyperref[sec:index_3]{$\FI{4}{3}$}&$243$&$1$&$1$&$36$&$8100$&$2822400$&$1200622500$\\
\oddrow \hyperref[sec:index_3]{$\FI{4}{2}$}&$162$&$1$&$1$&$72$&$37800$&$31046400$&$31216185000$\\
\evnrow \hyperref[sec:index_3]{$\FI{4}{1}$}&$81$&$1$&$1$&$360$&$1247400$&$6861254400$&$46381007673000$\\
\addlinespace[1.1ex]
\end{longtable}

\pagebreak 

\begin{longtable}{cccccccc}
\caption{Four-dimensional Fano manifolds with Fano index $r=2$}
\label{table:index_2}\\
\toprule
\multicolumn{1}{c}{$X$}&\multicolumn{1}{c}{$(-K_X)^4$}&\multicolumn{1}{c}{$\rho_X$}&\multicolumn{1}{c}{$\alpha_{0}$}&\multicolumn{1}{c}{$\alpha_{2}$}&\multicolumn{1}{c}{$\alpha_{4}$}&\multicolumn{1}{c}{$\alpha_{6}$}&\multicolumn{1}{c}{$\alpha_{8}$}\\
\midrule
\endfirsthead
\multicolumn{8}{l}{\vspace{-0.7em}\tiny Continued from previous page.}\\
\addlinespace[1.7ex]
\midrule
\multicolumn{1}{c}{$X$}&\multicolumn{1}{c}{$(-K_X)^4$}&\multicolumn{1}{c}{$\rho_X$}&\multicolumn{1}{c}{$\alpha_{0}$}&\multicolumn{1}{c}{$\alpha_{2}$}&\multicolumn{1}{c}{$\alpha_{4}$}&\multicolumn{1}{c}{$\alpha_{6}$}&\multicolumn{1}{c}{$\alpha_{8}$}\\
\midrule
\endhead
\midrule
\multicolumn{8}{r}{\raisebox{0.2em}{\tiny Continued on next page.}}\\
\endfoot
\bottomrule
\endlastfoot
\oddrow \hyperref[sec:MW^4_15]{$\MW{4}{15}$}&$640$&$2$&$1$&$2$&$6$&$380$&$6790$\\
\evnrow \hyperref[sec:MW^4_14]{$\PP^1\times \PP^3$}&$512$&$2$&$1$&$2$&$30$&$740$&$12670$\\
\oddrow \hyperref[sec:MW^4_13]{$\MW{4}{13}$}&$480$&$2$&$1$&$2$&$54$&$740$&$21910$\\
\evnrow \hyperref[sec:MW^4_12]{$\MW{4}{12}$}&$416$&$2$&$1$&$2$&$54$&$1100$&$28630$\\
\oddrow \hyperref[sec:MW^4_17]{$\PP^1\times \MM{2}{35}$}&$448$&$3$&$1$&$4$&$60$&$1480$&$41020$\\
\evnrow \hyperref[sec:MW^4_11]{$\MW{4}{11}$}&$384$&$2$&$1$&$4$&$84$&$2200$&$70420$\\
\oddrow \hyperref[sec:MW^4_10]{$\MW{4}{10}$}&$352$&$2$&$1$&$4$&$84$&$2560$&$87220$\\
\evnrow \hyperref[sec:MW^4_7]{$\MW{4}{7}$}&$320$&$2$&$1$&$4$&$108$&$3280$&$126700$\\
\oddrow \hyperref[sec:MW^4_16]{$\PP^1\times \MM{2}{32}$}&$384$&$3$&$1$&$6$&$114$&$3300$&$114450$\\
\evnrow \hyperref[sec:MW^4_8]{$\MW{4}{8}$}&$320$&$2$&$1$&$6$&$138$&$4740$&$194250$\\
\oddrow \hyperref[sec:V^4_18]{$\VV{4}{18}$}&$288$&$1$&$1$&$6$&$162$&$6180$&$284130$\\
\evnrow \hyperref[sec:MW^4_5]{$\MW{4}{5}$}&$256$&$2$&$1$&$6$&$186$&$7980$&$410970$\\
\oddrow \hyperref[sec:MW^4_18]{$\PP^1\times \PP^1\times \PP^1\times \PP^1$}&$384$&$4$&$1$&$8$&$168$&$5120$&$190120$\\
\evnrow \hyperref[sec:MW^4_9]{$\PP^1\times \BB{3}{5}$}&$320$&$2$&$1$&$8$&$192$&$6920$&$303520$\\
\oddrow \hyperref[sec:V^4_16]{$\VV{4}{16}$}&$256$&$1$&$1$&$8$&$240$&$10880$&$597520$\\
\evnrow \hyperref[sec:V^4_14]{$\VV{4}{14}$}&$224$&$1$&$1$&$8$&$288$&$15200$&$968800$\\
\oddrow \hyperref[sec:MW^4_4]{$\MW{4}{4}$}&$192$&$2$&$1$&$8$&$360$&$22400$&$1695400$\\
\evnrow \hyperref[sec:MW^4_6]{$\PP^1\times \BB{3}{4}$}&$256$&$2$&$1$&$10$&$318$&$15220$&$886270$\\
\oddrow \hyperref[sec:V^4_12]{$\VV{4}{12}$}&$192$&$1$&$1$&$10$&$438$&$28900$&$2310070$\\
\evnrow \hyperref[sec:V^4_10]{$\VV{4}{10}$}&$160$&$1$&$1$&$12$&$684$&$58800$&$6129900$\\
\oddrow \hyperref[sec:MW^4_3]{$\PP^1\times \BB{3}{3}$}&$192$&$2$&$1$&$14$&$690$&$50900$&$4540690$\\
\evnrow \hyperref[sec:V^4_8]{$\VV{4}{8}$}&$128$&$1$&$1$&$16$&$1296$&$160000$&$24010000$\\
\oddrow \hyperref[sec:V^4_6]{$\VV{4}{6}$}&$96$&$1$&$1$&$24$&$3240$&$672000$&$169785000$\\
\evnrow \hyperref[sec:MW^4_2]{$\PP^1\times \BB{3}{2}$}&$128$&$2$&$1$&$26$&$2814$&$447380$&$84832510$\\
\oddrow \hyperref[sec:V^4_4]{$\VV{4}{4}$}&$64$&$1$&$1$&$48$&$15120$&$7392000$&$4414410000$\\
\evnrow \hyperref[sec:MW^4_1]{$\PP^1\times \BB{3}{1}$}&$64$&$2$&$1$&$122$&$84606$&$84187220$&$98308169470$\\
\oddrow \hyperref[sec:V^4_2]{$\VV{4}{2}$}&$32$&$1$&$1$&$240$&$498960$&$1633632000$&$6558930378000$\\
\addlinespace[1.1ex]
\end{longtable}

\begin{landscape}
  \LTcapwidth=\textwidth
  \begin{longtable}{ccccccccccc}
\caption{Certain four-dimensional Fano manifolds with Fano index $r=1$}
\label{table:index_1}\\
\toprule
\multicolumn{1}{c}{$X$}&\multicolumn{1}{c}{$(-K_X)^4$}&\multicolumn{1}{c}{$\rho_X$}&\multicolumn{1}{c}{$\alpha_{0}$}&\multicolumn{1}{c}{$\alpha_{1}$}&\multicolumn{1}{c}{$\alpha_{2}$}&\multicolumn{1}{c}{$\alpha_{3}$}&\multicolumn{1}{c}{$\alpha_{4}$}&\multicolumn{1}{c}{$\alpha_{5}$}&\multicolumn{1}{c}{$\alpha_{6}$}&\multicolumn{1}{c}{$\alpha_{7}$}\\
\midrule
\endfirsthead
\multicolumn{11}{l}{\vspace{-0.7em}\tiny Continued from previous page.}\\
\addlinespace[1.7ex]
\midrule
\multicolumn{1}{c}{$X$}&\multicolumn{1}{c}{$(-K_X)^4$}&\multicolumn{1}{c}{$\rho_X$}&\multicolumn{1}{c}{$\alpha_{0}$}&\multicolumn{1}{c}{$\alpha_{1}$}&\multicolumn{1}{c}{$\alpha_{2}$}&\multicolumn{1}{c}{$\alpha_{3}$}&\multicolumn{1}{c}{$\alpha_{4}$}&\multicolumn{1}{c}{$\alpha_{5}$}&\multicolumn{1}{c}{$\alpha_{6}$}&\multicolumn{1}{c}{$\alpha_{7}$}\\
\midrule
\endhead
\midrule
\multicolumn{11}{r}{\raisebox{0.2em}{\tiny Continued on next page.}}\\
\endfoot
\bottomrule
\endlastfoot
\oddrow $\Obro{4}{115}$&$512$&$2$&$1$&$0$&$0$&$0$&$24$&$120$&$0$&$0$\\
\evnrow $\Obro{4}{21}$&$594$&$2$&$1$&$0$&$0$&$6$&$0$&$0$&$90$&$1260$\\
\oddrow $\Obro{4}{118}$&$513$&$2$&$1$&$0$&$0$&$6$&$0$&$120$&$90$&$0$\\
\evnrow $\Obro{4}{17}$&$450$&$3$&$1$&$0$&$0$&$6$&$0$&$120$&$90$&$1260$\\
\oddrow $\Obro{4}{47}$&$513$&$2$&$1$&$0$&$0$&$6$&$24$&$0$&$90$&$2520$\\
\evnrow $\Obro{4}{94}$&$459$&$3$&$1$&$0$&$0$&$6$&$24$&$120$&$90$&$1260$\\
\oddrow $\Obro{4}{37}$&$417$&$3$&$1$&$0$&$0$&$6$&$24$&$120$&$90$&$2520$\\
\evnrow $\Obro{4}{74}$&$486$&$3$&$1$&$0$&$0$&$6$&$48$&$0$&$90$&$2520$\\
\oddrow $\Obro{4}{86}$&$405$&$3$&$1$&$0$&$0$&$6$&$48$&$120$&$90$&$2520$\\
\evnrow $\Obro{4}{114}$&$401$&$3$&$1$&$0$&$0$&$12$&$0$&$120$&$900$&$0$\\
\oddrow $\Obro{4}{46}$&$406$&$3$&$1$&$0$&$0$&$12$&$24$&$0$&$900$&$3780$\\
\evnrow $\Obro{4}{87}$&$364$&$4$&$1$&$0$&$0$&$12$&$24$&$120$&$900$&$3780$\\
\oddrow $\Obro{4}{32}$&$322$&$4$&$1$&$0$&$0$&$12$&$24$&$240$&$900$&$5040$\\
\evnrow $\Obro{4}{30}$&$327$&$4$&$1$&$0$&$0$&$12$&$48$&$120$&$900$&$7560$\\
\oddrow $\Obro{4}{31}$&$249$&$5$&$1$&$0$&$0$&$18$&$72$&$360$&$2430$&$18900$\\
\evnrow $\Str_1$&$225$&$2$&$1$&$0$&$0$&$30$&$120$&$240$&$5850$&$50400$\\
\oddrow $\Obro{4}{2}$&$800$&$2$&$1$&$0$&$2$&$0$&$6$&$0$&$20$&$840$\\
\evnrow $\Obro{4}{1}$&$605$&$3$&$1$&$0$&$2$&$0$&$6$&$0$&$380$&$840$\\
\oddrow $\Obro{4}{12}$&$560$&$3$&$1$&$0$&$2$&$0$&$6$&$60$&$380$&$840$\\
\evnrow $\Obro{4}{121}$&$544$&$2$&$1$&$0$&$2$&$0$&$6$&$120$&$20$&$2520$\\
\oddrow $\Obro{4}{105}$&$489$&$3$&$1$&$0$&$2$&$0$&$6$&$120$&$380$&$2520$\\
\evnrow $\Obro{4}{18}$&$529$&$3$&$1$&$0$&$2$&$0$&$30$&$60$&$380$&$840$\\
\oddrow $\Obro{4}{10}$&$496$&$4$&$1$&$0$&$2$&$0$&$30$&$60$&$740$&$840$\\
\evnrow $\Obro{4}{109}$&$464$&$3$&$1$&$0$&$2$&$0$&$30$&$120$&$380$&$2520$\\
\oddrow $\Obro{4}{104}$&$431$&$3$&$1$&$0$&$2$&$0$&$30$&$120$&$740$&$2520$\\
\evnrow $\Obro{4}{15}$&$433$&$4$&$1$&$0$&$2$&$0$&$30$&$180$&$380$&$3360$\\
\oddrow $\Obro{4}{11}$&$415$&$4$&$1$&$0$&$2$&$0$&$30$&$180$&$740$&$3360$\\
\evnrow $\Obro{4}{8}$&$576$&$3$&$1$&$0$&$2$&$6$&$6$&$60$&$110$&$1680$\\
\oddrow $\Obro{4}{26}$&$560$&$3$&$1$&$0$&$2$&$6$&$6$&$60$&$470$&$420$\\
\evnrow $\Obro{4}{7}$&$592$&$3$&$1$&$0$&$2$&$6$&$6$&$120$&$110$&$1260$\\
\oddrow $\Obro{4}{20}$&$400$&$3$&$1$&$0$&$2$&$6$&$6$&$120$&$830$&$2520$\\
\evnrow $\Obro{4}{111}$&$480$&$3$&$1$&$0$&$2$&$6$&$6$&$180$&$110$&$2940$\\
\oddrow $\Obro{4}{24}$&$442$&$4$&$1$&$0$&$2$&$6$&$6$&$180$&$470$&$2940$\\
\evnrow $\Obro{4}{106}$&$496$&$3$&$1$&$0$&$2$&$6$&$30$&$60$&$470$&$2940$\\
\oddrow $\Obro{4}{45}$&$432$&$3$&$1$&$0$&$2$&$6$&$30$&$60$&$830$&$2940$\\
\evnrow $\Obro{4}{41}$&$433$&$3$&$1$&$0$&$2$&$6$&$30$&$120$&$470$&$3780$\\
\oddrow $\Obro{4}{6}$&$463$&$4$&$1$&$0$&$2$&$6$&$30$&$120$&$470$&$3780$\\
\evnrow $\PP^1\times \MM{2}{33}$&$432$&$3$&$1$&$0$&$2$&$6$&$30$&$120$&$830$&$2520$\\
\oddrow $\Obro{4}{82}$&$432$&$4$&$1$&$0$&$2$&$6$&$30$&$180$&$470$&$4200$\\
\evnrow $\Obro{4}{113}$&$400$&$3$&$1$&$0$&$2$&$6$&$30$&$180$&$470$&$5460$\\
\oddrow $\Obro{4}{92}$&$384$&$4$&$1$&$0$&$2$&$6$&$30$&$180$&$830$&$5460$\\
\evnrow $\Obro{4}{70}$&$411$&$4$&$1$&$0$&$2$&$6$&$30$&$240$&$470$&$5040$\\
\oddrow $\Obro{4}{16}$&$337$&$4$&$1$&$0$&$2$&$6$&$30$&$240$&$1190$&$7560$\\
\evnrow $\Obro{4}{52}$&$464$&$4$&$1$&$0$&$2$&$6$&$54$&$60$&$830$&$2940$\\
\oddrow $\Obro{4}{71}$&$390$&$4$&$1$&$0$&$2$&$6$&$54$&$120$&$1190$&$3780$\\
\evnrow $\Obro{4}{91}$&$384$&$4$&$1$&$0$&$2$&$6$&$54$&$180$&$830$&$5460$\\
\oddrow $\Obro{4}{13}$&$368$&$4$&$1$&$0$&$2$&$6$&$54$&$180$&$830$&$5880$\\
\evnrow $\Obro{4}{81}$&$357$&$4$&$1$&$0$&$2$&$6$&$54$&$240$&$1190$&$6300$\\
\oddrow $\PP^2\times F_1$&$432$&$3$&$1$&$0$&$2$&$12$&$6$&$180$&$920$&$1680$\\
\evnrow $\PP^1\times Q^3$&$432$&$2$&$1$&$0$&$2$&$12$&$6$&$240$&$560$&$2520$\\
\oddrow $\Obro{4}{27}$&$417$&$4$&$1$&$0$&$2$&$12$&$6$&$240$&$560$&$3360$\\
\evnrow $\Obro{4}{60}$&$448$&$4$&$1$&$0$&$2$&$12$&$30$&$120$&$920$&$4620$\\
\oddrow $\Obro{4}{88}$&$389$&$4$&$1$&$0$&$2$&$12$&$30$&$180$&$1280$&$5460$\\
\evnrow $\Obro{4}{35}$&$369$&$4$&$1$&$0$&$2$&$12$&$30$&$180$&$1280$&$5460$\\
\oddrow $\PP^1\times \MM{2}{30}$&$368$&$3$&$1$&$0$&$2$&$12$&$30$&$240$&$1280$&$5040$\\
\evnrow $\Obro{4}{93}$&$347$&$4$&$1$&$0$&$2$&$12$&$30$&$300$&$1280$&$7980$\\
\oddrow $\Obro{4}{85}$&$352$&$4$&$1$&$0$&$2$&$12$&$54$&$240$&$1280$&$9660$\\
\evnrow $\Obro{4}{42}$&$326$&$4$&$1$&$0$&$2$&$12$&$54$&$240$&$1640$&$10080$\\
\oddrow $\Obro{4}{51}$&$480$&$4$&$1$&$0$&$2$&$18$&$6$&$180$&$1370$&$1260$\\
\evnrow $\PP^1\times \MM{2}{28}$&$320$&$3$&$1$&$0$&$2$&$18$&$30$&$360$&$2090$&$7560$\\
\oddrow $\Obro{4}{73}$&$352$&$4$&$1$&$0$&$2$&$18$&$54$&$180$&$2090$&$11340$\\
\evnrow $\Str_2$&$240$&$2$&$1$&$0$&$2$&$30$&$54$&$600$&$6590$&$26040$\\
\oddrow $\PP^1\times \MM{2}{36}$&$496$&$3$&$1$&$0$&$4$&$0$&$36$&$60$&$400$&$3360$\\
\evnrow $\Obro{4}{43}$&$464$&$3$&$1$&$0$&$4$&$0$&$36$&$120$&$400$&$5040$\\
\oddrow $\PP^1\times \MM{3}{29}$&$400$&$4$&$1$&$0$&$4$&$0$&$60$&$60$&$1480$&$3360$\\
\evnrow $\Obro{4}{36}$&$384$&$4$&$1$&$0$&$4$&$0$&$60$&$120$&$1480$&$5040$\\
\oddrow $\Obro{4}{3}$&$558$&$4$&$1$&$0$&$4$&$6$&$36$&$120$&$490$&$3360$\\
\evnrow $\Obro{4}{22}$&$505$&$4$&$1$&$0$&$4$&$6$&$36$&$120$&$850$&$2100$\\
\oddrow $\Obro{4}{5}$&$478$&$4$&$1$&$0$&$4$&$6$&$36$&$180$&$490$&$5460$\\
\evnrow $\Obro{4}{9}$&$382$&$4$&$1$&$0$&$4$&$6$&$36$&$180$&$1210$&$6720$\\
\oddrow $\Obro{4}{95}$&$447$&$4$&$1$&$0$&$4$&$6$&$36$&$240$&$490$&$7140$\\
\evnrow $\PP^2\times \PP^1\times \PP^1$&$432$&$3$&$1$&$0$&$4$&$6$&$36$&$240$&$490$&$7560$\\
\oddrow $\Obro{4}{25}$&$409$&$4$&$1$&$0$&$4$&$6$&$36$&$240$&$850$&$7140$\\
\evnrow $\Obro{4}{100}$&$415$&$4$&$1$&$0$&$4$&$6$&$60$&$120$&$1570$&$4620$\\
\oddrow $\PP^1\times \MM{3}{30}$&$400$&$4$&$1$&$0$&$4$&$6$&$60$&$180$&$1570$&$5460$\\
\evnrow $\Obro{4}{34}$&$369$&$4$&$1$&$0$&$4$&$6$&$60$&$180$&$1570$&$6720$\\
\oddrow $\Obro{4}{56}$&$405$&$5$&$1$&$0$&$4$&$6$&$60$&$240$&$1210$&$8400$\\
\evnrow $\PP^1\times \MM{3}{26}$&$368$&$4$&$1$&$0$&$4$&$6$&$60$&$240$&$1570$&$8820$\\
\oddrow $\Obro{4}{102}$&$367$&$4$&$1$&$0$&$4$&$6$&$60$&$240$&$1570$&$9660$\\
\evnrow $\Obro{4}{44}$&$351$&$4$&$1$&$0$&$4$&$6$&$60$&$240$&$1930$&$9660$\\
\oddrow $\Obro{4}{48}$&$442$&$5$&$1$&$0$&$4$&$6$&$84$&$120$&$1930$&$4620$\\
\evnrow $\PP^1\times \MM{3}{22}$&$320$&$4$&$1$&$0$&$4$&$6$&$84$&$300$&$2650$&$13440$\\
\oddrow $\Obro{4}{29}$&$310$&$5$&$1$&$0$&$4$&$6$&$84$&$360$&$2650$&$15120$\\
\evnrow $\PP^1\times \MM{3}{31}$&$416$&$4$&$1$&$0$&$4$&$12$&$36$&$360$&$940$&$8400$\\
\oddrow $F_1\times F_1$&$384$&$4$&$1$&$0$&$4$&$12$&$36$&$360$&$1300$&$8400$\\
\evnrow $\SS{2}{7}\times \PP^2$&$378$&$4$&$1$&$0$&$4$&$12$&$36$&$360$&$1300$&$9660$\\
\oddrow $\PP^1\times \MM{2}{31}$&$368$&$3$&$1$&$0$&$4$&$12$&$36$&$420$&$940$&$11760$\\
\evnrow $\Obro{4}{54}$&$405$&$5$&$1$&$0$&$4$&$12$&$60$&$300$&$1660$&$10080$\\
\oddrow $\Obro{4}{58}$&$373$&$5$&$1$&$0$&$4$&$12$&$60$&$300$&$2020$&$10080$\\
\evnrow $\PP^1\times \MM{3}{25}$&$352$&$4$&$1$&$0$&$4$&$12$&$60$&$360$&$2020$&$10920$\\
\oddrow $\Obro{4}{66}$&$332$&$5$&$1$&$0$&$4$&$12$&$60$&$360$&$2020$&$13440$\\
\evnrow $\PP^1\times \MM{3}{23}$&$336$&$4$&$1$&$0$&$4$&$12$&$60$&$420$&$2020$&$14280$\\
\oddrow $\Obro{4}{28}$&$321$&$5$&$1$&$0$&$4$&$12$&$60$&$420$&$2020$&$16380$\\
\evnrow $\Obro{4}{65}$&$331$&$5$&$1$&$0$&$4$&$12$&$84$&$420$&$2380$&$17640$\\
\oddrow $\Obro{4}{80}$&$325$&$5$&$1$&$0$&$4$&$12$&$84$&$420$&$2740$&$17640$\\
\evnrow $\PP^1\times \MM{3}{19}$&$304$&$4$&$1$&$0$&$4$&$12$&$84$&$480$&$3100$&$20160$\\
\oddrow $\Obro{4}{50}$&$394$&$5$&$1$&$0$&$4$&$18$&$36$&$480$&$1750$&$10500$\\
\evnrow $\Obro{4}{68}$&$363$&$5$&$1$&$0$&$4$&$18$&$36$&$480$&$2110$&$10500$\\
\oddrow $\Obro{4}{59}$&$341$&$5$&$1$&$0$&$4$&$18$&$60$&$480$&$2830$&$15540$\\
\evnrow $\PP^1\times \MM{2}{27}$&$304$&$3$&$1$&$0$&$4$&$18$&$60$&$600$&$2830$&$19740$\\
\oddrow $\Obro{4}{53}$&$330$&$5$&$1$&$0$&$4$&$18$&$84$&$480$&$3190$&$20580$\\
\evnrow $\Obro{4}{69}$&$310$&$5$&$1$&$0$&$4$&$18$&$84$&$480$&$3550$&$20580$\\
\oddrow $\Obro{4}{84}$&$299$&$5$&$1$&$0$&$4$&$18$&$84$&$600$&$3550$&$25620$\\
\evnrow $\PP^1\times \MM{3}{14}$&$256$&$4$&$1$&$0$&$4$&$18$&$132$&$780$&$6070$&$42420$\\
\oddrow $\PP^1\times \MM{3}{9}$&$208$&$4$&$1$&$0$&$4$&$36$&$228$&$1560$&$15340$&$122640$\\
\evnrow $\Obro{4}{38}$&$385$&$4$&$1$&$0$&$6$&$0$&$90$&$120$&$1860$&$7560$\\
\oddrow $\PP^1\times \PP^1\times F_1$&$384$&$4$&$1$&$0$&$6$&$6$&$90$&$300$&$1950$&$13020$\\
\evnrow $\PP^1\times \MM{4}{13}$&$368$&$5$&$1$&$0$&$6$&$6$&$114$&$240$&$3390$&$9660$\\
\oddrow $\PP^1\times \MM{3}{24}$&$336$&$4$&$1$&$0$&$6$&$6$&$114$&$300$&$3390$&$14280$\\
\evnrow $\Obro{4}{33}$&$305$&$5$&$1$&$0$&$6$&$6$&$114$&$360$&$3750$&$18480$\\
\oddrow $\Obro{4}{4}$&$364$&$5$&$1$&$0$&$6$&$12$&$90$&$420$&$2760$&$17220$\\
\evnrow $\Obro{4}{23}$&$354$&$5$&$1$&$0$&$6$&$12$&$90$&$480$&$2760$&$20160$\\
\oddrow $\PP^1\times \MM{4}{12}$&$352$&$5$&$1$&$0$&$6$&$12$&$90$&$540$&$2400$&$21420$\\
\evnrow $\SS{2}{7}\times F_1$&$336$&$5$&$1$&$0$&$6$&$12$&$90$&$540$&$2760$&$21420$\\
\oddrow $\PP^1\times \MM{2}{29}$&$320$&$3$&$1$&$0$&$6$&$12$&$90$&$600$&$2400$&$26040$\\
\evnrow $\Obro{4}{96}$&$334$&$5$&$1$&$0$&$6$&$12$&$114$&$480$&$3840$&$22680$\\
\oddrow $\PP^1\times \MM{4}{10}$&$320$&$5$&$1$&$0$&$6$&$12$&$114$&$540$&$3840$&$23940$\\
\evnrow $\PP^1\times \MM{3}{20}$&$304$&$4$&$1$&$0$&$6$&$12$&$114$&$600$&$3840$&$28560$\\
\oddrow $\PP^1\times \MM{3}{17}$&$288$&$4$&$1$&$0$&$6$&$12$&$138$&$600$&$5280$&$31080$\\
\evnrow $\SS{2}{6}\times \PP^2$&$324$&$5$&$1$&$0$&$6$&$18$&$90$&$720$&$3570$&$28980$\\
\oddrow $\PP^1\times \MM{3}{18}$&$288$&$4$&$1$&$0$&$6$&$18$&$114$&$840$&$4650$&$38220$\\
\evnrow $\Obro{4}{57}$&$298$&$6$&$1$&$0$&$6$&$18$&$138$&$780$&$5730$&$39480$\\
\oddrow $\PP^1\times \MM{3}{16}$&$272$&$4$&$1$&$0$&$6$&$18$&$138$&$900$&$6090$&$46620$\\
\evnrow $\PP^1\times \MM{2}{25}$&$256$&$3$&$1$&$0$&$6$&$24$&$114$&$1200$&$5820$&$57120$\\
\oddrow $\Obro{4}{49}$&$308$&$6$&$1$&$0$&$6$&$24$&$138$&$960$&$6180$&$46200$\\
\evnrow $\Obro{4}{55}$&$298$&$6$&$1$&$0$&$6$&$24$&$138$&$960$&$6540$&$46200$\\
\oddrow $\Obro{4}{63}$&$278$&$6$&$1$&$0$&$6$&$24$&$138$&$1080$&$6540$&$53760$\\
\evnrow $\Obro{4}{64}$&$268$&$6$&$1$&$0$&$6$&$24$&$162$&$960$&$7980$&$53760$\\
\oddrow $\PP^1\times \MM{2}{24}$&$240$&$3$&$1$&$0$&$6$&$24$&$186$&$1260$&$10140$&$78120$\\
\evnrow $\Obro{4}{39}$&$307$&$5$&$1$&$0$&$8$&$0$&$168$&$120$&$5120$&$10080$\\
\oddrow $\SS{2}{7}\times \PP^1\times \PP^1$&$336$&$5$&$1$&$0$&$8$&$6$&$168$&$360$&$5210$&$19740$\\
\evnrow $\PP^1\times \MM{3}{21}$&$304$&$4$&$1$&$0$&$8$&$6$&$192$&$360$&$7010$&$21000$\\
\oddrow $\PP^1\times \MM{4}{9}$&$304$&$5$&$1$&$0$&$8$&$12$&$168$&$720$&$5660$&$39480$\\
\evnrow $\SS{2}{7}\times \SS{2}{7}$&$294$&$6$&$1$&$0$&$8$&$12$&$168$&$720$&$6020$&$39480$\\
\oddrow $\PP^1\times \MM{4}{8}$&$288$&$5$&$1$&$0$&$8$&$12$&$192$&$720$&$7460$&$42000$\\
\evnrow $\PP^1\times \MM{2}{26}$&$272$&$3$&$1$&$0$&$8$&$12$&$192$&$780$&$7460$&$47880$\\
\oddrow $\SS{2}{6}\times F_1$&$288$&$6$&$1$&$0$&$8$&$18$&$168$&$1020$&$6830$&$54600$\\
\evnrow $\PP^1\times \MM{5}{2}$&$288$&$6$&$1$&$0$&$8$&$18$&$192$&$1020$&$7910$&$57120$\\
\oddrow $\PP^1\times \MM{4}{7}$&$272$&$5$&$1$&$0$&$8$&$18$&$192$&$1080$&$8270$&$63000$\\
\evnrow $\PP^1\times \MM{3}{15}$&$256$&$4$&$1$&$0$&$8$&$18$&$216$&$1140$&$10070$&$72660$\\
\oddrow $\PP^1\times \MM{4}{5}$&$256$&$5$&$1$&$0$&$8$&$24$&$216$&$1440$&$10880$&$89040$\\
\evnrow $\PP^1\times \MM{2}{22}$&$240$&$3$&$1$&$0$&$8$&$24$&$216$&$1560$&$11240$&$100800$\\
\oddrow $\PP^1\times \MM{3}{13}$&$240$&$4$&$1$&$0$&$8$&$24$&$240$&$1560$&$13040$&$105840$\\
\evnrow $\PP^1\times \MM{3}{11}$&$224$&$4$&$1$&$0$&$8$&$30$&$264$&$1980$&$16370$&$142800$\\
\oddrow $\PP^1\times \MM{2}{18}$&$192$&$3$&$1$&$0$&$8$&$48$&$360$&$3360$&$31040$&$295680$\\
\evnrow $\Obro{4}{40}$&$230$&$6$&$1$&$0$&$10$&$0$&$270$&$240$&$10900$&$25200$\\
\oddrow $\SS{2}{6}\times \PP^1\times \PP^1$&$288$&$6$&$1$&$0$&$10$&$12$&$270$&$840$&$11080$&$55440$\\
\evnrow $\PP^1\times \MM{2}{23}$&$240$&$3$&$1$&$0$&$10$&$12$&$318$&$960$&$15760$&$74760$\\
\oddrow $\SS{2}{6}\times \SS{2}{7}$&$252$&$7$&$1$&$0$&$10$&$18$&$270$&$1320$&$12610$&$91560$\\
\evnrow $\PP^1\times \MM{4}{6}$&$256$&$5$&$1$&$0$&$10$&$18$&$294$&$1320$&$14050$&$94080$\\
\oddrow $\PP^1\times \MM{4}{4}$&$240$&$5$&$1$&$0$&$10$&$24$&$318$&$1800$&$17380$&$135240$\\
\evnrow $\PP^1\times \MM{2}{21}$&$224$&$3$&$1$&$0$&$10$&$24$&$342$&$1920$&$19900$&$154560$\\
\oddrow $\PP^1\times \MM{3}{12}$&$224$&$4$&$1$&$0$&$10$&$30$&$342$&$2340$&$21070$&$186060$\\
\evnrow $\PP^1\times \MM{2}{19}$&$208$&$3$&$1$&$0$&$10$&$30$&$342$&$2520$&$21430$&$208740$\\
\oddrow $\PP^2\times \SS{2}{5}$&$270$&$6$&$1$&$0$&$10$&$36$&$270$&$2160$&$15040$&$134400$\\
\evnrow $\PP^1\times \MM{2}{20}$&$208$&$3$&$1$&$0$&$10$&$36$&$390$&$2940$&$27640$&$255360$\\
\oddrow $\SS{2}{6}\times \SS{2}{6}$&$216$&$8$&$1$&$0$&$12$&$24$&$396$&$2160$&$23160$&$186480$\\
\evnrow $\PP^1\times \MM{4}{3}$&$224$&$5$&$1$&$0$&$12$&$24$&$444$&$2160$&$26760$&$191520$\\
\oddrow $F_1\times \SS{2}{5}$&$240$&$7$&$1$&$0$&$12$&$36$&$396$&$2820$&$24060$&$219240$\\
\evnrow $\PP^1\times \MM{3}{10}$&$208$&$4$&$1$&$0$&$12$&$36$&$492$&$3360$&$35220$&$319200$\\
\oddrow $\PP^1\times \MM{5}{1}$&$224$&$6$&$1$&$0$&$12$&$42$&$468$&$3480$&$32430$&$300300$\\
\evnrow $\PP^1\times \MM{2}{17}$&$192$&$3$&$1$&$0$&$12$&$42$&$540$&$4140$&$43230$&$423360$\\
\oddrow $\PP^1\times \MM{3}{7}$&$192$&$4$&$1$&$0$&$12$&$48$&$564$&$4680$&$48000$&$486360$\\
\evnrow $\PP^1\times \MM{2}{16}$&$176$&$3$&$1$&$0$&$12$&$60$&$636$&$6120$&$63300$&$693000$\\
\oddrow $\PP^1\times \PP^1\times \SS{2}{5}$&$240$&$7$&$1$&$0$&$14$&$30$&$546$&$2760$&$33350$&$246540$\\
\evnrow $\SS{2}{7}\times \SS{2}{5}$&$210$&$8$&$1$&$0$&$14$&$36$&$546$&$3480$&$37040$&$330540$\\
\oddrow $\PP^1\times \MM{2}{15}$&$176$&$3$&$1$&$0$&$14$&$36$&$714$&$4320$&$59720$&$519120$\\
\evnrow $\PP^1\times \MM{4}{2}$&$208$&$5$&$1$&$0$&$14$&$42$&$618$&$4200$&$46490$&$425880$\\
\oddrow $\PP^1\times \MM{4}{1}$&$192$&$5$&$1$&$0$&$14$&$48$&$690$&$5280$&$59540$&$594720$\\
\evnrow $\PP^1\times \MM{3}{8}$&$192$&$4$&$1$&$0$&$14$&$54$&$690$&$5700$&$61070$&$631260$\\
\oddrow $\PP^1\times \VV{3}{22}$&$176$&$2$&$1$&$0$&$14$&$60$&$786$&$6960$&$78800$&$859320$\\
\evnrow $\SS{2}{6}\times \SS{2}{5}$&$180$&$9$&$1$&$0$&$16$&$42$&$720$&$4920$&$58390$&$567840$\\
\oddrow $\PP^1\times \MM{3}{6}$&$176$&$4$&$1$&$0$&$16$&$66$&$936$&$8280$&$97630$&$1086540$\\
\evnrow $\PP^1\times \MM{2}{12}$&$160$&$3$&$1$&$0$&$16$&$72$&$1056$&$9840$&$122920$&$1428000$\\
\oddrow $\PP^1\times \MM{2}{13}$&$160$&$3$&$1$&$0$&$16$&$84$&$1104$&$11400$&$137860$&$1685040$\\
\evnrow $\PP^1\times \MM{2}{11}$&$144$&$3$&$1$&$0$&$16$&$108$&$1248$&$15600$&$188260$&$2538480$\\
\oddrow $\PP^1\times \MM{2}{14}$&$160$&$3$&$1$&$0$&$18$&$90$&$1302$&$13260$&$168570$&$2089080$\\
\evnrow $\SS{2}{5}\times \SS{2}{5}$&$150$&$10$&$1$&$0$&$20$&$60$&$1140$&$9120$&$121700$&$1377600$\\
\oddrow $\SS{2}{4}\times \PP^2$&$216$&$7$&$1$&$0$&$20$&$102$&$1188$&$11760$&$123050$&$1391880$\\
\evnrow $\PP^1\times \VV{3}{18}$&$144$&$2$&$1$&$0$&$20$&$120$&$1788$&$20760$&$285680$&$3926160$\\
\oddrow $\Str_3$&$86$&$2$&$1$&$0$&$20$&$156$&$2700$&$41040$&$697700$&$12503400$\\
\evnrow $\SS{2}{4}\times F_1$&$192$&$8$&$1$&$0$&$22$&$102$&$1434$&$13740$&$160510$&$1881180$\\
\oddrow $\PP^1\times \MM{3}{3}$&$144$&$4$&$1$&$0$&$22$&$132$&$2058$&$24360$&$345280$&$4867800$\\
\evnrow $\SS{2}{4}\times \PP^1\times \PP^1$&$192$&$8$&$1$&$0$&$24$&$96$&$1704$&$14400$&$193920$&$2150400$\\
\oddrow $\SS{2}{4}\times \SS{2}{7}$&$168$&$9$&$1$&$0$&$24$&$102$&$1704$&$15720$&$205530$&$2452380$\\
\evnrow $\PP^1\times \MM{3}{5}$&$160$&$4$&$1$&$0$&$24$&$126$&$1992$&$21300$&$290130$&$3813600$\\
\oddrow $\PP^1\times \MM{2}{9}$&$128$&$3$&$1$&$0$&$24$&$174$&$2784$&$37680$&$578490$&$9059820$\\
\evnrow $\SS{2}{4}\times \SS{2}{6}$&$144$&$10$&$1$&$0$&$26$&$108$&$1998$&$19080$&$270440$&$3435600$\\
\oddrow $\PP^1\times \MM{3}{4}$&$144$&$4$&$1$&$0$&$26$&$156$&$2574$&$31080$&$457640$&$6657840$\\
\evnrow $\PP^1\times \VV{3}{16}$&$128$&$2$&$1$&$0$&$26$&$192$&$3198$&$44160$&$700820$&$11249280$\\
\oddrow $\PP^1\times \MM{2}{8}$&$112$&$3$&$1$&$0$&$28$&$216$&$3900$&$58800$&$984520$&$17334240$\\
\evnrow $\SS{2}{4}\times \SS{2}{5}$&$120$&$11$&$1$&$0$&$30$&$126$&$2658$&$27720$&$439590$&$6247500$\\
\oddrow $\PP^1\times \MM{2}{10}$&$128$&$3$&$1$&$0$&$30$&$216$&$3858$&$54000$&$891660$&$14726880$\\
\evnrow $\PP^1\times \VV{3}{14}$&$112$&$2$&$1$&$0$&$34$&$312$&$5910$&$97920$&$1820140$&$34520640$\\
\oddrow $\PP^1\times \MM{2}{7}$&$112$&$3$&$1$&$0$&$38$&$348$&$6954$&$117840$&$2268560$&$44336040$\\
\evnrow $\SS{2}{4}\times \SS{2}{4}$&$96$&$12$&$1$&$0$&$40$&$192$&$4776$&$59520$&$1120000$&$19138560$\\
\oddrow $\PP^1\times \MM{2}{6}$&$96$&$3$&$1$&$0$&$46$&$528$&$11826$&$238560$&$5341780$&$122340960$\\
\evnrow $\PP^1\times \VV{3}{12}$&$96$&$2$&$1$&$0$&$50$&$600$&$13758$&$288480$&$6659420$&$157802400$\\
\oddrow $\SS{2}{3}\times \PP^2$&$162$&$8$&$1$&$0$&$54$&$498$&$9882$&$162000$&$2938770$&$54057780$\\
\evnrow $\SS{2}{3}\times F_1$&$144$&$9$&$1$&$0$&$56$&$498$&$10536$&$171900$&$3240110$&$60897480$\\
\oddrow $\PP^1\times \MM{3}{1}$&$96$&$4$&$1$&$0$&$56$&$672$&$16296$&$350400$&$8393600$&$205470720$\\
\evnrow $\SS{2}{3}\times \PP^1\times \PP^1$&$144$&$9$&$1$&$0$&$58$&$492$&$11214$&$178440$&$3502120$&$65938320$\\
\oddrow $\SS{2}{3}\times \SS{2}{7}$&$126$&$10$&$1$&$0$&$58$&$498$&$11214$&$181800$&$3561250$&$68151720$\\
\evnrow $\SS{2}{3}\times \SS{2}{6}$&$108$&$11$&$1$&$0$&$60$&$504$&$11916$&$195120$&$3962040$&$78104880$\\
\oddrow $\PP^1\times \MM{3}{2}$&$112$&$4$&$1$&$0$&$60$&$600$&$13884$&$259440$&$5613000$&$122354400$\\
\evnrow $\SS{2}{3}\times \SS{2}{5}$&$90$&$12$&$1$&$0$&$64$&$522$&$13392$&$225720$&$4887190$&$102194400$\\
\oddrow $\PP^1\times \MM{2}{5}$&$96$&$3$&$1$&$0$&$68$&$816$&$21012$&$465960$&$11662880$&$297392760$\\
\evnrow $\SS{2}{3}\times \SS{2}{4}$&$72$&$13$&$1$&$0$&$74$&$588$&$17550$&$319560$&$7862600$&$185440080$\\
\oddrow $\PP^1\times \VV{3}{10}$&$80$&$2$&$1$&$0$&$80$&$1320$&$38688$&$1078320$&$32604200$&$1016215200$\\
\evnrow $\PP^1\times \MM{2}{4}$&$80$&$3$&$1$&$0$&$92$&$1518$&$47172$&$1357680$&$42774050$&$1385508600$\\
\oddrow $\SS{2}{3}\times \SS{2}{3}$&$54$&$14$&$1$&$0$&$108$&$984$&$37260$&$848880$&$26609400$&$804368880$\\
\evnrow $\PP^1\times \VV{3}{8}$&$64$&$2$&$1$&$0$&$154$&$3840$&$159486$&$6504960$&$284808340$&$12889551360$\\
\oddrow $\SS{2}{2}\times \PP^2$&$108$&$9$&$1$&$0$&$276$&$6822$&$314532$&$12870000$&$570227370$&$25599296520$\\
\evnrow $\SS{2}{2}\times F_1$&$96$&$10$&$1$&$0$&$278$&$6822$&$317850$&$13006380$&$579688190$&$26140920540$\\
\oddrow $\SS{2}{2}\times \PP^1\times \PP^1$&$96$&$10$&$1$&$0$&$280$&$6816$&$321192$&$13126080$&$588430720$&$26621521920$\\
\evnrow $\SS{2}{2}\times \SS{2}{7}$&$84$&$11$&$1$&$0$&$280$&$6822$&$321192$&$13142760$&$589248730$&$26688271260$\\
\oddrow $\SS{2}{2}\times \SS{2}{6}$&$72$&$12$&$1$&$0$&$282$&$6828$&$324558$&$13295880$&$599727720$&$27308448720$\\
\evnrow $\SS{2}{2}\times \SS{2}{5}$&$60$&$13$&$1$&$0$&$286$&$6846$&$331362$&$13619400$&$621807910$&$28636256460$\\
\oddrow $\SS{2}{2}\times \SS{2}{4}$&$48$&$14$&$1$&$0$&$296$&$6912$&$348840$&$14492160$&$681885440$&$32334274560$\\
\evnrow $\PP^1\times \MM{2}{3}$&$64$&$3$&$1$&$0$&$302$&$8472$&$442194$&$21352560$&$1128405740$&$61403700960$\\
\oddrow $\SS{2}{2}\times \SS{2}{3}$&$36$&$15$&$1$&$0$&$330$&$7308$&$413838$&$18050760$&$935040840$&$48854892240$\\
\evnrow $\PP^1\times \VV{3}{6}$&$48$&$2$&$1$&$0$&$398$&$17616$&$1221810$&$85572960$&$6386359700$&$493612489440$\\
\oddrow $\PP^1\times \MM{2}{2}$&$48$&$3$&$1$&$0$&$472$&$21216$&$1568424$&$115141440$&$9050108800$&$736102993920$\\
\evnrow $\SS{2}{2}\times \SS{2}{2}$&$24$&$16$&$1$&$0$&$552$&$13632$&$1086120$&$63331200$&$4672300800$&$350133073920$\\
\oddrow $\PP^1\times \VV{3}{4}$&$32$&$2$&$1$&$0$&$1946$&$215808$&$35318526$&$5981882880$&$1074550170260$&$200205416839680$\\
\evnrow $\SS{2}{1}\times \PP^2$&$54$&$10$&$1$&$0$&$10260$&$2021286$&$618874020$&$184451042160$&$57876574021290$&$18570362883899400$\\
\oddrow $\SS{2}{1}\times F_1$&$48$&$11$&$1$&$0$&$10262$&$2021286$&$618997146$&$184491467820$&$57895141165310$&$18578110239211740$\\
\evnrow $\SS{2}{1}\times \PP^1\times \PP^1$&$48$&$11$&$1$&$0$&$10264$&$2021280$&$619120296$&$184531277760$&$57913469449600$&$18585729302999040$\\
\oddrow $\SS{2}{1}\times \SS{2}{7}$&$42$&$12$&$1$&$0$&$10264$&$2021286$&$619120296$&$184531893480$&$57913712003290$&$18585859292400540$\\
\evnrow $\SS{2}{1}\times \SS{2}{6}$&$36$&$13$&$1$&$0$&$10266$&$2021292$&$619243470$&$184572934920$&$57932529089640$&$18593740045797840$\\
\oddrow $\SS{2}{1}\times \SS{2}{5}$&$30$&$14$&$1$&$0$&$10270$&$2021310$&$619489890$&$184655634120$&$57970416902950$&$18609636764945100$\\
\evnrow $\SS{2}{1}\times \SS{2}{4}$&$24$&$15$&$1$&$0$&$10280$&$2021376$&$620106408$&$184864542720$&$58066057475840$&$18649867837440000$\\
\oddrow $\SS{2}{1}\times \SS{2}{3}$&$18$&$16$&$1$&$0$&$10314$&$2021772$&$622208142$&$185592555720$&$58399032538440$&$18790790073224400$\\
\evnrow $\PP^1\times \MM{2}{1}$&$32$&$3$&$1$&$0$&$10382$&$2082840$&$650724306$&$199392674160$&$64624270834220$&$21530238491351520$\\
\oddrow $\SS{2}{1}\times \SS{2}{2}$&$12$&$17$&$1$&$0$&$10536$&$2028096$&$636179112$&$190741334400$&$60762986684160$&$19811992617768960$\\
\evnrow $\SS{2}{1}\times \SS{2}{1}$&$6$&$18$&$1$&$0$&$20520$&$4042560$&$1869353640$&$783667509120$&$387953543059200$&$204188081194137600$\\
\oddrow $\PP^1\times \VV{3}{2}$&$16$&$2$&$1$&$0$&$68762$&$55200000$&$61055606526$&$71592493125120$&$88810659628444820$&$114429017109750013440$\\
\addlinespace[1.1ex]
\end{longtable}

  It appears from Table~5 as if the regularized quantum period might coincide for the pairs $\{\Obro{4}{6},\Obro{4}{41}\}$ and $\{\Obro{4}{35}, \Obro{4}{88}\}$.  This is not the case.  The coefficients $\alpha_8$,~$\alpha_9$ in these cases are:
  \begin{center}
    \begin{tabular}{ccc}
      \toprule
      $X$ & $\alpha_8$ & $\alpha_9$ \\
      \midrule
      \oddrow $\Obro{4}{6}$ & 14350 & 87360\\ 
      \evnrow $\Obro{4}{35}$ & 32830 & 227640\\ 
      \oddrow $\Obro{4}{41}$ & 10990 & 102480\\ 
      \evnrow $\Obro{4}{88}$ & 32830 & 212520\\ \bottomrule \\
    \end{tabular}
  \end{center}
  Thus~10 terms of the Taylor expansion of the regularized quantum period suffice to distinguish all of the four-dimensional Fano manifolds considered in this paper.
\end{landscape}

\section{Quantum Differential Operators for Four-Dimensional Fano Manifolds of Index $r>1$: Numerical Results}
\label{appendix:qdes}

In this Appendix we record the quantum differential operators for all four-dimensional Fano manifolds of Fano index $r>1$.   These were computed numerically, as described in \S\ref{sec:qdes}, from $500$ terms of the Taylor expansion of the quantum period.  They pass a number of strong consistency checks, and so we are reasonably confident that they are correct, but this has not been rigorously proven.  We record also the local log-monodromies and ramification defect for the quantum local system, that is, for the local system of solutions to the regularized quantum differential equation.  These are derived using exact computer algebra from the (numerically computed) operators $L_X$, as described in \S\ref{sec:qdes}.

\subsection{$\PP^4$}
\label{operator:P4}

\hfill [\hyperref[sec:index_5]{description p.~\pageref*{sec:index_5}}, \hyperref[table:index_5]{regularized quantum period p.~\pageref*{table:index_5}}]  

The quantum differential operator is:
{\small \[
(5t-1)(625t^4+125t^3+25t^2+5t+1) D^{4} + 
31250t^5 D^{3} + 
109375t^5 D^{2} + 
156250t^5 D + 
75000t^5
\]}
The local log-monodromies for the quantum local system:
\begin{align*}
& \tiny \begin{pmatrix}
0&1&0&0\\
0&0&1&0\\
0&0&0&1\\
0&0&0&0
\end{pmatrix} && \text{at $\textstyle t=0$} \\& \tiny \begin{pmatrix}
0&0&0&0\\
0&0&0&0\\
0&0&0&1\\
0&0&0&0
\end{pmatrix} && \text{at $\textstyle t=\frac{1}{5}$} \\& \tiny \begin{pmatrix}
0&0&0&0\\
0&0&0&0\\
0&0&0&1\\
0&0&0&0
\end{pmatrix} && \text{at the roots of $\textstyle 625t^4+125t^3+25t^2+5t+1=0$} \\\end{align*}
The operator $L_X$ is extremal.

\subsection{$Q^4$}
\label{operator:Q4}

\hfill [\hyperref[sec:index_4]{description p.~\pageref*{sec:index_4}}, \hyperref[table:index_4]{regularized quantum period p.~\pageref*{table:index_4}}]  

The quantum differential operator is:
{\small \[
(32t^2-1)(32t^2+1) D^{4} + 
8192t^4 D^{3} + 
23552t^4 D^{2} + 
28672t^4 D + 
12288t^4
\]}
The local log-monodromies for the quantum local system:
\begin{align*}
& \tiny \begin{pmatrix}
0&1&0&0\\
0&0&1&0\\
0&0&0&1\\
0&0&0&0
\end{pmatrix} && \text{at $\textstyle t=0$} \\& \tiny \begin{pmatrix}
0&0&0&0\\
0&0&0&0\\
0&0&0&1\\
0&0&0&0
\end{pmatrix} && \text{at the roots of $\textstyle 32t^2-1=0$} \\& \tiny \begin{pmatrix}
0&0&0&0\\
0&0&0&0\\
0&0&0&1\\
0&0&0&0
\end{pmatrix} && \text{at the roots of $\textstyle 32t^2+1=0$} \\& \tiny \begin{pmatrix}
0&0&0&0\\
0&0&0&0\\
0&0&0&1\\
0&0&0&0
\end{pmatrix} && \text{at $t=\infty$} \\\end{align*}
The operator $L_X$ is extremal.

\pagebreak

\subsection{$\FI{4}{1}$}
\label{operator:FI^4_1}

\hfill [\hyperref[sec:index_3]{description p.~\pageref*{sec:index_3}}, \hyperref[table:index_3]{regularized quantum period p.~\pageref*{table:index_3}}]  

The quantum differential operator is:
{\small \[
(11664t^3-1) D^{4} + 
69984t^3 D^{3} + 
142884t^3 D^{2} + 
113724t^3 D + 
29160t^3
\]}
The local log-monodromies for the quantum local system:
\begin{align*}
& \tiny \begin{pmatrix}
0&1&0&0\\
0&0&1&0\\
0&0&0&1\\
0&0&0&0
\end{pmatrix} && \text{at $\textstyle t=0$} \\& \tiny \begin{pmatrix}
0&0&0&0\\
0&0&0&0\\
0&0&0&1\\
0&0&0&0
\end{pmatrix} && \text{at the roots of $\textstyle 11664t^3-1=0$} \\& \tiny \begin{pmatrix}
0&0&0&0\\
0&0&0&0\\
0&0&\frac{1}{2}&0\\
0&0&0&\frac{1}{2}
\end{pmatrix} && \text{at $t=\infty$} \\\end{align*}
The operator $L_X$ is extremal.

\subsection{$\FI{4}{2}$}
\label{operator:FI^4_2}

\hfill [\hyperref[sec:index_3]{description p.~\pageref*{sec:index_3}}, \hyperref[table:index_3]{regularized quantum period p.~\pageref*{table:index_3}}]  

The quantum differential operator is:
{\small \[
(12t-1)(144t^2+12t+1) D^{4} + 
10368t^3 D^{3} + 
21924t^3 D^{2} + 
19116t^3 D + 
5832t^3
\]}
The local log-monodromies for the quantum local system:
\begin{align*}
& \tiny \begin{pmatrix}
0&1&0&0\\
0&0&1&0\\
0&0&0&1\\
0&0&0&0
\end{pmatrix} && \text{at $\textstyle t=0$} \\& \tiny \begin{pmatrix}
0&0&0&0\\
0&0&0&0\\
0&0&0&1\\
0&0&0&0
\end{pmatrix} && \text{at $\textstyle t=\frac{1}{12}$} \\& \tiny \begin{pmatrix}
0&0&0&0\\
0&0&0&0\\
0&0&0&1\\
0&0&0&0
\end{pmatrix} && \text{at the roots of $\textstyle 144t^2+12t+1=0$} \\& \tiny \begin{pmatrix}
0&0&0&0\\
0&0&0&0\\
0&0&\frac{3}{4}&0\\
0&0&0&\frac{1}{4}
\end{pmatrix} && \text{at $t=\infty$} \\\end{align*}
The operator $L_X$ is extremal.

\subsection{$\FI{4}{3}$}
\label{operator:FI^4_3}

\hfill [\hyperref[sec:index_3]{description p.~\pageref*{sec:index_3}}, \hyperref[table:index_3]{regularized quantum period p.~\pageref*{table:index_3}}]  

The quantum differential operator is:
{\small \[
(9t-1)(81t^2+9t+1) D^{4} + 
4374t^3 D^{3} + 
9477t^3 D^{2} + 
8748t^3 D + 
2916t^3
\]}
The local log-monodromies for the quantum local system:
\begin{align*}
& \tiny \begin{pmatrix}
0&1&0&0\\
0&0&1&0\\
0&0&0&1\\
0&0&0&0
\end{pmatrix} && \text{at $\textstyle t=0$} \\& \tiny \begin{pmatrix}
0&0&0&0\\
0&0&0&0\\
0&0&0&1\\
0&0&0&0
\end{pmatrix} && \text{at $\textstyle t=\frac{1}{9}$} \\& \tiny \begin{pmatrix}
0&0&0&0\\
0&0&0&0\\
0&0&0&1\\
0&0&0&0
\end{pmatrix} && \text{at the roots of $\textstyle 81t^2+9t+1=0$} \\& \tiny \begin{pmatrix}
0&1&0&0\\
0&0&0&0\\
0&0&0&1\\
0&0&0&0
\end{pmatrix} && \text{at $t=\infty$} \\\end{align*}
The operator $L_X$ is extremal.

\subsection{$\FI{4}{4}$}
\label{operator:FI^4_4}

\hfill [\hyperref[sec:index_3]{description p.~\pageref*{sec:index_3}}, \hyperref[table:index_3]{regularized quantum period p.~\pageref*{table:index_3}}]  

The quantum differential operator is:
{\small \[
(432t^3-1) D^{4} + 
2592t^3 D^{3} + 
5724t^3 D^{2} + 
5508t^3 D + 
1944t^3
\]}
The local log-monodromies for the quantum local system:
\begin{align*}
& \tiny \begin{pmatrix}
0&1&0&0\\
0&0&1&0\\
0&0&0&1\\
0&0&0&0
\end{pmatrix} && \text{at $\textstyle t=0$} \\& \tiny \begin{pmatrix}
0&0&0&0\\
0&0&0&0\\
0&0&0&1\\
0&0&0&0
\end{pmatrix} && \text{at the roots of $\textstyle 432t^3-1=0$} \\& \tiny \begin{pmatrix}
0&0&0&0\\
0&0&0&0\\
0&0&\frac{1}{2}&1\\
0&0&0&\frac{1}{2}
\end{pmatrix} && \text{at $t=\infty$} \\\end{align*}
The operator $L_X$ is extremal.

\subsection{$\FI{4}{5}$}
\label{operator:FI^4_5}

\hfill [\hyperref[sec:index_3]{description p.~\pageref*{sec:index_3}}, \hyperref[table:index_3]{regularized quantum period p.~\pageref*{table:index_3}}]  

The quantum differential operator is:
{\small \[
(729t^6+297t^3-1) D^{4} + 
162t^3(54t^3+11) D^{3} + 
27t^3(1323t^3+148) D^{2} + 
81t^3(702t^3+49) D + 
1458t^3(20t^3+1)
\]}
The local log-monodromies for the quantum local system:
\begin{align*}
& \tiny \begin{pmatrix}
0&1&0&0\\
0&0&1&0\\
0&0&0&1\\
0&0&0&0
\end{pmatrix} && \text{at $\textstyle t=0$} \\& \tiny \begin{pmatrix}
0&0&0&0\\
0&0&0&0\\
0&0&0&1\\
0&0&0&0
\end{pmatrix} && \text{at the roots of $\textstyle 729t^6+297t^3-1=0$} \\\end{align*}
The ramification defect of $L_X$ is $1$.

\subsection{$\PP^2\times \PP^2$}
\label{operator:FI^4_6}

\hfill [\hyperref[sec:index_3]{description p.~\pageref*{sec:index_3}}, \hyperref[table:index_3]{regularized quantum period p.~\pageref*{table:index_3}}]  

The quantum differential operator is:
{\small
  \begin{multline*}
    (3t+1)(6t-1)(9t^2-3t+1)(36t^2+6t+1) D^{4} + 
    162t^3(432t^3+7) D^{3} + \\
    27t^3(10584t^3+95) D^{2} + 
    1296t^3(351t^3+2) D + 
    972t^3(240t^3+1)
  \end{multline*}}
The local log-monodromies for the quantum local system:
\begin{align*}
& \tiny \begin{pmatrix}
0&1&0&0\\
0&0&1&0\\
0&0&0&1\\
0&0&0&0
\end{pmatrix} && \text{at $\textstyle t=0$} \\& \tiny \begin{pmatrix}
0&0&0&0\\
0&0&0&0\\
0&0&0&1\\
0&0&0&0
\end{pmatrix} && \text{at $\textstyle t=\frac{1}{6}$} \\& \tiny \begin{pmatrix}
0&0&0&0\\
0&0&0&0\\
0&0&0&1\\
0&0&0&0
\end{pmatrix} && \text{at $\textstyle t=-\frac{1}{3}$} \\& \tiny \begin{pmatrix}
0&0&0&0\\
0&0&0&0\\
0&0&0&1\\
0&0&0&0
\end{pmatrix} && \text{at the roots of $\textstyle 9t^2-3t+1=0$} \\& \tiny \begin{pmatrix}
0&0&0&0\\
0&0&0&0\\
0&0&0&1\\
0&0&0&0
\end{pmatrix} && \text{at the roots of $\textstyle 36t^2+6t+1=0$} \\\end{align*}
The ramification defect of $L_X$ is $1$.

\subsection{$\VV{4}{2}$} \label{operator:V^4_2}  \hfill [\hyperref[sec:V^4_2]{description p.~\pageref*{sec:V^4_2}}, \hyperref[table:index_2]{regularized quantum period p.~\pageref*{table:index_2}}]

The quantum differential operator is:
{\small \[
(6912t^2-1) D^{4} + 
27648t^2 D^{3} + 
38400t^2 D^{2} + 
21504t^2 D + 
3840t^2
\]}
The local log-monodromies for the quantum local system:
\begin{align*}
& \tiny \begin{pmatrix}
0&1&0&0\\
0&0&1&0\\
0&0&0&1\\
0&0&0&0
\end{pmatrix} && \text{at $\textstyle t=0$} \\& \tiny \begin{pmatrix}
0&0&0&0\\
0&0&0&0\\
0&0&0&1\\
0&0&0&0
\end{pmatrix} && \text{at the roots of $\textstyle 6912t^2-1=0$} \\& \tiny \begin{pmatrix}
0&1&0&0\\
0&0&0&0\\
0&0&\frac{2}{3}&0\\
0&0&0&\frac{1}{3}
\end{pmatrix} && \text{at $t=\infty$} \\\end{align*}
The operator $L_X$ is extremal.

\subsection{$\VV{4}{4}$} \label{operator:V^4_4}  \hfill [\hyperref[sec:V^4_4]{description p.~\pageref*{sec:V^4_4}}, \hyperref[table:index_2]{regularized quantum period p.~\pageref*{table:index_2}}]

The quantum differential operator is:
{\small \[
(32t-1)(32t+1) D^{4} + 
4096t^2 D^{3} + 
5888t^2 D^{2} + 
3584t^2 D + 
768t^2
\]}
The local log-monodromies for the quantum local system:
\begin{align*}
& \tiny \begin{pmatrix}
0&1&0&0\\
0&0&1&0\\
0&0&0&1\\
0&0&0&0
\end{pmatrix} && \text{at $\textstyle t=0$} \\& \tiny \begin{pmatrix}
0&0&0&0\\
0&0&0&0\\
0&0&0&1\\
0&0&0&0
\end{pmatrix} && \text{at $\textstyle t=\frac{1}{32}$} \\& \tiny \begin{pmatrix}
0&0&0&0\\
0&0&0&0\\
0&0&0&1\\
0&0&0&0
\end{pmatrix} && \text{at $\textstyle t=-\frac{1}{32}$} \\& \tiny \begin{pmatrix}
0&1&0&0\\
0&0&0&0\\
0&0&\frac{1}{2}&0\\
0&0&0&\frac{1}{2}
\end{pmatrix} && \text{at $t=\infty$} \\\end{align*}
The operator $L_X$ is extremal.

\subsection{$\VV{4}{6}$} \label{operator:V^4_6}  \hfill [\hyperref[sec:V^4_6]{description p.~\pageref*{sec:V^4_6}}, \hyperref[table:index_2]{regularized quantum period p.~\pageref*{table:index_2}}]

The quantum differential operator is:
{\small \[
(432t^2-1) D^{4} + 
1728t^2 D^{3} + 
2544t^2 D^{2} + 
1632t^2 D + 
384t^2
\]}
The local log-monodromies for the quantum local system:
\begin{align*}
& \tiny \begin{pmatrix}
0&1&0&0\\
0&0&1&0\\
0&0&0&1\\
0&0&0&0
\end{pmatrix} && \text{at $\textstyle t=0$} \\& \tiny \begin{pmatrix}
0&0&0&0\\
0&0&0&0\\
0&0&0&1\\
0&0&0&0
\end{pmatrix} && \text{at the roots of $\textstyle 432t^2-1=0$} \\& \tiny \begin{pmatrix}
0&1&0&0\\
0&0&0&0\\
0&0&\frac{2}{3}&0\\
0&0&0&\frac{1}{3}
\end{pmatrix} && \text{at $t=\infty$} \\\end{align*}
The operator $L_X$ is extremal.

\pagebreak

\subsection{$\VV{4}{8}$} \label{operator:V^4_8}  \hfill [\hyperref[sec:V^4_8]{description p.~\pageref*{sec:V^4_8}}, \hyperref[table:index_2]{regularized quantum period p.~\pageref*{table:index_2}}]

The quantum differential operator is:
{\small \[
(16t-1)(16t+1) D^{4} + 
1024t^2 D^{3} + 
1536t^2 D^{2} + 
1024t^2 D + 
256t^2
\]}
The local log-monodromies for the quantum local system:
\begin{align*}
& \tiny \begin{pmatrix}
0&1&0&0\\
0&0&1&0\\
0&0&0&1\\
0&0&0&0
\end{pmatrix} && \text{at $\textstyle t=0$} \\& \tiny \begin{pmatrix}
0&0&0&0\\
0&0&0&0\\
0&0&0&1\\
0&0&0&0
\end{pmatrix} && \text{at $\textstyle t=\frac{1}{16}$} \\& \tiny \begin{pmatrix}
0&0&0&0\\
0&0&0&0\\
0&0&0&1\\
0&0&0&0
\end{pmatrix} && \text{at $\textstyle t=-\frac{1}{16}$} \\& \tiny \begin{pmatrix}
0&1&0&0\\
0&0&1&0\\
0&0&0&1\\
0&0&0&0
\end{pmatrix} && \text{at $t=\infty$} \\\end{align*}
The operator $L_X$ is extremal.

\subsection{$\VV{4}{10}$} \label{operator:V^4_10}  \hfill [\hyperref[sec:V^4_10]{description p.~\pageref*{sec:V^4_10}}, \hyperref[table:index_2]{regularized quantum period p.~\pageref*{table:index_2}}]

The quantum differential operator is:
{\small \[
(256t^4+176t^2-1) D^{4} + 
64t^2(32t^2+11) D^{3} + 
16t^2(352t^2+67) D^{2} + 
32t^2(192t^2+23) D + 
192t^2(12t^2+1)
\]}
The local log-monodromies for the quantum local system:
\begin{align*}
& \tiny \begin{pmatrix}
0&1&0&0\\
0&0&1&0\\
0&0&0&1\\
0&0&0&0
\end{pmatrix} && \text{at $\textstyle t=0$} \\& \tiny \begin{pmatrix}
0&0&0&0\\
0&0&0&0\\
0&0&0&1\\
0&0&0&0
\end{pmatrix} && \text{at the roots of $\textstyle 256t^4+176t^2-1=0$} \\& \tiny \begin{pmatrix}
0&1&0&0\\
0&0&0&0\\
0&0&0&1\\
0&0&0&0
\end{pmatrix} && \text{at $t=\infty$} \\\end{align*}
The ramification defect of $L_X$ is $1$.

\subsection{$\VV{4}{12}$} \label{operator:V^4_12}  \hfill [\hyperref[sec:V^4_12]{description p.~\pageref*{sec:V^4_12}}, \hyperref[table:index_2]{regularized quantum period p.~\pageref*{table:index_2}}]

The quantum differential operator is:
{\small \[
(4t^2-12t+1)(4t^2+12t+1) D^{4} + 
32t^2(4t^2-17) D^{3} + 
8t^2(46t^2-105) D^{2} + 
16t^2(28t^2-37) D + 
32t^2(6t^2-5)
\]}
The local log-monodromies for the quantum local system:
\begin{align*}
& \tiny \begin{pmatrix}
0&1&0&0\\
0&0&1&0\\
0&0&0&1\\
0&0&0&0
\end{pmatrix} && \text{at $\textstyle t=0$} \\& \tiny \begin{pmatrix}
0&0&0&0\\
0&0&0&0\\
0&0&0&1\\
0&0&0&0
\end{pmatrix} && \text{at the roots of $\textstyle 4t^2-12t+1=0$} \\& \tiny \begin{pmatrix}
0&0&0&0\\
0&0&0&0\\
0&0&0&1\\
0&0&0&0
\end{pmatrix} && \text{at the roots of $\textstyle 4t^2+12t+1=0$} \\& \tiny \begin{pmatrix}
0&0&0&0\\
0&0&0&0\\
0&0&0&1\\
0&0&0&0
\end{pmatrix} && \text{at $t=\infty$} \\\end{align*}
The operator $L_X$ is extremal.

\subsection{$\VV{4}{14}$} \label{operator:V^4_14}  \hfill [\hyperref[sec:V^4_14]{description p.~\pageref*{sec:V^4_14}}, \hyperref[table:index_2]{regularized quantum period p.~\pageref*{table:index_2}}]

The quantum differential operator is:
{\small \[
(4t^2+1)(108t^2-1) D^{4} + 
32t^2(108t^2+13) D^{3} + 
24t^2(406t^2+27) D^{2} + 
16t^2(708t^2+29) D + 
128t^2(36t^2+1)
\]}
The local log-monodromies for the quantum local system:
\begin{align*}
& \tiny \begin{pmatrix}
0&1&0&0\\
0&0&1&0\\
0&0&0&1\\
0&0&0&0
\end{pmatrix} && \text{at $\textstyle t=0$} \\& \tiny \begin{pmatrix}
0&0&0&0\\
0&0&0&0\\
0&0&0&1\\
0&0&0&0
\end{pmatrix} && \text{at the roots of $\textstyle 108t^2-1=0$} \\& \tiny \begin{pmatrix}
0&0&0&0\\
0&0&0&0\\
0&0&0&1\\
0&0&0&0
\end{pmatrix} && \text{at the roots of $\textstyle 4t^2+1=0$} \\& \tiny \begin{pmatrix}
0&0&0&0\\
0&0&0&0\\
0&0&\frac{2}{3}&0\\
0&0&0&\frac{1}{3}
\end{pmatrix} && \text{at $t=\infty$} \\\end{align*}
The ramification defect of $L_X$ is $1$.

\subsection{$\VV{4}{16}$} \label{operator:V^4_16}  \hfill [\hyperref[sec:V^4_16]{description p.~\pageref*{sec:V^4_16}}, \hyperref[table:index_2]{regularized quantum period p.~\pageref*{table:index_2}}]

The quantum differential operator is:
{\small
  \begin{multline*}
    (16t^2-8t-1)(16t^2+8t-1) D^{4} + 
    128t^2(16t^2-3) D^{3} + 
    32t^2(184t^2-19) D^{2} + \\ 
    448t^2(4t-1)(4t+1) D + 
    128t^2(24t^2-1)
  \end{multline*}}
The local log-monodromies for the quantum local system:
\begin{align*}
& \tiny \begin{pmatrix}
0&1&0&0\\
0&0&1&0\\
0&0&0&1\\
0&0&0&0
\end{pmatrix} && \text{at $\textstyle t=0$} \\& \tiny \begin{pmatrix}
0&0&0&0\\
0&0&0&0\\
0&0&0&1\\
0&0&0&0
\end{pmatrix} && \text{at the roots of $\textstyle 16t^2-8t-1=0$} \\& \tiny \begin{pmatrix}
0&0&0&0\\
0&0&0&0\\
0&0&0&1\\
0&0&0&0
\end{pmatrix} && \text{at the roots of $\textstyle 16t^2+8t-1=0$} \\& \tiny \begin{pmatrix}
0&0&0&0\\
0&0&0&0\\
0&0&0&1\\
0&0&0&0
\end{pmatrix} && \text{at $t=\infty$} \\\end{align*}
The operator $L_X$ is extremal.

\subsection{$\VV{4}{18}$} \label{operator:V^4_18} \hfill [\hyperref[sec:V^4_18]{description p.~\pageref*{sec:V^4_18}}, \hyperref[table:index_2]{regularized quantum period p.~\pageref*{table:index_2}}]

The quantum differential operator is:
{\small \[
(432t^4+72t^2-1) D^{4} + 
288t^2(12t^2+1) D^{3} + 
24t^2(414t^2+19) D^{2} + 
336t^2(36t^2+1) D + 
96t^2(54t^2+1)
\]}
The local log-monodromies for the quantum local system:
\begin{align*}
& \tiny \begin{pmatrix}
0&1&0&0\\
0&0&1&0\\
0&0&0&1\\
0&0&0&0
\end{pmatrix} && \text{at $\textstyle t=0$} \\& \tiny \begin{pmatrix}
0&0&0&0\\
0&0&0&0\\
0&0&0&1\\
0&0&0&0
\end{pmatrix} && \text{at the roots of $\textstyle 432t^4+72t^2-1=0$} \\& \tiny \begin{pmatrix}
0&0&0&0\\
0&0&0&0\\
0&0&0&1\\
0&0&0&0
\end{pmatrix} && \text{at $t=\infty$} \\\end{align*}
The operator $L_X$ is extremal.

\subsection{$\MW{4}{1}$} \label{operator:MW^4_1}  \hfill [\hyperref[sec:MW^4_1]{description p.~\pageref*{sec:MW^4_1}}, \hyperref[table:index_2]{regularized quantum period p.~\pageref*{table:index_2}}]

The quantum differential operator is:
{\small
  \begin{multline*}
    (2t-1)^2(2t+1)^2(1724t^2-4t-1)(1724t^2+4t-1)\\
    \shoveright{(1384128950480t^6-34997928616t^4-263676995t^2+9409) D^{6} }\\
    \shoveleft{+ 2(2t-1)(2t+1)(181010493290961157120t^{12}-17468834144875533568t^{10}+374429340495784832t^8+}\\
    \shoveright{3959280486757728t^6-1227737299988t^4-725619617t^2+18818) D^{5}} \\    
    \shoveleft{+ 4\left(3126544884116601804800t^{14}-725557954486979610624t^{12}+31166631689741025792t^{10}- \right. }\\
    \shoveright{\left. 440963660134839040t^8-7399870298607304t^6-348759582360t^4-504354223t^2+9409\right) D^{4}} \\
    \shoveleft{+ 8t^2(6746754749935824947200t^{12}-1095289161198143939072t^{10}+33852557194447324800t^8-}\\
    \shoveright{138056179574882528t^6-6267098983057824t^4-620133376448t^2-10735669) D^{3}} \\
    \shoveleft{+ 32t^2(3803277296533945221760t^{12}-460852243532660846400t^{10}+13408867107650109352t^8+}\\
    \shoveright{72880113460188392t^6-1247120973283936t^4-144892007990t^2-573949) D^{2}}\\
    \shoveleft{+ 256t^4(524004808703094940640t^{10}-51043969670376668752t^8+1573801437077923102t^6+}\\
    \shoveright{16338545311012128t^4-54334824441981t^2-5942952755) D} \\
    \shoveleft{+ 1536t^4(35996404915816139200t^{10}-3025511420019920960t^8+}\\
    99559010182515260t^6+1234528802429310t^4-1125770982819t^2-89024854)
\end{multline*}}
The local log-monodromies for the quantum local system:
\begin{align*}
& \tiny \begin{pmatrix}
0&1&0&0&0&0\\
0&0&0&0&0&0\\
0&0&0&1&0&0\\
0&0&0&0&1&0\\
0&0&0&0&0&1\\
0&0&0&0&0&0
\end{pmatrix} && \text{at $\textstyle t=0$} \\& \tiny \begin{pmatrix}
0&0&0&0&0&0\\
0&0&0&0&0&0\\
0&0&0&0&0&0\\
0&0&0&0&0&0\\
0&0&0&0&\frac{5}{6}&0\\
0&0&0&0&0&\frac{1}{6}
\end{pmatrix} && \text{at $\textstyle t=\frac{1}{2}$} \\& \tiny \begin{pmatrix}
0&0&0&0&0&0\\
0&0&0&0&0&0\\
0&0&0&0&0&0\\
0&0&0&0&0&0\\
0&0&0&0&\frac{5}{6}&0\\
0&0&0&0&0&\frac{1}{6}
\end{pmatrix} && \text{at $\textstyle t=-\frac{1}{2}$} \\& \tiny \begin{pmatrix}
0&0&0&0&0&0\\
0&0&0&0&0&0\\
0&0&0&0&0&0\\
0&0&0&0&0&0\\
0&0&0&0&0&1\\
0&0&0&0&0&0
\end{pmatrix} && \text{at the roots of $\textstyle 1724t^2-4t-1=0$} \\& \tiny \begin{pmatrix}
0&0&0&0&0&0\\
0&0&0&0&0&0\\
0&0&0&0&0&0\\
0&0&0&0&0&0\\
0&0&0&0&0&1\\
0&0&0&0&0&0
\end{pmatrix} && \text{at the roots of $\textstyle 1724t^2+4t-1=0$} \\\end{align*}
The operator $L_X$ is extremal.

\pagebreak

\subsection{$\MW{4}{2}$} \label{operator:MW^4_2}  \hfill [\hyperref[sec:MW^4_2]{description p.~\pageref*{sec:MW^4_2}}, \hyperref[table:index_2]{regularized quantum period p.~\pageref*{table:index_2}}]

The quantum differential operator is:
{\small
  \begin{multline*}
    (2t-1)^2(2t+1)^2(14t-1)(14t+1)(18t-1)(18t+1)(7544656t^6-3112t^4-6667t^2+1) D^{6} \\
    \shoveleft{+ 2(2t-1)(2t+1)} \\
    \shoveright{(21081096723456t^{12}-1556683668736t^{10}-25006224512t^8+
    2282941792t^6-4014452t^4-18937t^2+2) D^{5}} \\
    \shoveleft{+ 4\left(364128034314240t^{14}-74616093755392t^{12}+1020882362880t^{10}+\right.} \\
    \shoveright{\left. 104092887296t^8-4150928136t^6-1064664t^4-12743t^2+1\right) D^{4}}\\
    \shoveleft{+ 8t^2(785749968783360t^{12}-103037075309056t^{10}-578957540736t^8+}\\
    \shoveright{167061307168t^6-3800819360t^4-2728128t^2-205) D^{3}}\\
    \shoveleft{+ 32t^2(442942589109888t^{12}-38313217780544t^{10}-591376611224t^8+}\\
    \shoveright{79457857384t^6-864235264t^4-697078t^2-13) D^{2}} \\
    \shoveright{+ 256t^4(61027379435232t^{10}-3662113129808t^8-80739092050t^6+9991256448t^4-46797421t^2-33179) D}\\
    + 7680t^4(838452710592t^{10}-37503518528t^8-1048500436t^6+130762470t^4-249799t^2-110)
  \end{multline*}}
The local log-monodromies for the quantum local system:
\begin{align*}
& \tiny \begin{pmatrix}
0&1&0&0&0&0\\
0&0&0&0&0&0\\
0&0&0&1&0&0\\
0&0&0&0&1&0\\
0&0&0&0&0&1\\
0&0&0&0&0&0
\end{pmatrix} && \text{at $\textstyle t=0$} \\& \tiny \begin{pmatrix}
0&0&0&0&0&0\\
0&0&0&0&0&0\\
0&0&0&1&0&0\\
0&0&0&0&0&0\\
0&0&0&0&0&1\\
0&0&0&0&0&0
\end{pmatrix} && \text{at $\textstyle t=\frac{1}{2}$} \\& \tiny \begin{pmatrix}
0&0&0&0&0&0\\
0&0&0&0&0&0\\
0&0&0&0&0&0\\
0&0&0&0&0&0\\
0&0&0&0&0&1\\
0&0&0&0&0&0
\end{pmatrix} && \text{at $\textstyle t=\frac{1}{14}$} \\& \tiny \begin{pmatrix}
0&0&0&0&0&0\\
0&0&0&0&0&0\\
0&0&0&0&0&0\\
0&0&0&0&0&0\\
0&0&0&0&0&1\\
0&0&0&0&0&0
\end{pmatrix} && \text{at $\textstyle t=\frac{1}{18}$} \\& \tiny \begin{pmatrix}
0&0&0&0&0&0\\
0&0&0&0&0&0\\
0&0&0&0&0&0\\
0&0&0&0&0&0\\
0&0&0&0&0&1\\
0&0&0&0&0&0
\end{pmatrix} && \text{at $\textstyle t=-\frac{1}{18}$} \\& \tiny \begin{pmatrix}
0&0&0&0&0&0\\
0&0&0&0&0&0\\
0&0&0&0&0&0\\
0&0&0&0&0&0\\
0&0&0&0&0&1\\
0&0&0&0&0&0
\end{pmatrix} && \text{at $\textstyle t=-\frac{1}{14}$} \\& \tiny \begin{pmatrix}
0&0&0&0&0&0\\
0&0&0&0&0&0\\
0&0&0&1&0&0\\
0&0&0&0&0&0\\
0&0&0&0&0&1\\
0&0&0&0&0&0
\end{pmatrix} && \text{at $\textstyle t=-\frac{1}{2}$} \\\end{align*}
The operator $L_X$ is extremal.

\pagebreak

\subsection{$\MW{4}{3}$} \label{operator:MW^4_3}  \hfill [\hyperref[sec:MW^4_3]{description p.~\pageref*{sec:MW^4_3}}, \hyperref[table:index_2]{regularized quantum period p.~\pageref*{table:index_2}}]

The quantum differential operator is:
{\small
  \begin{multline*}
    (2t-1)^2(2t+1)^2(104t^2-4t-1)(104t^2+4t-1)(124883200t^6+3445552t^4-190621t^2+50) D^{6} \\
    \shoveleft{+ 2(2t-1)(2t+1)(59432414412800t^{12}-3046657163264t^{10}-307452667136t^8}\\
    \shoveright{+12555781056t^6-36422116t^4-548263t^2+100) D^{5}} \\
    \shoveleft{+ 4\left(1026559885312000t^{14}-181195540611072t^{12}-5355282845184t^{10}+794966941312t^8 \right.}\\
    \shoveright{\left. -20516326820t^6-1544700t^4-364427t^2+50\right) D^{4}} \\
    \shoveleft{+ 8t^2(2215208173568000t^{12}-216948202172416t^{10}-15687219459072t^8} \\
    \shoveright{+1005364344896t^6-20392735992t^4-36932632t^2-5075) D^{3}}\\
    \shoveleft{+ 64t^2(624378035507200t^{12}-30899527376640t^{10}-4029754891664t^8}\\
    \shoveright{+205143060452t^6-2542262335t^4-5022410t^2-175) D^{2}}\\
    \shoveleft{+ 256t^4(172050086041600t^{10}-3462364820096t^8} \\
    \shoveright{-968115899896t^6+47585432988t^4-312087909t^2-517438) D} \\
    + 1536t^4(11818946048000t^{10}-33971342080t^8-59641497680t^6+2994403590t^4-9592347t^2-9140)
  \end{multline*}}
The local log-monodromies for the quantum local system:
\begin{align*}
& \tiny \begin{pmatrix}
0&1&0&0&0&0\\
0&0&0&0&0&0\\
0&0&0&1&0&0\\
0&0&0&0&1&0\\
0&0&0&0&0&1\\
0&0&0&0&0&0
\end{pmatrix} && \text{at $\textstyle t=0$} \\& \tiny \begin{pmatrix}
0&0&0&0&0&0\\
0&0&0&0&0&0\\
0&0&0&0&0&0\\
0&0&0&0&0&0\\
0&0&0&0&\frac{5}{6}&0\\
0&0&0&0&0&\frac{1}{6}
\end{pmatrix} && \text{at $\textstyle t=\frac{1}{2}$} \\& \tiny \begin{pmatrix}
0&0&0&0&0&0\\
0&0&0&0&0&0\\
0&0&0&0&0&0\\
0&0&0&0&0&0\\
0&0&0&0&\frac{5}{6}&0\\
0&0&0&0&0&\frac{1}{6}
\end{pmatrix} && \text{at $\textstyle t=-\frac{1}{2}$} \\& \tiny \begin{pmatrix}
0&0&0&0&0&0\\
0&0&0&0&0&0\\
0&0&0&0&0&0\\
0&0&0&0&0&0\\
0&0&0&0&0&1\\
0&0&0&0&0&0
\end{pmatrix} && \text{at the roots of $\textstyle 104t^2-4t-1=0$} \\& \tiny \begin{pmatrix}
0&0&0&0&0&0\\
0&0&0&0&0&0\\
0&0&0&0&0&0\\
0&0&0&0&0&0\\
0&0&0&0&0&1\\
0&0&0&0&0&0
\end{pmatrix} && \text{at the roots of $\textstyle 104t^2+4t-1=0$} \\\end{align*}
The operator $L_X$ is extremal.

\pagebreak

\subsection{$\MW{4}{4}$} \label{operator:MW^4_4}  \hfill [\hyperref[sec:MW^4_4]{description p.~\pageref*{sec:MW^4_4}}, \hyperref[table:index_2]{regularized quantum period p.~\pageref*{table:index_2}}]

The quantum differential operator is:
{\small \[
(16t^2+1)(128t^2-1) D^{4} + 
64t^2(256t^2+7) D^{3} + 
16t^2(2816t^2+43) D^{2} + 
96t^2(512t^2+5) D + 
128t^2(144t^2+1)
\]}
The local log-monodromies for the quantum local system:
\begin{align*}
& \tiny \begin{pmatrix}
0&1&0&0\\
0&0&1&0\\
0&0&0&1\\
0&0&0&0
\end{pmatrix} && \text{at $\textstyle t=0$} \\& \tiny \begin{pmatrix}
0&0&0&0\\
0&0&0&0\\
0&0&0&1\\
0&0&0&0
\end{pmatrix} && \text{at the roots of $\textstyle 128t^2-1=0$} \\& \tiny \begin{pmatrix}
0&0&0&0\\
0&0&0&0\\
0&0&0&1\\
0&0&0&0
\end{pmatrix} && \text{at the roots of $\textstyle 16t^2+1=0$} \\& \tiny \begin{pmatrix}
0&1&0&0\\
0&0&0&0\\
0&0&0&1\\
0&0&0&0
\end{pmatrix} && \text{at $t=\infty$} \\\end{align*}
The ramification defect of $L_X$ is $1$.

\subsection{$\MW{4}{5}$} \label{operator:MW^4_5}  \hfill [\hyperref[sec:MW^4_5]{description p.~\pageref*{sec:MW^4_5}}, \hyperref[table:index_2]{regularized quantum period p.~\pageref*{table:index_2}}]

The quantum differential operator is:
{\small
  \begin{multline*}
    (32t^4-144t^3+40t^2-12t+1)(32t^4+144t^3+40t^2+12t+1)\\
    \shoveright{(1379024896t^8-181690112t^6+32203856t^4+160775t^2+136) D^{6}} \\
    \shoveleft{+2\left(14121214935040t^{16}-152437496479744t^{14}+18570092085248t^{12}-3527276827648t^{10}-\right.} \\
    \shoveright{\left.277037824256t^8-556812800t^6-140567296t^4-499733t^2-272\right) D^{5}}\\
    \shoveleft{+ 4\left(54366677499904t^{16}-340827602288640t^{14}+64054093561856t^{12}-13315483720192t^{10}\right.}\\
    \shoveright{\left.-406200526464t^8-986296160t^6+122964452t^4+370629t^2+136\right) D^{4}}\\
    \shoveleft{+ 8t^2(103084869025792t^{14}-369261049675776t^{12}+91720177627136t^{10}-21127821675520t^8}\\
    \shoveright{-305989253824t^6-1259677440t^4-1614768t^2-5831) D^{3}}\\
    \shoveleft{+128t^2(12720125640704t^{14}-26591028115456t^{12}+8287153333632t^{10}}\\
    \shoveright{-2016286904136t^8-15422020074t^6-337096385t^4-533452t^2-102) D^{2}}\\
    \shoveleft{+ 256t^4(6200095932416t^{12}-7989995554816t^{10}+3020985441920t^8}\\
    \shoveright{-739274260600t^6-4109085596t^4-179641501t^2-331640) D} \\
    \shoveleft{+ 1536t^4(386126970880t^{12}-333456404480t^{10}+147420386560t^8}\\
    -34919816144t^6-204213358t^4-10473275t^2-20944)
  \end{multline*}}
The local log-monodromies for the quantum local system:
\begin{align*}
& \tiny \begin{pmatrix}
0&1&0&0&0&0\\
0&0&0&0&0&0\\
0&0&0&1&0&0\\
0&0&0&0&1&0\\
0&0&0&0&0&1\\
0&0&0&0&0&0
\end{pmatrix} && \text{at $\textstyle t=0$} \\& \tiny \begin{pmatrix}
0&0&0&0&0&0\\
0&0&0&0&0&0\\
0&0&0&0&0&0\\
0&0&0&0&0&0\\
0&0&0&0&0&1\\
0&0&0&0&0&0
\end{pmatrix} && \text{at the roots of $\textstyle 32t^4-144t^3+40t^2-12t+1=0$} \\& \tiny \begin{pmatrix}
0&0&0&0&0&0\\
0&0&0&0&0&0\\
0&0&0&0&0&0\\
0&0&0&0&0&0\\
0&0&0&0&0&1\\
0&0&0&0&0&0
\end{pmatrix} && \text{at the roots of $\textstyle 32t^4+144t^3+40t^2+12t+1=0$} \\& \tiny \begin{pmatrix}
0&0&0&0&0&0\\
0&0&0&0&0&0\\
0&0&0&0&0&0\\
0&0&0&0&0&0\\
0&0&0&0&0&1\\
0&0&0&0&0&0
\end{pmatrix} && \text{at $t=\infty$} \\\end{align*}
The ramification defect of $L_X$ is $1$.

\subsection{$\MW{4}{6}$} \label{operator:MW^4_6}  \hfill [\hyperref[sec:MW^4_6]{description p.~\pageref*{sec:MW^4_6}}, \hyperref[table:index_2]{regularized quantum period p.~\pageref*{table:index_2}}]

The quantum differential operator is:
{\small
  \begin{multline*}
    (2t-1)^2(2t+1)^2(6t-1)(6t+1)(10t-1)(10t+1)(1433040t^6+80728t^4-2579t^2+1) D^{6} \\
    \shoveleft{+ 2(2t-1)(2t+1)(226993536000t^{12}-6984049920t^{10}}\\
    \shoveright{-2174297216t^8+67707232t^6-106452t^4-7441t^2+2) D^{5}} \\
    \shoveleft{+ 4\left(3920797440000t^{14}-584516290560t^{12}-52264472064t^{10}\right.}\\
    \shoveright{\left.+4812821760t^8-86750536t^6+175464t^4-4911t^2+1\right) D^{4}} \\
    \shoveleft{+ 8t^2(8460668160000t^{12}-550314631680t^{10}-112753881472t^8}\\
    \shoveright{+5245181728t^6-97715360t^4-329920t^2-69) D^{3}} \\
    \shoveleft{+ 32t^2(4769443728000t^{12}-59034734400t^{10}-53369502424t^8}\\
    \shoveright{+1881440232t^6-26802560t^4-94678t^2-5) D^{2}} \\
    \shoveright{+ 768t^4(219040164000t^{10}+4402781840t^8-2035415446t^6+67181120t^4-620103t^2-1737) D} \\
    + 4608t^4(15046920000t^{10}+593723200t^8-120630700t^6+4065994t^4-22137t^2-34)
  \end{multline*}}
The local log-monodromies for the quantum local system:
\begin{align*}
& \tiny \begin{pmatrix}
0&1&0&0&0&0\\
0&0&0&0&0&0\\
0&0&0&1&0&0\\
0&0&0&0&1&0\\
0&0&0&0&0&1\\
0&0&0&0&0&0
\end{pmatrix} && \text{at $\textstyle t=0$} \\& \tiny \begin{pmatrix}
0&0&0&0&0&0\\
0&0&0&0&0&0\\
0&0&0&0&0&0\\
0&0&0&0&0&0\\
0&0&0&0&\frac{1}{2}&1\\
0&0&0&0&0&\frac{1}{2}
\end{pmatrix} && \text{at $\textstyle t=\frac{1}{2}$} \\& \tiny \begin{pmatrix}
0&0&0&0&0&0\\
0&0&0&0&0&0\\
0&0&0&0&0&0\\
0&0&0&0&0&0\\
0&0&0&0&0&1\\
0&0&0&0&0&0
\end{pmatrix} && \text{at $\textstyle t=\frac{1}{6}$} \\& \tiny \begin{pmatrix}
0&0&0&0&0&0\\
0&0&0&0&0&0\\
0&0&0&0&0&0\\
0&0&0&0&0&0\\
0&0&0&0&0&1\\
0&0&0&0&0&0
\end{pmatrix} && \text{at $\textstyle t=\frac{1}{10}$} \\& \tiny \begin{pmatrix}
0&0&0&0&0&0\\
0&0&0&0&0&0\\
0&0&0&0&0&0\\
0&0&0&0&0&0\\
0&0&0&0&0&1\\
0&0&0&0&0&0
\end{pmatrix} && \text{at $\textstyle t=-\frac{1}{10}$} \\& \tiny \begin{pmatrix}
0&0&0&0&0&0\\
0&0&0&0&0&0\\
0&0&0&0&0&0\\
0&0&0&0&0&0\\
0&0&0&0&0&1\\
0&0&0&0&0&0
\end{pmatrix} && \text{at $\textstyle t=-\frac{1}{6}$} \\& \tiny \begin{pmatrix}
0&0&0&0&0&0\\
0&0&0&0&0&0\\
0&0&0&0&0&0\\
0&0&0&0&0&0\\
0&0&0&0&\frac{1}{2}&1\\
0&0&0&0&0&\frac{1}{2}
\end{pmatrix} && \text{at $\textstyle t=-\frac{1}{2}$} \\\end{align*}
The operator $L_X$ is extremal.

\pagebreak

\subsection{$\MW{4}{7}$} \label{operator:MW^4_7}  \hfill [\hyperref[sec:MW^4_7]{description p.~\pageref*{sec:MW^4_7}}, \hyperref[table:index_2]{regularized quantum period p.~\pageref*{table:index_2}}]

The quantum differential operator is:
{\small \[
(8t-1)(8t+1)(16t^2+1) D^{4} + 
64t^2(128t^2+3) D^{3} + 
16t^2(1456t^2+19) D^{2} + 
32t^2(864t^2+7) D + 
64t^2(180t^2+1)
\]}
The local log-monodromies for the quantum local system:
\begin{align*}
& \tiny \begin{pmatrix}
0&1&0&0\\
0&0&1&0\\
0&0&0&1\\
0&0&0&0
\end{pmatrix} && \text{at $\textstyle t=0$} \\& \tiny \begin{pmatrix}
0&0&0&0\\
0&0&0&0\\
0&0&0&1\\
0&0&0&0
\end{pmatrix} && \text{at $\textstyle t=\frac{1}{8}$} \\& \tiny \begin{pmatrix}
0&0&0&0\\
0&0&0&0\\
0&0&0&1\\
0&0&0&0
\end{pmatrix} && \text{at $\textstyle t=-\frac{1}{8}$} \\& \tiny \begin{pmatrix}
0&0&0&0\\
0&0&0&0\\
0&0&0&1\\
0&0&0&0
\end{pmatrix} && \text{at the roots of $\textstyle 16t^2+1=0$} \\& \tiny \begin{pmatrix}
0&0&0&0\\
0&0&0&0\\
0&0&\frac{1}{2}&0\\
0&0&0&\frac{1}{2}
\end{pmatrix} && \text{at $t=\infty$} \\\end{align*}
The ramification defect of $L_X$ is $1$.

\subsection{$\MW{4}{8}$} \label{operator:MW^4_8}  \hfill [\hyperref[sec:MW^4_8]{description p.~\pageref*{sec:MW^4_8}}, \hyperref[table:index_2]{regularized quantum period p.~\pageref*{table:index_2}}]

The quantum differential operator is:
{\small
  \begin{multline*}
    (2t-1)(2t+1)(8000t^6+528t^4+60t^2-1)(64404500t^8-1791160t^6+729408t^4+2144t^2+5) D^{6}  \\
    \shoveleft{+ 4\left(10304720000000t^{16}-1452885248000t^{14}+158410065280t^{12}\right.}\\
    \shoveright{\left.-11837396096t^{10}-726357052t^8+2787020t^6-1531744t^4-3536t^2-5\right) D^{5}} \\ 
    \shoveleft{+ 4\left(79346344000000t^{16}-7715142848000t^{14}+1509213811040t^{12}\right.}\\
    \shoveright{\left.-96608493568t^{10}-1083561820t^8+48291628t^6+2808508t^4+5602t^2+5\right) D^{4}}\\
    \shoveleft{+ 32t^2(37612228000000t^{14}-2748389976000t^{12}+827208001440t^{10}}\\
    \shoveright{-40569448408t^8-339211968t^6+5857560t^4-32366t^2-51) D^{3}}\\
    \shoveleft{+ 32t^2(74258388500000t^{14}-4572988268000t^{12}+1872888803480t^{10}}\\
    \shoveright{-62657669272t^8-239219588t^6-13755238t^4-70290t^2-15) D^{2}}\\  
    \shoveleft{+512t^4(4524416125000t^{12}-258077892500t^{10}+128903007200t^8}\\
    \shoveright{-2846763420t^6-5734916t^4-1401952t^2-4225) D} \\
    +3072t^4(281769687500t^{12}-15710668750t^{10}+8827474000t^8-131446365t^6-200152t^4-95918t^2-245)
  \end{multline*}}
The local log-monodromies for the quantum local system:
\begin{align*}
& \tiny \begin{pmatrix}
0&1&0&0&0&0\\
0&0&0&0&0&0\\
0&0&0&1&0&0\\
0&0&0&0&1&0\\
0&0&0&0&0&1\\
0&0&0&0&0&0
\end{pmatrix} && \text{at $\textstyle t=0$} \\& \tiny \begin{pmatrix}
0&0&0&0&0&0\\
0&0&0&0&0&0\\
0&0&0&0&0&0\\
0&0&0&0&0&0\\
0&0&0&0&0&1\\
0&0&0&0&0&0
\end{pmatrix} && \text{at $\textstyle t=\frac{1}{2}$} \\& \tiny \begin{pmatrix}
0&0&0&0&0&0\\
0&0&0&0&0&0\\
0&0&0&0&0&0\\
0&0&0&0&0&0\\
0&0&0&0&0&1\\
0&0&0&0&0&0
\end{pmatrix} && \text{at $\textstyle t=-\frac{1}{2}$} \\& \tiny \begin{pmatrix}
0&0&0&0&0&0\\
0&0&0&0&0&0\\
0&0&0&0&0&0\\
0&0&0&0&0&0\\
0&0&0&0&0&1\\
0&0&0&0&0&0
\end{pmatrix} && \text{at the roots of $\textstyle 8000t^6+528t^4+60t^2-1=0$} \\& \tiny \begin{pmatrix}
0&0&0&0&0&0\\
0&0&0&0&0&0\\
0&0&0&0&0&0\\
0&0&0&0&0&0\\
0&0&0&0&0&1\\
0&0&0&0&0&0
\end{pmatrix} && \text{at $t=\infty$} \\\end{align*}
The ramification defect of $L_X$ is $1$.

\subsection{$\MW{4}{9}$} \label{operator:MW^4_9}  \hfill [\hyperref[sec:MW^4_9]{description p.~\pageref*{sec:MW^4_9}}, \hyperref[table:index_2]{regularized quantum period p.~\pageref*{table:index_2}}]

The quantum differential operator is:
{\small
  \begin{multline*}
    (176t^4-144t^3+20t^2+8t-1)(176t^4+144t^3+20t^2-8t-1)(2109888t^6+174528t^4-2941t^2+2) D^{6} \\
    \shoveleft{+ 2\left(718914797568t^{14}-137405030400t^{12}-5419619584t^{10}\right.}\\
    \shoveright{\left.+2224774528t^8-45409520t^6-364024t^4+8407t^2-4\right) D^{5}} \\
    \shoveleft{+ 4\left(3104404807680t^{14}-245927227392t^{12}-51116043392t^{10}\right.}\\
    \shoveright{\left.+3834427616t^8-24121292t^6+553996t^4-5471t^2+2\right) D^{4}} \\
    \shoveleft{+ 8t^2(6698978795520t^{12}-43313504256t^{10}-101803764992t^8}\\
    \shoveright{+3311154624t^6-46433864t^4-300504t^2-107) D^{3}} \\
    \shoveright{+  64t^2(1888172529408t^{12}+75804905088t^{10}-23177683476t^8+452012856t^6-8043103t^4-46340t^2-4) D^{2}} \\
    \shoveright{+ 4608t^4(28905231168t^{10}+2010341344t^8-284249759t^6+4423780t^4-79284t^2-316) D} \\
    + 
    55296t^4(992819520t^{10}+86769760t^8-8133271t^6+123438t^4-1829t^2-4)
  \end{multline*}}
The local log-monodromies for the quantum local system:
\begin{align*}
& \tiny \begin{pmatrix}
0&1&0&0&0&0\\
0&0&0&0&0&0\\
0&0&0&1&0&0\\
0&0&0&0&1&0\\
0&0&0&0&0&1\\
0&0&0&0&0&0
\end{pmatrix} && \text{at $\textstyle t=0$} \\& \tiny \begin{pmatrix}
0&0&0&0&0&0\\
0&0&0&0&0&0\\
0&0&0&0&0&0\\
0&0&0&0&0&0\\
0&0&0&0&0&1\\
0&0&0&0&0&0
\end{pmatrix} && \text{at the roots of $\textstyle 176t^4-144t^3+20t^2+8t-1=0$} \\& \tiny \begin{pmatrix}
0&0&0&0&0&0\\
0&0&0&0&0&0\\
0&0&0&0&0&0\\
0&0&0&0&0&0\\
0&0&0&0&0&1\\
0&0&0&0&0&0
\end{pmatrix} && \text{at the roots of $\textstyle 176t^4+144t^3+20t^2-8t-1=0$} \\\end{align*}
The operator $L_X$ is extremal.

\pagebreak

\subsection{$\MW{4}{10}$} \label{operator:MW^4_10}  \hfill [\hyperref[sec:MW^4_10]{description p.~\pageref*{sec:MW^4_10}}, \hyperref[table:index_2]{regularized quantum period p.~\pageref*{table:index_2}}]

The quantum differential operator is:
{\small
  \begin{multline*}
    (4t^2+1)(16384t^6+512t^4+44t^2-1)(1473757184t^8+13746176t^6+7592448t^4+46808t^2+55) D^{6} \\ 
    \shoveleft{+ 4\left(482920754053120t^{16}+86447658893312t^{14}+4886210281472t^{12}+534951231488t^{10}\right.}\\
    \shoveright{\left.+18501230592t^8+119944000t^6+16234576t^4+72412t^2+55\right) D^{5}} \\
    \shoveleft{+ 4\left(3718489806209024t^{16}+396693178679296t^{14}+31496439267328t^{12}+3577124749312t^{10}\right.}\\
    \shoveright{\left.+37524877312t^8-202496928t^6-28001956t^4-105906t^2-55\right) D^{4}}\\
    \shoveleft{+ 16t^2(3525321504587776t^{14}+219939516448768t^{12}+31919813689344t^{10}}\\
    \shoveright{+2870602727424t^8+30098031616t^6+29514768t^4+38264t^2+755) D^{3}}\\
    \shoveleft{+ 64t^2(1740023841947648t^{14}+63218454626304t^{12}+17856939016192t^{10}}\\
    \shoveright{+1163732019200t^8+10148827776t^6+77992960t^4+145739t^2+55) D^{2}} \\
    \shoveleft{+ 256t^4(424064787152896t^{12}+8915881033728t^{10}+4878831042560t^8}\\
    \shoveright{+233990370304t^6+2025100864t^4+24310608t^2+53911) D} \\
    + 30720t^4(1320486436864t^{12}+15835594752t^{10}+16557547520t^8+617585152t^6+5801632t^4+78192t^2+187)
  \end{multline*}}
The local log-monodromies for the quantum local system:
\begin{align*}
& \tiny \begin{pmatrix}
0&1&0&0&0&0\\
0&0&0&0&0&0\\
0&0&0&1&0&0\\
0&0&0&0&1&0\\
0&0&0&0&0&1\\
0&0&0&0&0&0
\end{pmatrix} && \text{at $\textstyle t=0$} \\& \tiny \begin{pmatrix}
0&0&0&0&0&0\\
0&0&0&0&0&0\\
0&0&0&0&0&0\\
0&0&0&0&0&0\\
0&0&0&0&0&1\\
0&0&0&0&0&0
\end{pmatrix} && \text{at the roots of $\textstyle 4t^2+1=0$} \\& \tiny \begin{pmatrix}
0&0&0&0&0&0\\
0&0&0&0&0&0\\
0&0&0&0&0&0\\
0&0&0&0&0&0\\
0&0&0&0&0&1\\
0&0&0&0&0&0
\end{pmatrix} && \text{at the roots of $\textstyle 16384t^6+512t^4+44t^2-1=0$} \\& \tiny \begin{pmatrix}
0&0&0&0&0&0\\
0&0&0&0&0&0\\
0&0&0&0&0&0\\
0&0&0&0&0&0\\
0&0&0&0&0&1\\
0&0&0&0&0&0
\end{pmatrix} && \text{at $t=\infty$} \\\end{align*}
The ramification defect of $L_X$ is $1$.

\subsection{$\MW{4}{11}$} \label{operator:MW^4_11}  \hfill [\hyperref[sec:MW^4_11]{description p.~\pageref*{sec:MW^4_11}}, \hyperref[table:index_2]{regularized quantum period p.~\pageref*{table:index_2}}]

The quantum differential operator is:
{\small
  \begin{multline*}
    (176t^4-96t^3+8t^2-4t-1)(176t^4+96t^3+8t^2+4t-1)(46963840t^6-6320080t^4+10817t^2-60) D^{6} \\
    \shoveleft{+ 2\left(16002270986240t^{14}-4453229608960t^{12}+179162526976t^{10}+27549781760t^8\right.}\\
    \shoveright{\left.-288092960t^6+25314336t^4-28611t^2+120\right) D^{5}} \\
    \shoveleft{+ 4\left(69100715622400t^{14}-16714357862400t^{12}+879319798400t^{10}\right.} \\
    \shoveright{\left.+38172435520t^8-54159304t^6-23633284t^4+6819t^2-60\right) D^{4}}\\
    \shoveleft{+ 8t^2(149112070553600t^{12}-33265162726400t^{10}+1526941953280t^8}\\
    \shoveright{+17140032640t^6-188273072t^4+3160464t^2+1615) D^{3}}\\
    \shoveleft{+ 128t^2(21014345918720t^{12}-4511525020640t^{10}+157602014826t^8}\\
    \shoveright{-362882280t^6+22427697t^4+282960t^2+30) D^{2}}\\
    \shoveright{+ 512t^4(5790594508160t^{10}-1225795140240t^8+31283552519t^6-265518840t^4+10109387t^2+49308) D } \\
    + 15360t^4(79556744960t^{10}-16784238240t^8+319629734t^6-3585072t^4+146945t^2+456)
  \end{multline*}}
  The local log-monodromies for the quantum local system:
\begin{align*}
& \tiny \begin{pmatrix}
0&1&0&0&0&0\\
0&0&0&0&0&0\\
0&0&0&1&0&0\\
0&0&0&0&1&0\\
0&0&0&0&0&1\\
0&0&0&0&0&0
\end{pmatrix} && \text{at $\textstyle t=0$} \\& \tiny \begin{pmatrix}
0&0&0&0&0&0\\
0&0&0&0&0&0\\
0&0&0&0&0&0\\
0&0&0&0&0&0\\
0&0&0&0&0&1\\
0&0&0&0&0&0
\end{pmatrix} && \text{at the roots of $\textstyle 176t^4-96t^3+8t^2-4t-1=0$} \\& \tiny \begin{pmatrix}
0&0&0&0&0&0\\
0&0&0&0&0&0\\
0&0&0&0&0&0\\
0&0&0&0&0&0\\
0&0&0&0&0&1\\
0&0&0&0&0&0
\end{pmatrix} && \text{at the roots of $\textstyle 176t^4+96t^3+8t^2+4t-1=0$} \\\end{align*}
The operator $L_X$ is extremal.

\subsection{$\MW{4}{12}$} \label{operator:MW^4_12}  \hfill [\hyperref[sec:MW^4_12]{description p.~\pageref*{sec:MW^4_12}}, \hyperref[table:index_2]{regularized quantum period p.~\pageref*{table:index_2}}]

The quantum differential operator is:
{\small
  \begin{multline*}
    (16t^4+44t^2-1)(432t^4+36t^2+1)(314924112t^8-117964512t^6+14238144t^4+164850t^2+221) D^{6} \\
    \shoveleft{+ 4\left(10883777310720t^{16}+14017882540032t^{14}-7127443839744t^{12}\right.}\\
    \shoveright{\left.+913572411264t^{10}+49215528432t^8+208091040t^6+29910072t^4+249043t^2+221\right) D^{5}}\\
    \shoveleft{+4\left(83805085292544t^{16}+42638961696768t^{14}-38573795723328t^{12}+6607519967520t^{10}\right.}\\
    \shoveright{\left.+197257285008t^8+2651761872t^6-52981980t^4-386208t^2-221\right) D^{4}}\\
    \shoveleft{+ 16t^2(79451574368256t^{14}+656487804672t^{12}-24209536131072t^{10}}\\
    \shoveright{+5352527374416t^8+119397810240t^6+1063433832t^4+1841528t^2+1513) D^{3}}\\
    \shoveleft{+ 32t^2(78431220245376t^{14}-24146078481216t^{12}-14886119891520t^{10}}\\
    \shoveright{+4353594087744t^8+89900843112t^6+822085824t^4+1734954t^2+221) D^{2}}\\
    \shoveleft{+ 2304t^4(1061924105664t^{12}-538846899984t^{10}-114223084416t^8}\\
    \shoveright{+48172072716t^6+1084223724t^4+9219514t^2+22737) D}\\
    + 27648t^4(33067031760t^{12}-20911507560t^{10}-1683208512t^8+1243973394t^6+32605206t^4+248744t^2+663)
  \end{multline*}}
The local log-monodromies for the quantum local system:
\begin{align*}
& \tiny \begin{pmatrix}
0&1&0&0&0&0\\
0&0&0&0&0&0\\
0&0&0&1&0&0\\
0&0&0&0&1&0\\
0&0&0&0&0&1\\
0&0&0&0&0&0
\end{pmatrix} && \text{at $\textstyle t=0$} \\& \tiny \begin{pmatrix}
0&0&0&0&0&0\\
0&0&0&0&0&0\\
0&0&0&0&0&0\\
0&0&0&0&0&0\\
0&0&0&0&0&1\\
0&0&0&0&0&0
\end{pmatrix} && \text{at the roots of $\textstyle 432t^4+36t^2+1=0$} \\& \tiny \begin{pmatrix}
0&0&0&0&0&0\\
0&0&0&0&0&0\\
0&0&0&0&0&0\\
0&0&0&0&0&0\\
0&0&0&0&0&1\\
0&0&0&0&0&0
\end{pmatrix} && \text{at the roots of $\textstyle 16t^4+44t^2-1=0$} \\& \tiny \begin{pmatrix}
0&0&0&0&0&0\\
0&0&0&0&0&0\\
0&0&0&0&0&0\\
0&0&0&0&0&0\\
0&0&0&0&0&1\\
0&0&0&0&0&0
\end{pmatrix} && \text{at $t=\infty$} \\\end{align*}
The ramification defect of $L_X$ is $1$.

\pagebreak

\subsection{$\MW{4}{13}$} \label{operator:MW^4_13} \hfill [\hyperref[sec:MW^4_13]{description p.~\pageref*{sec:MW^4_13}}, \hyperref[table:index_2]{regularized quantum period p.~\pageref*{table:index_2}}]

The quantum differential operator is:
{\small
  \begin{multline*}
    (2t-1)(2t+1)(136t^3-20t^2+6t-1)(136t^3+20t^2+6t+1)(66683996t^6-1058780t^4-19394t^2+15) D^{6} \\
    \shoveleft{+ 4\left(27134518180352t^{14}-3636686279168t^{12}-76696829376t^{10}\right.}\\
    \shoveright{\left.+3834603264t^8-107218310t^6+2061400t^4+29091t^2-15\right) D^{5}}\\
    \shoveleft{+ 4\left(234343566103040t^{14}-21804386465792t^{12}+102596245216t^{10}\right.}\\
    \shoveright{\left.+15956069184t^8+542705432t^6-1694520t^4-36762t^2+15\right) D^{4}}\\
    \shoveleft{+ 32t^2(126422186976640t^{12}-8667906294144t^{10}+129634375864t^8}\\
    \shoveright{+2690207848t^6+17651600t^4-204240t^2-49) D^{3}}\\
    \shoveleft{+ 32t^2(285066614292448t^{12}-15534154687040t^{10}+251203935368t^8}\\
    \shoveright{-291393828t^6-1136928t^4-258660t^2-15) D^{2}}\\
    \shoveright{+ 128t^4(78551346664144t^{10}-3663416463632t^8+48257529676t^6-820898270t^4-4423658t^2-37545) D}\\
    + 3072t^4(1349017239080t^{10}-57143052340t^8+533807468t^6-19445050t^4-103234t^2-345)
  \end{multline*}}
The local log-monodromies for the quantum local system:
\begin{align*}
& \tiny \begin{pmatrix}
0&1&0&0&0&0\\
0&0&0&0&0&0\\
0&0&0&1&0&0\\
0&0&0&0&1&0\\
0&0&0&0&0&1\\
0&0&0&0&0&0
\end{pmatrix} && \text{at $\textstyle t=0$} \\& \tiny \begin{pmatrix}
0&0&0&0&0&0\\
0&0&0&0&0&0\\
0&0&0&0&0&0\\
0&0&0&0&0&0\\
0&0&0&0&0&1\\
0&0&0&0&0&0
\end{pmatrix} && \text{at $\textstyle t=\frac{1}{2}$} \\& \tiny \begin{pmatrix}
0&0&0&0&0&0\\
0&0&0&0&0&0\\
0&0&0&0&0&0\\
0&0&0&0&0&0\\
0&0&0&0&0&1\\
0&0&0&0&0&0
\end{pmatrix} && \text{at $\textstyle t=-\frac{1}{2}$} \\& \tiny \begin{pmatrix}
0&0&0&0&0&0\\
0&0&0&0&0&0\\
0&0&0&0&0&0\\
0&0&0&0&0&0\\
0&0&0&0&0&1\\
0&0&0&0&0&0
\end{pmatrix} && \text{at the roots of $\textstyle 136t^3-20t^2+6t-1=0$} \\& \tiny \begin{pmatrix}
0&0&0&0&0&0\\
0&0&0&0&0&0\\
0&0&0&0&0&0\\
0&0&0&0&0&0\\
0&0&0&0&0&1\\
0&0&0&0&0&0
\end{pmatrix} && \text{at the roots of $\textstyle 136t^3+20t^2+6t+1=0$} \\\end{align*}
The operator $L_X$ is extremal.

\subsection{$\MW{4}{14}$} \label{operator:MW^4_14}  \hfill [\hyperref[sec:MW^4_14]{description p.~\pageref*{sec:MW^4_14}}, \hyperref[table:index_2]{regularized quantum period p.~\pageref*{table:index_2}}]

The quantum differential operator is:
{\small
  \begin{multline*}
    (2t-1)(2t+1)(6t-1)(6t+1)(20t^2-4t+1)(20t^2+4t+1)(127920t^6-45016t^4-293t^2-1) D^{6} \\
    \shoveleft{+ 2\left(81050112000t^{14}-42347458560t^{12}+4138403072t^{10}\right.}\\
    \shoveright{\left.+109225472t^8+95280t^6+187248t^4+911t^2+2\right) D^{5}}\\
    \shoveleft{+ 4\left(349989120000t^{14}-177028331520t^{12}+12857605632t^{10}\right.}\\
    \shoveright{\left.+227619584t^8+1578760t^6-171496t^4-769t^2-1\right) D^{4}}\\
    \shoveleft{+ 8t^2(755239680000t^{12}-383806594560t^{10}+19129972096t^8}\\
    \shoveright{+278237408t^6+1853984t^4+10304t^2+13) D^{3}}\\
    \shoveright{+ 32t^2(425743344000t^{12}-221217604800t^{10}+7159648792t^8+109530520t^6+1224240t^4+4582t^2+1) D^{2}}\\
    \shoveright{+ 768t^4(19552572000t^{10}-10422296720t^8+211812198t^6+4163952t^4+55827t^2+173) D}\\
    + 23040t^4(268632000t^{10}-146189120t^8+1864348t^6+52366t^4+729t^2+2)
  \end{multline*}}
The local log-monodromies for the quantum local system:
\begin{align*}
& \tiny \begin{pmatrix}
0&1&0&0&0&0\\
0&0&0&0&0&0\\
0&0&0&1&0&0\\
0&0&0&0&1&0\\
0&0&0&0&0&1\\
0&0&0&0&0&0
\end{pmatrix} && \text{at $\textstyle t=0$} \\& \tiny \begin{pmatrix}
0&0&0&0&0&0\\
0&0&0&0&0&0\\
0&0&0&0&0&0\\
0&0&0&0&0&0\\
0&0&0&0&0&1\\
0&0&0&0&0&0
\end{pmatrix} && \text{at $\textstyle t=\frac{1}{2}$} \\& \tiny \begin{pmatrix}
0&0&0&0&0&0\\
0&0&0&0&0&0\\
0&0&0&0&0&0\\
0&0&0&0&0&0\\
0&0&0&0&0&1\\
0&0&0&0&0&0
\end{pmatrix} && \text{at $\textstyle t=\frac{1}{6}$} \\& \tiny \begin{pmatrix}
0&0&0&0&0&0\\
0&0&0&0&0&0\\
0&0&0&0&0&0\\
0&0&0&0&0&0\\
0&0&0&0&0&1\\
0&0&0&0&0&0
\end{pmatrix} && \text{at $\textstyle t=-\frac{1}{6}$} \\& \tiny \begin{pmatrix}
0&0&0&0&0&0\\
0&0&0&0&0&0\\
0&0&0&0&0&0\\
0&0&0&0&0&0\\
0&0&0&0&0&1\\
0&0&0&0&0&0
\end{pmatrix} && \text{at $\textstyle t=-\frac{1}{2}$} \\& \tiny \begin{pmatrix}
0&0&0&0&0&0\\
0&0&0&0&0&0\\
0&0&0&0&0&0\\
0&0&0&0&0&0\\
0&0&0&0&0&1\\
0&0&0&0&0&0
\end{pmatrix} && \text{at the roots of $\textstyle 20t^2-4t+1=0$} \\& \tiny \begin{pmatrix}
0&0&0&0&0&0\\
0&0&0&0&0&0\\
0&0&0&0&0&0\\
0&0&0&0&0&0\\
0&0&0&0&0&1\\
0&0&0&0&0&0
\end{pmatrix} && \text{at the roots of $\textstyle 20t^2+4t+1=0$} \\\end{align*}
The operator $L_X$ is extremal.

\subsection{$\MW{4}{15}$} \label{operator:MW^4_15}  \hfill [\hyperref[sec:MW^4_15]{description p.~\pageref*{sec:MW^4_15}}, \hyperref[table:index_2]{regularized quantum period p.~\pageref*{table:index_2}}]

The quantum differential operator is:
{\small
  \begin{multline*}
    (416t^4-144t^3+8t^2+4t-1)(416t^4+144t^3+8t^2-4t-1)(119179008t^6+10942640t^4+192779t^2+980) D^{6} \\
    \shoveleft{+ 2\left(226871066492928t^{14}+10977991065600t^{12}-661584438272t^{10}\right.}\\
    \shoveright{\left.-1433003264t^8-363743360t^6-47681568t^4-641057t^2-1960\right) D^{5}}\\
    \shoveleft{+ 4\left(979670514401280t^{14}+74598187089920t^{12}-1635236114944t^{10}\right.}\\
    \shoveright{\left.-46841533568t^8+1268915136t^6+55009348t^4+503977t^2+980\right) D^{4}}\\
    \shoveleft{+ 8t^2(2114025846865920t^{12}+203938713538560t^{10}-1101057117184t^8}\\
    \shoveright{-117410965440t^6-475968000t^4+3889200t^2-12299) D^{3}}\\
    \shoveleft{+ 128t^2(297929404790784t^{12}+33127871570560t^{10}+111493204952t^8}\\
    \shoveright{-15417741170t^6-129552291t^4+148450t^2-245) D^{2}}\\
    \shoveright{+ 256t^4(164191489173504t^{10}+19958501546880t^8+167108388648t^6-7614297740t^4-82910507t^2-24740) D}\\
    + 1536t^4(11279101317120t^{10}+1447829510400t^8+16271407536t^6-474232070t^4-5907869t^2-4760)
  \end{multline*}}
The local log-monodromies for the quantum local system:
\begin{align*}
& \tiny \begin{pmatrix}
0&1&0&0&0&0\\
0&0&0&0&0&0\\
0&0&0&1&0&0\\
0&0&0&0&1&0\\
0&0&0&0&0&1\\
0&0&0&0&0&0
\end{pmatrix} && \text{at $\textstyle t=0$} \\& \tiny \begin{pmatrix}
0&0&0&0&0&0\\
0&0&0&0&0&0\\
0&0&0&0&0&0\\
0&0&0&0&0&0\\
0&0&0&0&0&1\\
0&0&0&0&0&0
\end{pmatrix} && \text{at the roots of $\textstyle 416t^4-144t^3+8t^2+4t-1=0$} \\& \tiny \begin{pmatrix}
0&0&0&0&0&0\\
0&0&0&0&0&0\\
0&0&0&0&0&0\\
0&0&0&0&0&0\\
0&0&0&0&0&1\\
0&0&0&0&0&0
\end{pmatrix} && \text{at the roots of $\textstyle 416t^4+144t^3+8t^2-4t-1=0$} \\\end{align*}
The operator $L_X$ is extremal.

\subsection{$\MW{4}{16}$} \label{operator:MW^4_16}  \hfill [\hyperref[sec:MW^4_16]{description p.~\pageref*{sec:MW^4_16}}, \hyperref[table:index_2]{regularized quantum period p.~\pageref*{table:index_2}}]

The quantum differential operator is:
{\small
  \begin{multline*}
    (8t^2-4t+1)(8t^2+4t+1)(28t^2-4t-1)(28t^2+4t-1)(1245440t^6+159472t^4-353t^2+2) D^{6} \\
    \shoveleft{+ 2\left(687403171840t^{14}+55847092224t^{12}-2940634112t^{10} \right.}\\
    \shoveright{\left.+646645248t^8-7816080t^6-602296t^4+771t^2-4\right) D^{5}}\\
    \shoveleft{+ 4\left(2968331878400t^{14}+342433251328t^{12}-16569335296t^{10}\right.}\\
    \shoveright{\left.+962246144t^8+20691276t^6+609636t^4-327t^2+2\right) D^{4}}\\
    \shoveright{+ 8t^2(6405347737600t^{12}+905030160384t^{10}-30840349696t^8+529133120t^6+4111656t^4-69016t^2-79) D^{3}}\\
    \shoveright{+ 64t^2(1805409751040t^{12}+290148947200t^{10}-6893123664t^8+13678020t^6-1333787t^4-13834t^2-3) D^{2}}\\
    \shoveright{+ 256t^4(497488517120t^{10}+86925519744t^8-1508603672t^6-6132868t^4-554137t^2-2494) D}\\
    + 1536t^4(34174873600t^{10}+6291237120t^8-85724560t^6-530810t^4-38831t^2-116)
  \end{multline*}}
The local log-monodromies for the quantum local system:
\begin{align*}
& \tiny \begin{pmatrix}
0&1&0&0&0&0\\
0&0&0&0&0&0\\
0&0&0&1&0&0\\
0&0&0&0&1&0\\
0&0&0&0&0&1\\
0&0&0&0&0&0
\end{pmatrix} && \text{at $\textstyle t=0$} \\& \tiny \begin{pmatrix}
0&0&0&0&0&0\\
0&0&0&0&0&0\\
0&0&0&0&0&0\\
0&0&0&0&0&0\\
0&0&0&0&0&1\\
0&0&0&0&0&0
\end{pmatrix} && \text{at the roots of $\textstyle 8t^2-4t+1=0$} \\& \tiny \begin{pmatrix}
0&0&0&0&0&0\\
0&0&0&0&0&0\\
0&0&0&0&0&0\\
0&0&0&0&0&0\\
0&0&0&0&0&1\\
0&0&0&0&0&0
\end{pmatrix} && \text{at the roots of $\textstyle 28t^2-4t-1=0$} \\& \tiny \begin{pmatrix}
0&0&0&0&0&0\\
0&0&0&0&0&0\\
0&0&0&0&0&0\\
0&0&0&0&0&0\\
0&0&0&0&0&1\\
0&0&0&0&0&0
\end{pmatrix} && \text{at the roots of $\textstyle 28t^2+4t-1=0$} \\& \tiny \begin{pmatrix}
0&0&0&0&0&0\\
0&0&0&0&0&0\\
0&0&0&0&0&0\\
0&0&0&0&0&0\\
0&0&0&0&0&1\\
0&0&0&0&0&0
\end{pmatrix} && \text{at the roots of $\textstyle 8t^2+4t+1=0$} \\\end{align*}
The operator $L_X$ is extremal.

\pagebreak

\subsection{$\MW{4}{17}$} \label{operator:MW^4_17}  \hfill [\hyperref[sec:MW^4_17]{description p.~\pageref*{sec:MW^4_17}}, \hyperref[table:index_2]{regularized quantum period p.~\pageref*{table:index_2}}]

The quantum differential operator is:
{\small
  \begin{multline*} 
    (2752t^6-1152t^5+224t^4+96t^3-52t^2+12t-1)(2752t^6+1152t^5+224t^4-96t^3-52t^2-12t-1)\\
    \shoveleft{ (4516691026120601600t^{18}+10894175535784019520t^{16}+161057014788668272t^{14}+223186423846901825t^{12}}\\
    \shoveright{+1489656860655194t^{10}-37076036387883t^8-944190030122t^6+1145046509t^4+34077463t^2+3136) D^{8}} \\
    \shoveleft{+ 2\left(718350728614858094700134400t^{30}+1808773177510341584824565760t^{28}\right.}\\
+11029946474586078383243264t^{26}+38551285621538701581713408t^{24}-573876948723528931270656t^{22}\\
-55410907815277159905280t^{20}-10868619181236655325696t^{18}-205717988100675369984t^{16}\\
+54690133059712073568t^{14}-2288987306843185052t^{12}-22309755134355486t^{10}\\
\shoveright{\left.+253350406694334t^8+8580745232122t^6-7798176720t^4-238542241t^2-18816\right) D^{7}}\\
\shoveleft{+ 4\left(6234258109050375607576166400t^{30}+16563148489700147386089144320t^{28}\right.}\\
+153884025675978814793531392t^{26}+396376397592542781690715136t^{24}-244086311073605059692544t^{22}\\
-293272503356977770941696t^{20}-34041797107502106577312t^{18}+652775043341067168068t^{16}\\
-44917328848849757800t^{14}+11574061561240458564t^{12}+61805608644992588t^{10}\\
\shoveright{\left.-1336778298658238t^8-25795200795638t^6+35399715688t^4+581693399t^2+40768\right) D^{6}} \\
\shoveleft{+8\left(29067549125736936474830438400t^{30}+81851870301201769764088381440t^{28}\right.}\\
+939457832323447863369383936t^{26}+2148139133199504182578666496t^{24}+5371023638579317537625088t^{22}\\
-957869720387013465562880t^{20}-5020278555805201268384t^{18}-1791882536502223335404t^{16}\\
-85650713028232429760t^{14}-20843283163258659946t^{12}-101280721448407430t^{10}\\
\shoveright{\left.+1741410204251200t^8+30185303650830t^6-49184868070t^4-583650151t^2-37632\right) D^{5}}\\
\shoveleft{+16\left(79230237058744125391346073600t^{30}+236991057245518887863871979520t^{28}\right.}\\
+3064756009927980485218619392t^{26}+6725558154012574042351982336t^{24}+25714515380315969257756928t^{22}\\
-2214587183336657609276512t^{20}+14905942833336792786992t^{18}+2425296396558205778092t^{16}\\
+70030906648265241400t^{14}+13039662696372133523t^{12}+68594134702584874t^{10}\\
\shoveright{\left.-642583292426727t^8-12336212241784t^6+23535298811t^4+206616410t^2+12544\right) D^{4}}\\
\shoveleft{+64t^2(64244286367595768799677644800t^{28}+203835202251539030078676664320t^{26}}\\
+2826396757245234563525490688t^{24}+6176058394659175094246477056t^{22}+30060979005675148672022400t^{20}\\
-1750848106552043666387424t^{18}-31680726103451488064224t^{16}+328919199593772520366t^{14}\\
+33052797280494574660t^{12}-311682535251628245t^{10}-3943292702734802t^8\\
\shoveright{-5555362326911t^6+105428531260t^4-716272453t^2-28448) D^{3}}\\
\shoveleft{+512t^2(15009134071247813042455347200t^{28}+50222389883894861028705538560t^{26}}\\
+725335071957348016538310016t^{24}+1603230224901344022068592328t^{22}+9154893897382658634810352t^{20}\\
-428192374996045407448412t^{18}-15519523249001954206792t^{16}+52781806444610795042t^{14}\\
-1276462080564820298t^{12}-107386844952242335t^{10}-792211501649686t^8\\
\shoveright{-5580848553730t^6+16358540366t^4-35218463t^2-1568) D^{2}}\\
\shoveleft{+6144t^6(1216860211646830162557747200t^{24}+4254138344852923650551665920t^{22}}\\
+62879395621367262268113088t^{20}+141203555473546515828319348t^{18}+892698338486015205162136t^{16}\\
-36871341336439978899176t^{14}-1637218553379789670928t^{12}+2207989224928555682t^{10}\\
\shoveright{-562698415042789318t^8-9517978064339709t^6-42261757561740t^4-328350172920t^2+368309382) D}\\
\shoveleft{+15482880t^6(183729957560533831884800t^{24}+663265247268250820234880t^{22}}\\
+9929341713271041302112t^{20}+22607862359961003907882t^{18}+151963033359678362364t^{16}\\
-5846359991941694249t^{14}-276152453516510822t^{12}+64100253358393t^{10}\\
-116439515097332t^8-1405235953421t^6-3886412020t^4-21911230t^2-41272)
  \end{multline*}}
The local log-monodromies for the quantum local system:
\begin{align*}
& \tiny \begin{pmatrix}
0&1&0&0&0&0&0&0\\
0&0&0&0&0&0&0&0\\
0&0&0&1&0&0&0&0\\
0&0&0&0&0&0&0&0\\
0&0&0&0&0&1&0&0\\
0&0&0&0&0&0&1&0\\
0&0&0&0&0&0&0&1\\
0&0&0&0&0&0&0&0
\end{pmatrix} && \text{at $\textstyle t=0$} \\& \tiny \begin{pmatrix}
0&0&0&0&0&0&0&0\\
0&0&0&0&0&0&0&0\\
0&0&0&0&0&0&0&0\\
0&0&0&0&0&0&0&0\\
0&0&0&0&0&0&0&0\\
0&0&0&0&0&0&0&0\\
0&0&0&0&0&0&0&1\\
0&0&0&0&0&0&0&0
\end{pmatrix} && \text{at the roots of $\textstyle 2752t^6-1152t^5+224t^4+96t^3-52t^2+12t-1=0$} \\& \tiny \begin{pmatrix}
0&0&0&0&0&0&0&0\\
0&0&0&0&0&0&0&0\\
0&0&0&0&0&0&0&0\\
0&0&0&0&0&0&0&0\\
0&0&0&0&0&0&0&0\\
0&0&0&0&0&0&0&0\\
0&0&0&0&0&0&0&1\\
0&0&0&0&0&0&0&0
\end{pmatrix} && \text{at the roots of $\textstyle 2752t^6+1152t^5+224t^4-96t^3-52t^2-12t-1=0$} \\\end{align*}
The ramification defect of $L_X$ is $1$.

\subsection{$\MW{4}{18}$} \label{operator:MW^4_18}  \hfill [\hyperref[sec:MW^4_18]{description p.~\pageref*{sec:MW^4_18}}, \hyperref[table:index_2]{regularized quantum period p.~\pageref*{table:index_2}}]

The quantum differential operator is:
{\small
  \begin{multline*}
    (4t-1)(4t+1)(8t-1)(8t+1) D^{4} + 
    64t^2(128t^2-5) D^{3} + 
    16t^2(1472t^2-33) D^{2} + \\
    32t^2(896t^2-13) D + 
    128t^2(96t^2-1)
  \end{multline*}}
The local log-monodromies for the quantum local system:
\begin{align*}
& \tiny \begin{pmatrix}
0&1&0&0\\
0&0&1&0\\
0&0&0&1\\
0&0&0&0
\end{pmatrix} && \text{at $\textstyle t=0$} \\& \tiny \begin{pmatrix}
0&0&0&0\\
0&0&0&0\\
0&0&0&1\\
0&0&0&0
\end{pmatrix} && \text{at $\textstyle t=\frac{1}{4}$} \\& \tiny \begin{pmatrix}
0&0&0&0\\
0&0&0&0\\
0&0&0&1\\
0&0&0&0
\end{pmatrix} && \text{at $\textstyle t=\frac{1}{8}$} \\& \tiny \begin{pmatrix}
0&0&0&0\\
0&0&0&0\\
0&0&0&1\\
0&0&0&0
\end{pmatrix} && \text{at $\textstyle t=-\frac{1}{8}$} \\& \tiny \begin{pmatrix}
0&0&0&0\\
0&0&0&0\\
0&0&0&1\\
0&0&0&0
\end{pmatrix} && \text{at $\textstyle t=-\frac{1}{4}$} \\& \tiny \begin{pmatrix}
0&0&0&0\\
0&0&0&0\\
0&0&0&1\\
0&0&0&0
\end{pmatrix} && \text{at $t=\infty$} \\\end{align*}
The operator $L_X$ is extremal.

\bibliographystyle{amsplain} 
\bibliography{bibliography}

\end{document}